\documentclass[english,11pt]{amsart}
\usepackage[utf8]{inputenc}
\usepackage{amssymb,amsmath,amsthm,mathtools,amsfonts}

\usepackage{graphicx}
\usepackage[usenames,dvipsnames]{xcolor}
\usepackage{cite}
\usepackage{url}
\usepackage{hyperref}
\usepackage{xcolor}
\usepackage[all]{xy}
\usepackage[version=4]{mhchem}
\usepackage{cancel}
\hypersetup{
    bookmarks=true,         	
    unicode=false,          		
    pdftoolbar=true,        		
    pdfmenubar=true,        	
    pdffitwindow=false,     		
    pdfstartview={FitH},    		
   pdftitle={long-time dynamics of the sine-Gordon equation },    					
   pdfauthor={Cheng He, Jiaqi Liu, Changzheng Qu},     	
    linktocpage=true,			
    pdfnewwindow=true,      	
    colorlinks=true,       			
    linkcolor=red,          		
    citecolor=PineGreen,    	
    filecolor=magenta,      		
    urlcolor=cyan           		
}

\usepackage{tikz}
\usetikzlibrary{decorations.pathmorphing}
\usetikzlibrary{arrows,decorations.pathmorphing}

\usepackage{float}
\usepackage[section]{placeins}		

\theoremstyle{plain}
\newtheorem{theorem}{Theorem}[section]
\newtheorem{corollary}[theorem]{Corollary}
\newtheorem{lemma}[theorem]{Lemma}
\newtheorem{proposition}[theorem]{Proposition}

\theoremstyle{definition}
\newtheorem{definition}[theorem]{Definition}
\newtheorem{problem}[theorem]{Problem}

\theoremstyle{remark}
\newtheorem{remark}[theorem]{Remark}

\numberwithin{figure}{section}
\numberwithin{equation}{section}

\DeclareMathOperator{\ad}{ad}

\DeclareMathOperator{\Real}{Re}

\DeclareMathOperator{\Res}{Res}

\allowdisplaybreaks

\newenvironment{doublecases}
{
	\left\{
			\begin{array}{lllll}
}
{			
			\end{array}
			\right.
}
			
\title[Massive Thirring Model]{Massive Thirring Model: Inverse Scattering and Soliton Resolution}

\author{Cheng He }
\author{Jiaqi Liu}
\author{Changzheng Qu}
\address[He]{School of Mathematics, Ningbo University. Ningbo, Zhejiang China}
\email{1811071003@nbu.edu.cn}
\address[Liu]{School of mathematics, University of Chinese Academy of Sciences. Beijing China }
\email{jqliu@ucas.ac.cn}
\address[Qu]{School of Mathematics, Ningbo University. Ningbo, Zhejiang China}
\email{quchangzheng@nbu.edu.cn}
\date{\today}

\begin{document}

\maketitle
\begin{abstract}
 In this paper the long-time dynamics of the massive Thirring model is investigated. Firstly the nonlinear steepest descent method for Riemann-Hilbert problem is explored to obtain the soliton resolution of the solutions to the massive Thirring model whose initial data belong to some weighted-Sobolev spaces. Secondly, the asymptotic stability  of multi-solitons follow as a corollary. The main difficulty in studying the massive Thirring model through inverse scattering is that the corresponding Lax pair has singularities at the origin and infinity. We overcome this difficulty by making use of two transforms that separate the singularities.
\end{abstract}

\tableofcontents

%
%

\newcommand{\eps}{\varepsilon}
\newcommand{\lam}{\zeta}

\newcommand{\bfN}{\mathbf{N}}
\newcommand{\calbR}{\mathcal{ \breve{R}}}
\newcommand{\rhobar}{\overline{\rho}}
\newcommand{\zetabar}{\overline{\zeta}}

\newcommand{\rarr}{\rightarrow}
\newcommand{\darr}{\downarrow}
\newcommand{\ttheta}{\widetilde{\theta}}
\newcommand{\dee}{\partial}
\newcommand{\dbar}{\overline{\partial}}
\newcommand{\tM}{\widetilde{M}}
\newcommand{\dint}{\displaystyle{\int}}

\newcommand{\dotarg}{\, \cdot \, }

%
%

\newcommand{\RHP}{\mathrm{LC}}			
\newcommand{\PC}{\mathrm{PC}}
\newcommand{\w}{w^{(2)}}
%
%

\newcommand{\zbar}{\overline{z}}

\newcommand{\bbC}{\mathbb{C}}
\newcommand{\bbR}{\mathbb{R}}

\newcommand{\calB}{\mathcal{B}}
\newcommand{\calC}{\mathcal{C}}
\newcommand{\calR}{\mathcal{R}}
\newcommand{\calS}{\mathcal{S}}
\newcommand{\calZ}{\mathcal{Z}}
\newcommand{\tgamma}{\widetilde{\gamma}}

\newcommand{\ba}{\breve{a}}
\newcommand{\bb}{\breve{b}}
\newcommand{\bchi}{\breve{\chi}}

\newcommand{\balpha}{\breve{\alpha}}
\newcommand{\brho}{\breve{\rho}}

\newcommand{\tPhi}{{\widetilde{\Phi}}}
\newcommand{\tdr}{\widetilde{r}}

\newcommand{\tp}{\text{p}}
\newcommand{\tq}{\text{q}}
\newcommand{\tr}{\text{r}}
\newcommand{\tw}{\widetilde{W}}
\newcommand{\bfe}{\mathbf{e}}
\newcommand{\bfn}{\mathbf{n}}

\newcommand{\tA}{\widetilde{A}}
\newcommand{\tB}{\widetilde{B}}
\newcommand{\tomega}{\widetilde{\omega}}
\newcommand{\tc}{\widetilde{c}}
\newcommand{\tz}{\widetilde{z}}
\newcommand{\tv}{\widetilde{V}}
\newcommand{\mhat}{\hat{m}}

\newcommand{\bphi}{\breve{\Phi}}
\newcommand{\bN}{\breve{N}}
\newcommand{\bV}{\breve{V}}
\newcommand{\bR}{\breve{R}}
\newcommand{\bdelta}{\breve{\delta}}
\newcommand{\bzeta}{\breve{\zeta}}
\newcommand{\bbeta}{\breve{\beta}}
\newcommand{\bm}{\breve{m}}
\newcommand{\br}{\breve{r}}
\newcommand{\bnu}{\breve{\nu}}
\newcommand{\bbfN}{\breve{\mathbf{N}}}
\newcommand{\rbar}{\overline{r}}

\newcommand{\One}{\mathbf{1}}
\newcommand{\tabincell}[2]{\begin{tabular}{@{}#1@{}}#2\end{tabular}}%
\renewcommand{\Re}{\operatorname{Re}}
\renewcommand{\Im}{\operatorname{Im}}
\newcommand{\diag}{\textrm{diag}}
\newcommand{\off}{\textrm{off}}
\newcommand{\T}{\textrm{T}}
\newcommand{\tm}{\widetilde{M}}
%
%

\newcommand{\bigO}[2][ ]
{
\mathcal{O}_{#1}
\left(
{#2}
\right)
}

\newcommand{\littleO}[1]{{o}\left( {#1} \right)}

\newcommand{\norm}[2]
{
\left\Vert		{#1}	\right\Vert_{#2}
}

%
%

\newcommand{\rowvec}[2]
{
\left(
	\begin{array}{cc}
		{#1}	&	{#2}	
	\end{array}
\right)
}

\newcommand{\uppermat}[1]
{
\left(
	\begin{array}{cc}
	0		&	{#1}	\\
	0		&	0
	\end{array}
\right)
}

\newcommand{\lowermat}[1]
{
\left(
	\begin{array}{cc}
	0		&	0	\\
	{#1}	&	0
	\end{array}
\right)
}

\newcommand{\offdiagmat}[2]
{
\left(
	\begin{array}{cc}
	0		&	{#1}	\\
	{#2}	&	0
	\end{array}
\right)
}

\newcommand{\diagmat}[2]
{
\left(
	\begin{array}{cc}
		{#1}	&	0	\\
		0		&	{#2}
		\end{array}
\right)
}

\newcommand{\Offdiagmat}[2]
{
\left(
	\begin{array}{cc}
		0			&		{#1} 	\\
		\\
		{#2}		&		0
		\end{array}
\right)
}

\newcommand{\twomat}[4]
{
\left(
	\begin{array}{cc}
		{#1}	&	{#2}	\\
		{#3}	&	{#4}
		\end{array}
\right)
}

\newcommand{\unitupper}[1]
{	
	\twomat{1}{#1}{0}{1}
}

\newcommand{\unitlower}[1]
{
	\twomat{1}{0}{#1}{1}
}

\newcommand{\Twomat}[4]
{
\left(
	\begin{array}{cc}
		{#1}	&	{#2}	\\[10pt]
		{#3}	&	{#4}
		\end{array}
\right)
}

%
%
%

\newcommand{\JumpMatrixFactors}[6]
{
	\begin{equation}
	\label{#2}
	{#1} =	\begin{cases}
					{#3} {#4}, 	&	\lambda \in (-\infty,\xi) \\
					\\
					{#5}{#6},	&	\lambda \in (\xi,\infty)
				\end{cases}
	\end{equation}
}


%
%
%

\newcommand{\RMatrix}[9]
{
\begin{equation}
\label{#1}
\begin{aligned}
\left. R_1 \right|_{(\xi,\infty)} 	&= {#2} &	\qquad\qquad		
\left. R_1 \right|_{\Sigma_1}		&= {#3}
\\[5pt]
\left. R_3 \right|_{(-\infty,\xi)} 	&= {#4} 	&	
\left. R_3 \right|_{\Sigma_2} 	&= {#5}
\\[5pt]
\left. R_4 \right|_{(-\infty,\xi)} 	&= {#6} &	
\left. R_4 \right|_{\Sigma_3} 	&= {#7}
\\[5pt]
\left. R_6 \right|_{(\xi,\infty)}  	&= {#8} &	
\left. R_6 \right|_{\Sigma_4} 	&= {#9}
\end{aligned}
\end{equation}
}

%
%

%
%
%
%
%
%

\newcommand{\SixMatrix}[6]
{
\begin{figure}
\centering
\caption{#1}
\vskip 15pt
\begin{tikzpicture}
[scale=0.75]
%
%
\draw[thick]	 (-4,0) -- (4,0);
\draw[thick] 	(-4,4) -- (4,-4);
\draw[thick] 	(-4,-4) -- (4,4);
%
%
\draw	[fill]		(0,0)						circle[radius=0.075];
\node[below] at (0,-0.1) 				{$\tz_0$};
%
%
\node[above] at (3.5,2.5)				{$\Omega_3$};
\node[below]  at (3.5,-2.5)			{$\Omega_7$};
\node[above] at (0,3.25)				{$\Omega_1$};
\node[below] at (0,-3.25)				{$\Omega_2$};
\node[above] at (-3.5,2.5)			{$\Omega_8$};
\node[below] at (-3.5,-2.5)			{$\Omega_4$};
%
%
\node[above] at (0,1.25)				{$\twomat{1}{0}{0}{1}$};
\node[below] at (0,-1.25)				{$\twomat{1}{0}{0}{1}$};
%
%
\node[right] at (1.20,0.70)			{$#3$};
\node[left]   at (-1.20,0.70)			{$#4$};
\node[left]   at (-1.20,-0.70)			{$#5$};
\node[right] at (1.20,-0.70)			{$#6$};
\end{tikzpicture}
\label{#2}
\end{figure}
}

\newcommand{\sixmatrix}[6]
{
\begin{figure}
\centering
\caption{#1}
\vskip 15pt
\begin{tikzpicture}
[scale=0.75]
%
%
\draw[thick]	 (-4,0) -- (4,0);
\draw[thick] 	(-4,4) -- (4,-4);
\draw[thick] 	(-4,-4) -- (4,4);
%
%
\draw	[fill]		(0,0)						circle[radius=0.075];
\node[below] at (0,-0.1) 				{$-\tz_0$};
%
%
\node[above] at (3.5,2.5)				{$\Omega_9$};
\node[below]  at (3.5,-2.5)			{$\Omega_5$};
\node[above] at (0,3.25)				{$\Omega_1$};
\node[below] at (0,-3.25)				{$\Omega_2$};
\node[above] at (-3.5,2.5)			{$\Omega_6$};
\node[below] at (-3.5,-2.5)			{$\Omega_{10}$};
%
%
\node[above] at (0,1.25)				{$\twomat{1}{0}{0}{1}$};
\node[below] at (0,-1.25)				{$\twomat{1}{0}{0}{1}$};
%
%
\node[right] at (1.20,0.70)			{$#3$};
\node[left]   at (-1.20,0.70)			{$#4$};
\node[left]   at (-1.20,-0.70)			{$#5$};
\node[right] at (1.20,-0.70)			{$#6$};
\end{tikzpicture}
\label{#2}
\end{figure}
}

\newcommand{\JumpMatrixRightCut}[6]
{
\begin{figure}
\centering
\caption{#1}
\vskip 15pt
\begin{tikzpicture}[scale=0.85]
%
%
\draw [fill] (4,4) circle [radius=0.075];						
\node at (4.0,3.65) {$\xi$};										
%
%
\draw 	[->, thick]  	(4,4) -- (5,5) ;								
\draw		[thick] 		(5,5) -- (6,6) ;
\draw		[->, thick] 	(2,6) -- (3,5) ;								
\draw		[thick]		(3,5) -- (4,4);	
\draw		[->, thick]	(2,2) -- (3,3);								
\draw		[thick]		(3,3) -- (4,4);
\draw		[->,thick]	(4,4) -- (5,3);								
\draw		[thick]  		(5,3) -- (6,2);
%
%
\draw [  thick, blue, decorate, decoration={snake,amplitude=0.5mm}] (4,4)  -- (8,4);				
\node at (1.5,4) {$0 < \arg (\zeta-\xi) < 2\pi$};
%
%
\node at (8.5,8.5)  	{$\Sigma_1$};
\node at (-0.5,8.5) 	{$\Sigma_2$};
\node at (-0.5,-0.5)	{$\Sigma_3$};
\node at (8.5,-0.5) 	{$\Sigma_4$};
%
%
\node at (7,7) {${#3}$};						
\node at (1,7) {${#4}$};						
\node at (1,1) {${#5}$};						
\node at (7,1) {${#6}$};						
\end{tikzpicture}
\label{#2}
\end{figure}
}

%
%

\section{Introduction}
\subsection{General introduction}
In this paper we apply the inverse scattering transform (IST) to study the long time asymptotics of the massive Thirring model (MTM):
\begin{equation}
\label{eq: MTM}
\left\{\begin{array}{c}
i\left(u_{t}+u_{x}\right)+v+|v|^{2} u=0 \\
i\left(v_{t}-v_{x}\right)+u+|u|^{2} v=0
\end{array}\right.
\end{equation}
subject to an initial condition \begin{equation}
(u(x,0), v(x,0)):=(u_0, v_0)\in \mathcal{I}\times \mathcal{I},
\end{equation}
where
\begin{equation}\label{space:initial}
\mathcal{I}:=\left\{f:f,f',f''\in L^{2,1}(\mathbb{R})\right\}
\end{equation}
with the space defined by
\begin{equation}
\|f\|_{L^{2,1}}:=\|(1+|\cdot|)f(\cdot)\|_{L^2}.
\end{equation}
The massive Thirring model was first introduced by W. Thirring in \cite{Thirring}.  Such a model arises as a modification of a nonlinear Schr\"odinger equation to relativistic velocities.  We first point out that \eqref{eq: MTM} is one type of the \textit{nonlinear Dirac equation}. For a general exposition, see \cite{EGT}. For the global well-posedness of the MTM through PDE methods, we refer the reader to the work of Candy\cite{Candy}. Also, in \cite{PS}, Pelinovsky and Saalmann established well-posedness by studying the bijective property of the direct and inverse scattering transform. They follow the work of Zhou  \cite{Zhou98},  where the author developed a rigorous analysis of the direct and inverse scattering transform of the AKNS system
for a class of initial conditions $u_0(x)=u(x,t=0)$ belonging to the space  $H^{i,j}(\bbR)$.
Here,  $H^{i,j}(\bbR)$  denotes  
the completion of $C_0^\infty(\bbR)$ in the norm
\begin{equation}
\label{sp: weighted}
    \norm{u}{H^{i,j}(\bbR)}
= \left( \norm{(1+|x|^j)u}{2}^2 + \norm{\frac{\partial^i u}{\partial x^i}}{2}^2 \right)^{1/2}. 
\end{equation} 
Besides well-posedness, another fundamental question for dispersive PDEs is the long-time asymptotics. In \cite{CL}, Candy and Lindblad calculated the long-time scattering of the solution to \eqref{eq: MTM} of small initial condition in \textit{Sobolev} space using PDE methods. In the current paper, we shall making use of the completely integrable structure of the MTM first introduced by \cite{KN}. 
\smallskip

Using the complete integrability of the modified KdV equation,  Deift and Zhou in their seminal work \cite{DZ93} developed the celebrated nonlinear steepest descent method for oscillatory Riemann-Hilbert problems. Since then, analysis of  long-time behavior
of integrable systems have been extensively treated by
many authors. The nonlinear steepest descent method provides a systematic way to 
reduce the original RHP 
to a canonical model RHP whose solution is calculated in terms of special functions. Dieng and McLaughlin in \cite{DM08} (see also an extended version \cite{DMM18})later developed a variant of Deift-Zhou method. In their approach 
rational approximation of the reflection coefficient is replaced by some 
non-analytic extension of the jump matrices off the real axis, which  leads to a $\bar{\partial}$-problem to 
be solved in some regions of the complex plane. 
 This method has shown its robustness in its application to other integrable models. Notably, for focussing NLS, this method was
successfully applied to address the soliton resolution in \cite{BJM}. Later, in \cite{JLPS18 } and \cite{CLL}, soliton resolution for the derivative NLS equation and sine-Gordon equation were established respectively. The massive Thirring model admits a family of soliton solutions:
\begin{equation}
\left\{\begin{array}{l}
u_\lambda(x, t)=i \delta^{-1} \sin \gamma \operatorname{sech}\left[\alpha(x+c t)-i \frac{\gamma}{2}\right] e^{-i \beta(t+c x)} \\
v_\lambda(x, t)=-i \delta \sin \gamma \operatorname{sech}\left[\alpha(x+c t)+i \frac{\gamma}{2}\right] e^{-i \beta(t+c x)}
\end{array}\right.
\end{equation}
We first mention that the solitons do not exist in MTM for small initial data and the first result of this was obtained in the last section of \cite{P}. Later in \cite{CPS} (see also \cite{PS-2}), the authors proved $L^2$ orbital stability of the solitons above, through the use of the auto-B\"acklund transformation. In the current paper, we shall employ the \textit{nonlinear steepest descent } method to to analyze the large $t$ asymptotics  of
the MTM  in  weighted Sobolev space given by \eqref{space:initial}. Consequentially, we obtain the soliton resolution and multi-soliton asymptotic stability of the MTM. 
\subsection{Main results}
\begin{theorem}[Soliton resolution]
		\label{thm:main1} 
		Let the initial data $(u_0, v_0)=(u(x,0), v(x,0))\in\mathcal{I}\times \mathcal{I}$
	 be \emph{generic} in the sense of Definition \ref{def:generic'} and let $(u(x,t), v(x,t))$
	be  the unique  solution to the massive Thirring \eqref{eq: MTM} obtained by solving the Riemann-Hilbert Problem \ref{RHP2-v} with initial data $(u_0, v_0)$. Then the
	solution $(u(x,t), v(x,t))$ can be written as the superposition of solitons and radiation. More precisely, there exist a non negative
	integer $N_{1}$, and two sets of complex parameters
	\[
	\left\{ \lambda_{j}\right\} _{j=1,\ldots N_{1}},\ \left\{ \tc_{j}\right\} _{j=1,\ldots N_{1}}
	\]
	such that the solutions $u$ and $v$ have the following representation
		\begin{align*}
	u(x,t) & =\sum_{j=1}^{N_{1}}u^{sol}\left(x,t;\lambda_{j},\tc_{j}\right)+u_r(x,t)\\
	v(x,t) & =\sum_{j=1}^{N_{1}}v^{sol}\left(x,t;\lambda_{j},\tc_{j}\right)+v_r(x,t)
	\end{align*}
{where the radiation term  $u_r$ and $v_r$ decay to $0$ and it exhibits the modified scattering inside the light-cone.}
	\end{theorem}
\begin{remark}
{We remark that $u_r(x,t)$ and $v_r(x,t)$ are computed explicitly in terms of stationary points and the reflection coefficient. In particualr the error term from $u_r(x,t)$ and $v_r(x,t)$ depend on the weighted norm of the initial data $\mathcal{I}$.  {We would like to point out that strictly inside, near and strictly outside of the line-cone,  the error terms  depend on $\frac{x}{t}$ in different manners.}
For the explicit expressions of the radiation terms, see Theorem \ref{thm:maindetail}.}
\end{remark}

\begin{remark}
From the inverse scattering point of view, the analysis of the MTM \eqref{eq: MTM} combines techniques from that of the derivative NLS equation and the sine-Gordon equation in the following sense: 
\begin{itemize}
\item the spectral dependence of the linear problem is rational which leads to singularities at both the origin and infinity. We mention that, for the derivative NLS, the singularity is at infinity while the sine-Gordon equation has singularity at the origin. These difficulties are overcome by introducing a pair of transforms that divide the analysis into two different linear problems (see \cite{PS}).
\item the phase function of the reflection coefficient has singularity at the origin which makes the direct scattering map between Sobolev spaces more involved. Here we follow the arguments developed in \cite{CLL} to show 
the following map:
\begin{equation}
\mathcal{I}\times \mathcal{I} \ni (u_0, v_0)\mapsto r\in H^{1,1}_0(\mathbb{R})\subset H^1(\mathbb{R}).
\end{equation}
which is crucial in determining the long time asymptotics. And this fills in the gap previously left out by \cite{Saalmann-2} (see \cite[p.28]{Saalmann-2}).
\end{itemize}

\end{remark}
\begin{remark}
Compared with \cite{Saalmann-1}-\cite{Saalmann-2} in which the author also discussed the long time asymptotics of the MTM using \textit{nonlinear steepest descent}, the current paper is different in the following aspects:
\begin{itemize}
\item we reconstruct both $u(x,t)$ and $v(x,t)$ from a single Riemann-Hilbert problem.
\item We follow the argument developed in \cite{JLP} to show the existence of solutions to the Riemann-Hilbert problem without relying on Darboux transform.
\item Instead of showing the mapping properties of the inverse scattering $ \widetilde{r} \mapsto (u(x, t), v(x,t))$ as in \cite{PS}
 we use techniques from pde theory to show global existence of solutions in \eqref{space:initial}.
\end{itemize}

\end{remark}
\subsection{List of notations} We give a list of notations for the reader's convenience:\\\\
\begin{itemize}
	\item[1.] Throughout this paper, we set
	$$\ttheta(\lambda; x,t)=\dfrac{1}{2} \left(  \left(  \lambda-\dfrac{1}{\lambda}  \right )x+\left( \lambda+\dfrac{1}{\lambda}  \right)t  \right)$$
	
	\item[2.] Let $\sigma_3$ be the third Pauli matrix:
\begin{equation}\label{eq:sigma3}
\sigma_3=\twomat{1}{0}{0}{-1}
\end{equation}
then one has the matrix operation 
$$e^{\ad\sigma_3}\twomat{a}{ b}{c}{d} =\twomat{a}{e^{2} b}{e^{-2}c}{d}.$$
\item[3.] Let $A$ and $B$ be $2\times 2$ matrices, then we define the commutator by
$$\left[A, B \right]=AB-BA.$$
\item[4.] $C^\pm$ is the Cauchy projection:
\begin{equation}
(C^\pm f)(z)= \lim_{z\to \Sigma_\pm}\dfrac{1}{2\pi i} \int_{\Sigma} \dfrac{f(s)}{s-z}ds.
\end{equation}
where $+(-)$ denotes taking limit from the positive (negative) side of the oriented contour $\Sigma$.

Similarly, suppose $M(z)$ is a matrix-valued function in $\bbC$, then $M_\pm$ denotes its continuous boundary value from either side of the oriented contour.\\
\smallskip
\item[5.] We define the Fourier transform as 
	\begin{equation}
	\hat{h}\left(\xi\right)=\mathcal{F}\left[h\right]\left(\xi\right)=\frac{1}{2\pi}\int_\bbR e^{-ix\xi}h\left(x\right)\,dx.\label{eq:FT}
	\end{equation}
	Notice that by the duality between the physical space and the Fourier
	space, if $h\in L^{2,s}\left(\mathbb{R}\right)$ with $s>\frac{1}{2}$
	then
	\begin{equation}
	h\in L^{1}\left(\mathbb{R}\right),\,\hat{h}\in H^{s}\left(\mathbb{R}\right).\label{eq:weightL22}
	\end{equation}
	Then by the trivial Sobolev embedding, $\hat{h}\left(\xi\right)\in L^{\infty}.$

\item[6.] As usual, $"A:=B"$
or $"B=:A"$
is the definition of $A$ by means of the expression $B$. We use
the notation $\langle x\rangle=\left(1+|x|^{2}\right)^{\frac{1}{2}}$.
For positive quantities $a$ and $b$, we write
$a\lesssim b$ for $a\leq Cb$ where $C$ is some prescribed constant.
Also $a\simeq b$ for $a\lesssim b$ and $b\lesssim a$. 
 Throughout, we use $u_{t}:=\frac{\partial u}{\partial_{t}}$, for the derivative in the time variable and 
$u_{x}:=\frac{\partial}{\partial x}u$ for the derivative in the space variable. These two notations are used interchangebly.
\end{itemize}

\begin{center}
\begin{tabular}{|c|c|}
\hline
Notation     &Summary   \\
\hline
$x$, $y$ & The space variable\\
\hline
$\langle x \rangle$ & \tabincell{c}{$\sqrt{1+|x|^2}$}\\
\hline
$I_\infty$ & $\mathbb{R}\setminus[-1,1]$\\
\hline
$\zeta$ & The spectral parameter of the original Lax pair\\
\hline
$\lambda:=\zeta^2$ & The spectral parameter after gauge transformations\\
\hline
$I$ &  $2\times 2$ identity matrix\\
\hline
$\textbf{1}$ & The identity operator\\
\hline
$\psi^{(\pm)}_{1,2}$     &Jost solutions of the original Lax pair\\
\hline
$m^{(\pm)}_{1,2}$     &Normalized solution of the original Lax pair\\
\hline
$T$                      &Gauge transformation for small $\zeta$\\
\hline
$n^{(\pm)}_{1,2}$        &\tabincell{c}{ Normalized solution of the \\Lax pair under the transformation $T$}\\
\hline
$n^{\pm 0}_{1,2}$        &The limit of $n^{(\pm)}_{1,2}$ with $\lambda$ goes to zero\\
\hline

$\widetilde{T}$  &\tabincell{c}{ Gauge transformation for \\ large $\zeta$}\\
\hline
$\widetilde{n}^{(\pm)}_{1,2}$        & \tabincell{c}{Normalized solution of the Lax pair\\ under the transformation $\widetilde{T}$}\\
\hline
$\widetilde{n}^{\pm \infty}_{1,2}$        &The limit of $\widetilde{m}^{(\pm)}_{1,2}$ with $|\lambda|$ goes to $\infty$\\
\hline
    \tabincell{c}{$S(\zeta)=\twomat{a(\zeta)}{b(\zeta)}{\breve{b}(\zeta)}{\breve{a}(\zeta)}$} &\tabincell{c}{ The scattering matrix and data of\\ the original Lax pair}\\
\hline
    \tabincell{c}{$S_g(\lambda)=\twomat{\alpha(\lambda)}{\beta(\lambda)}{\breve{\beta}(\lambda)}{\breve{\alpha}(\lambda)}$} & \tabincell{c}{The scattering matrix and data of\\ Lax pair under the transformation $T$}\\
\hline

\tabincell{c}{$\widetilde{S}_g(\lambda)=\twomat{\widetilde{\alpha}(\lambda)}{\widetilde{\beta}(\lambda)}{\breve{\widetilde{\beta}}(\lambda)}{\breve{\widetilde{\alpha}}(\lambda)}$} & \tabincell{c}{The scattering matrix and data of Lax pair\\ under the transformation $\widetilde{T}$}\\
\hline
$r(\gamma)$, $\breve{r}(\gamma)$ &\tabincell{c}{ reflection coefficients corresponding\\ to the Lax pair under the transformation $T$}\\
\hline
$\widetilde{r}(\lambda)$, $\breve{\widetilde{r}}(\lambda)$ & \tabincell{c}{ reflection coefficients corresponding to \\the Lax pair under the transformation $\widetilde{T}$}\\
\hline
$\mathbb C^{\pm \pm}$ & \tabincell{c}{Four quadrants of the complex plane.\\ The first superscript is the symbol for the real part\\ and the second for imaginary part}\\
\hline
\end{tabular}
\end{center}

\section{Direct and Inverse scattering}

\subsection{The direct scattering}
As an integrable system, the massive Thirring model \eqref{eq: MTM} is the compatibility condition for the following Lax pair:
\begin{equation}
\label{eq:lax}
\begin{array}{c}
\psi_x=L\psi \\
\psi_t=A\psi
\end{array}
\end{equation}
in the sense that
\begin{equation}
L_{t}-A_{x}+[L, A]=0
\end{equation}
where
\begin{subequations}
\begin{align}
L&=\frac{i}{4}\left(|u|^{2}-|v|^{2}\right) \sigma_{3}-\frac{i \zeta}{2}\left(\begin{array}{cc}
0 & \bar{v} \\
v & 0
\end{array}\right)+\frac{i}{2 \zeta}\left(\begin{array}{cc}
0 & \bar{u} \\
u & 0
\end{array}\right)+\frac{i}{4}\left(\zeta^{2}-\frac{1}{\zeta^{2}}\right) \sigma_{3}
\\
A&=-\frac{i}{4}\left(|u|^{2}+|v|^{2}\right) \sigma_{3}-\frac{i \zeta}{2}\left(\begin{array}{cc}
0 & \bar{v} \\
v & 0
\end{array}\right)-\frac{i}{2 \zeta}\left(\begin{array}{cc}
0 & \bar{u} \\
u & 0
\end{array}\right)+\frac{i}{4}\left(\zeta^{2}+\frac{1}{\zeta^{2}}\right) \sigma_{3}
\end{align}
\end{subequations}
and $\sigma_3$ is the third \textit{Pauli} matrix
$$\sigma_3=\twomat{1}{0}{0}{-1}.$$
We now define the simultaneous solutions ($2\times 2 $ matrix-valued \textit{Jost} functions)
$$\psi^\pm=\left[ \psi_{1}^{(\pm)}, \psi_{2}^{(\pm)} \right]$$
to the Lax pair \eqref{eq:lax} in the column-wise form and prescribe the boundary conditions at $\pm\infty$:
\begin{subequations}
\begin{align}
\psi_{1}^{(-)}(x ; \zeta) &\sim\left(\begin{array}{l}
1 \\
0
\end{array}\right) e^{ {i x}\left(\zeta^{2}-\zeta^{-2}\right)/4}, \quad \psi_{2}^{(-)}(x ; \zeta) \sim\left(\begin{array}{l}
0 \\
1
\end{array}\right) e^{{-i x}\left(\zeta^{2}-\zeta^{-2}\right)/4 } \quad \text { as } x \rightarrow-\infty
\\
\psi_{1}^{(+)}(x ; \zeta) &\sim\left(\begin{array}{l}
1 \\
0
\end{array}\right) e^{ {i x}\left(\zeta^{2}-\zeta^{-2}\right)/4 }, \quad \psi_{2}^{(+)}(x ; \zeta) \sim\left(\begin{array}{c}
0 \\
1
\end{array}\right) e^{-i x\left(\zeta^{2}-\zeta^{-2}\right)/4 } \quad \text { as } x \rightarrow+\infty.
\end{align}
\end{subequations}
We further normalize the Jost functions at $x=\pm\infty$:
\begin{subequations}
\begin{align}
\label{jost; normal-1}
  m_{1}^{(\pm)}(x ; \zeta)&=\psi_{1}^{(\pm)}(x ; \zeta) e^{ -i x \left(\zeta^{2}-\zeta^{-2}\right)/4 }\\
  \label{jost; normal-2}
  m_{2}^{(\pm)}(x ; \zeta)&=\psi_{2}^{(\pm)}(x ; \zeta) e^{i x \left(\zeta^{2}-\zeta^{-2}\right)/4 }
\end{align}
\end{subequations}
such that
\begin{subequations}
\begin{align}
\lim _{x \rightarrow \pm \infty} m_{1}^{(\pm)}(x ; \zeta)&= e_1 =(1,0)^{\T}\\
\lim _{x \rightarrow \pm \infty} m_{2}^{(\pm)}(x ; \zeta)&=e_{2}=(0,1)^{\T}.
\end{align}
\end{subequations}
It is easy to see from \eqref{eq:lax} that the normalized Jost functions \eqref{jost; normal-1}-\eqref{jost; normal-2} satisfy the following \textit{Volterra} integral equations:
\begin{eqnarray}\label{Volterra; origin}
\begin{aligned}
m^{(\pm)}(x;\zeta):&=\left[m_1^{(\pm)}(x;\zeta), m_2^{(\pm)}(x;\zeta)\right]\\
&=I+\int^x_{\pm\infty}e^{\frac i4(\zeta^2-\zeta^{-2})(x-y)\ad \sigma_3}\left[Q(u,v; \zeta) m^{(\pm)}(y ; \zeta)\right]dy
\end{aligned}
\end{eqnarray}
where
\begin{equation}
Q(u,v;\zeta)=\frac{i}{4}\left(|u|^{2}-|v|^{2}\right) \sigma_{3}-\frac{i \zeta}{2}\left(\begin{array}{cc}
0 & \bar{v} \\
v & 0
\end{array}\right)+\frac{i}{2 \zeta}\left(\begin{array}{cc}
0 & \bar{u} \\
u & 0
\end{array}\right)
\end{equation}
and $I$ is the $2\times 2$ identity matrix.
The following lemma follows from standard \textit{Volterra} theory which can be found in \cite{BDT88}:
\begin{lemma}
\label{lemma: 1}
For every $\zeta\in \left(\mathbb R\cup i\mathbb R\backslash \{0\}\right)$, if $ Q(u,v;\zeta) \in L^1(\mathbb R)$, then Volterra integral equations \eqref{Volterra; origin} admits unique solutions $m^{(\pm)}_1(\cdot,\zeta)$ and $m^{(\pm)}_2(\cdot,\zeta)$ in $L^\infty(\mathbb R)$.
\end{lemma}
For every $\zeta\in (\mathbb R\cup i\mathbb R)\backslash \{0\}$, by standard ODE theory, there is a matrix $S(\zeta)$, the scattering matrix, with
\begin{equation}
\left[\psi^{(+)}_1,\psi^{(+)}_2\right]=\left[\psi_1^{(-)},\psi_2^{(-)}\right]S(\zeta).
\end{equation}
The matrix $S(\zeta)$ takes the form
\begin{equation}
S(\zeta)=\left(\begin{array}{cc}
a(\zeta) & {b}(\zeta) \\
\breve{b}(\zeta) & \breve{a}(\zeta)
\end{array}\right)
\end{equation}
and the determinant relation gives
\begin{equation}
a(\zeta) \breve{a}(\zeta)-b(\zeta) \breve{b}(\zeta)=1.
\end{equation}
By uniqueness we have
\begin{equation}
\psi_{11}^{(\pm)}(\zeta)=\overline{\psi_{22}^{(\pm)}(\bar{\zeta})}, \quad \psi_{12}^{(\pm)}(\zeta)=-\overline{\psi_{21}^{(\pm)}(\bar{\zeta})}.
\end{equation}
This leads to the symmetry relation of the entries of $S(\zeta)$:
\begin{equation}
\breve{a}(\zeta)=\overline{a(\bar{\zeta})}, \quad \breve{b}(\zeta)=-\overline{b(\bar{\zeta})}
\end{equation}
and
\begin{equation}
\left\{\begin{array}{l}
a(\zeta)=\text{det}\left[\psi^{(+)}_1, \psi^{(-)}_2\right] \\\\
{{\breve{b}(\zeta)}}=\text{det}\left[\psi^{(-)}_1, \psi^{(+)}_1\right].
\end{array}\right.
\end{equation}
For $\zeta\in (\mathbb R\cup i\mathbb R)\backslash \{0\}$,  the determinant of $S(\zeta)$ is given by
\begin{equation}
a(\zeta) \overline{a(\bar{\zeta})}+b(\zeta) \overline{b(\bar{\zeta})}=1.
\end{equation}
In order to find uniform norm for $m^{(\pm)}_1(\cdot,\zeta)$ and $m^{(\pm)}_2(\cdot,\zeta)$  as $\zeta\to 0$ and $\zeta\to \pm \infty$, we introduce the following two transformations from \cite{PS} and \cite{PS-1}
\subsubsection{Transformation of the Jost functions for small $\zeta$:}
Define the transformation matrix by:
\begin{equation}
    \label{mtrx: small}
    T(u; \zeta)=\twomat{1}{0}{u}{\zeta^{-1}}
\end{equation}
and set
$${\Psi}:=T \psi.$$
Direct computation shows that $\Psi$ solves the following linear spectral problem
\begin{equation}
\label{spec: small}
{\Psi}_{x}={\mathcal{L}} {\Psi}
\end{equation}
where
\begin{equation}
{\mathcal{L}}={Q}_{1}(u, v)+{\zeta^{2}} {Q}_{2}(u, v)+\frac{i}{4}\left(\zeta^{2}-\frac{1}{\zeta^{2}}\right) \sigma_{3}
\end{equation}
and
\begin{subequations}
\begin{align}
{Q}_{1}(u, v)&=\left(\begin{array}{cc}
-\dfrac{i}{4}\left(|u|^{2}+|v|^{2}\right) & \dfrac{i}{2} \bar{u} \\
u_{x}-\dfrac{i}{2}|v|^{2} u-\dfrac{i}{2} v & \dfrac{i}{4}\left(|u|^{2}+|v|^{2}\right)
\end{array}\right)\\
\smallskip
{Q}_{2}(u, v) &=\dfrac{i}{2}\left(\begin{array}{cc}
u\bar{v}  & -\bar{v} \\
u+u^2\bar{v} & -u\bar{v}
\end{array}\right).
\end{align}
\end{subequations}
Letting $\lambda=\zeta^2$, we define a new set of normalized Jost functions
\begin{subequations}
\begin{align}
{n}_{1}^{(\pm)}(x ; \lambda)&={T}(u(x) ; \zeta) m_{1}^{(\pm)}(x ; \zeta)\\
 {n}_{2}^{(\pm)}(x ; \lambda)&=\zeta {T}(u(x) ; \zeta) m_{2}^{(\pm)}(x ; \zeta)
\end{align}
\end{subequations}
subject to the constant boundary conditions at infinity:
\begin{subequations}
\begin{align}
\lim _{x \rightarrow \pm \infty} {n}_{1}^{(\pm)}(x ; \lambda)&=e_{1}=(1,0)^\text{T}\\
 \lim _{x \rightarrow \pm \infty} {n}_{2}^{(\pm)}(x ; \lambda)&=e_{2}=(0,1)^\text{T}.
\end{align}
\end{subequations}
Setting
\begin{equation}
    \label{eq: J}
    J(\lambda):=\dfrac{1}{4}\left(\lambda-\dfrac{1}{\lambda} \right)
\end{equation}
the transformed Jost functions are solutions to the following \textit{Volterra} integral equations:
\begin{align}
\label{eq: n-s}
{n}^{(\pm)}(x;\lambda):&=\left[{n}_{1}^{(\pm)}(x ; \lambda), {n}_{2}^{(\pm)}(x ; \lambda)\right]\\
\nonumber
&=I+\int^x_{\pm\infty}e^{ iJ(\lambda)(x-y)\ad \sigma_3}\left(\left[{Q}_{1}(y)+\lambda {Q}_{2}(y)\right] {n}^{(\pm)}(y ; \lambda)\right)dy.
\end{align}
The following lemma will be useful:
\begin{lemma}
\label{lm:asy-s}
\cite{PS}If $ Q_1(u,v),  Q_2(u,v) \in L^1(\mathbb R)$, then for every $\lambda\in \mathbb{R}$, Volterra integral equations \eqref{eq: n-s} admit unique solutions $n^{(\pm)}_1(\cdot,\lambda)$ and $n^{(\pm)}_2(\cdot,\lambda)$ in $L^\infty(\mathbb R)$. \textit{Riemann-Lebesgue} lemma leads to the following asymptotic behavior for $|\lambda|<1$:
\begin{subequations}
\begin{align}
\lim _{\lambda \rightarrow 0} \frac{{n}_{1}^{(\pm)}(x ; \lambda)}{{n}_{1}^{\pm 0}(x)}&=e_{1}=(1,0)^\text{T}\\
\lim _{\lambda \rightarrow 0} \frac{{n}_{2}^{(\pm)}(x ; \lambda)}{{n}_{2}^{\pm0}(x)} &=e_{2}=(0,1)^\text{T}
\end{align}
\end{subequations}
where
\begin{subequations}
\begin{align}
{n}_{1}^{\pm 0}(x)&=e^{-\frac{i}{4} \int_{\pm \infty}^{x}\left(|u|^{2}+|v|^{2}\right) d y}\\
{n}_{2}^{\pm 0}(x)&=e^{\frac{i}{4} \int_{\pm \infty}^{x}\left(|u|^{2}+|v|^{2}\right) d y}.
\end{align}
\end{subequations}
If in addition $u,v \in C^{1}(\mathbb{R})$, then
\begin{subequations}
\begin{align}
    \lim _{|\lambda| \rightarrow 0} \frac 1{\lambda} \left[\frac{{n}_{1}^{(\pm)}(x ; \lambda)}{{n}_{1}^{\pm 0}(x)}-e_{1}\right]&=\left(\begin{array}{c}
-\int_{\pm \infty}^x\left[\bar{u}\left(u_x-\frac{i}{2} u|v|^2-\frac{i}{2} v\right)-\frac{i}{2} u \bar{v}\right] d y\\
2 i u_x+u|v|^2+v
\end{array}\right)\\
\lim _{|\lambda| \rightarrow 0} \frac 1{\lambda}\left[\frac{{n}_{2}^{(\pm)}(x ; \lambda)}{{n}_{2}^{\pm 0}(x)}-e_{2}\right]&=\left(\begin{array}{c}
\bar{u} \\
\int_{\pm \infty}^x\left[\bar{u}\left(u_x-\frac{i}{2} u|v|^2-\frac{i}{2} v\right)-\frac{i}{2} u \bar{v}\right] d y
\end{array}\right).
\end{align}
\end{subequations}

\end{lemma}
By using the transformation formulas for small $\zeta$, there is a new scattering data matrix $S_g(\lambda)$ such that
\begin{equation}
\left[n^{(+)}_{1}e^{ix(\lambda-\lambda^{-1})/4},n^{(+)}_{2}e^{-ix(\lambda-\lambda^{-1})/4}\right]=\left[n^{(-)}_{1}e^{ix(\lambda-\lambda^{-1})/4},n^{(-)}_{2}e^{-ix(\lambda-\lambda^{-1})/4}\right]S_g(\lambda).
\end{equation}
The matrix $S_g(\lambda)$ takes the form
\begin{equation}
S_g(\lambda)=\left(\begin{array}{cc}
\alpha(\lambda) & {\beta}(\lambda) \\
\breve{\beta}(\lambda) & \breve{\alpha}(\lambda)
\end{array}\right)=\left(\begin{array}{cc}
a(\zeta) & \zeta{b}(\zeta) \\
\zeta^{-1}\breve{b}(\zeta) & \breve{a}(\zeta)
\end{array}\right).
\end{equation}
It is clear that
\begin{subequations}
\begin{align}
    \alpha(\lambda) &=\text{det}\left[n^{(+)}_{1}e^{ix(\lambda-\lambda^{-1})/4}, n^{(-)}_{2}e^{-ix(\lambda-\lambda^{-1})/4}\right]=\overline{\breve{\alpha}(\overline{\lambda})} \\
\breve{\beta}(\lambda)&=\text{det}\left[n^{(-)}_{1}e^{ix(\lambda-\lambda^{-1})/4}, n^{(+)}_{1}e^{ix(\lambda-\lambda^{-1})/4}\right]=-\overline{{\beta}(\lambda)}/\lambda.
\end{align}
\end{subequations}
thus for $\lambda\in \bbR$
\begin{equation}
\label{mtrx:small}
S_g(\lambda)=\twomat{\alpha(\lambda)}{-\lambda\overline{\breve{\beta}(\lambda)}}{\breve{\beta}(\lambda)}{\overline{\alpha(\lambda)}}.
\end{equation}
\subsubsection{Transformation of the Jost functions for large $\zeta$}
Given matrix
\begin{equation}
\widetilde{T}(v; \zeta):=\left(\begin{array}{ll}
1 & 0 \\
v & \zeta
\end{array}\right)
\end{equation}
we define $$\widetilde{\Psi}:=\widetilde{T} \psi$$
that solves the following linear spectral problem
\begin{equation}
\label{spec: large}
\widetilde{\Psi}_{x}=\widetilde{\mathcal{L}} \widetilde{\Psi}
\end{equation}
where
\begin{equation}
\widetilde{\mathcal{L}}=\widetilde{Q}_{1}(u, v)+\frac{1}{\zeta^{2}} \widetilde{Q}_{2}(u, v)+\frac{i}{4}\left(\zeta^{2}-\frac{1}{\zeta^{2}}\right) \sigma_{3}
\end{equation}
and
\begin{subequations}
\begin{align}
\widetilde{Q}_{1}(u, v)&=\left(\begin{array}{cc}
\frac{i}{4}\left(|u|^{2}+|v|^{2}\right) & -\frac{i}{2} \bar{v} \\
v_{x}+\frac{i}{2}|u|^{2} v+\frac{i}{2} u & -\frac{i}{4}\left(|u|^{2}+|v|^{2}\right)
\end{array}\right)\\
\smallskip
\widetilde{Q}_{2}(u, v) &=-\frac{i}{2}\left(\begin{array}{cc}
\bar{u} v & -\bar{u} \\
v+\bar{u} v^{2} & -\bar{u} v
\end{array}\right).
\end{align}
\end{subequations}
Similar to the small $\zeta$ case, we define a new set of normalized Jost functions
\begin{subequations}
\label{jost:ntilde}
\begin{align}
\widetilde{n}_{1}^{(\pm)}(x ; \lambda)&=\widetilde{T}(v(x) ; \zeta) m_{1}^{(\pm)}(x ; \zeta)\\ \quad \widetilde{n}_{2}^{(\pm)}(x ; \lambda)&=\zeta^{-1} \widetilde{T}(v(x) ; \zeta) m_{2}^{(\pm)}(x ; \zeta)
\end{align}
\end{subequations}
subject to the constant boundary conditions at infinity:
\begin{subequations}
\begin{align}
\lim _{x \rightarrow \pm \infty} \widetilde{n}_{1}^{(\pm)}(x ; \lambda)&=e_{1} =(1,0)^\T
\\
\lim _{x \rightarrow \pm \infty} \widetilde{n}_{2}^{(\pm)}(x ; \lambda)&=e_{2}=(0,1)^\T.
\end{align}
\end{subequations}
The transformed Jost functions are solutions to the following \textit{Volterra} integral equations:
\begin{align}
\label{eq:n-l}
\widetilde{n}^{(\pm)}(x;\lambda):&=\left[\widetilde{n}_{1}^{(\pm)}(x ; \lambda),\widetilde{n}_{2}^{(\pm)}(x ; \lambda)\right]\\
\nonumber
&=I+\int^x_{\pm\infty}e^{ iJ(\lambda)(x-y)\ad \sigma_3}\left(\left[\widetilde{Q}_{1}(y)+\lambda^{-1}\widetilde{Q}_{2}(y)\right] \widetilde{n}^{(\pm)}(y ; \lambda)\right)dy.
\end{align}
We need the following lemma from \cite{PS}:
\begin{lemma}
\label{lem:asy-l}
If $ Q_1(u,v),  Q_2(u,v) \in L^1(\mathbb R)$, then for every $\lambda\in \mathbb{R}\setminus \{0\}$, Volterra integral equations \eqref{eq:n-l} admits unique solutions $\widetilde{n}^{(\pm)}_1(\cdot,\lambda)$ and $\widetilde{n}^{(\pm)}_2(\cdot,\lambda)$ in the space $L^\infty(\mathbb R)$. \textit{Riemann-Lebesgue} lemma leads to the following asymptotic behavior of $|\lambda|>1$:
\begin{subequations}
\begin{align}
\lim _{|\lambda| \rightarrow \infty} \frac{\widetilde{n}_{1}^{(\pm)}(x ; \lambda)}{\widetilde{n}_{1}^{\pm\infty}(x)}=e_{1}=(1,0)^\text{T}
\\
\lim _{|\lambda| \rightarrow \infty} \frac{\widetilde{n}_{2}^{(\pm)}(x ; \lambda)}{\widetilde{n}_{2}^{\pm\infty}(x)}=e_{2}=(0,1)^\text{T}
\end{align}
\end{subequations}
where
\begin{subequations}
\begin{align}
\label{ntilde-1}
\widetilde{n}_{1}^{\pm\infty}(x)&=e^{\frac{i}{4} \int_{\pm \infty}^{x}\left(|u|^{2}+|v|^{2}\right) d y}\\
\label{ntilde-2}
\widetilde{n}_{2}^{\pm\infty}(x)&=e^{-\frac{i}{4} \int_{\pm \infty}^{x}\left(|u|^{2}+|v|^{2}\right) d y}.
\end{align}
\end{subequations}
If in addition $v \in C^{1}(\mathbb{R})$, then
\begin{subequations}
\begin{align}
    \lim _{|\lambda| \rightarrow \infty} \lambda\left[\frac{\widetilde{n}_{1}^{(\pm)}(x ; \lambda)}{\widetilde{n}_{1}^{\pm\infty}(x)}-e_{1}\right] &=\left(\begin{array}{c}
-\int_{\pm \infty}^{x}\left[\bar{v}\left(v_{x}+\frac{i}{2}|u|^{2} v+\frac{i}{2} u\right)+\frac{i}{2} \bar{u} v\right] d y \\
-2 i v_{x}+|u|^{2} v+u
\end{array}\right)\\
\lim _{|\lambda| \rightarrow \infty} \lambda\left[\frac{\widetilde{n}_{2}^{(\pm)}(x ; \lambda)}{\widetilde{n}_{2}^{\pm\infty}(x)}-e_{2}\right]&=\left(\begin{array}{c}
\bar{v} \\
\int_{\pm \infty}^{x}\left[\bar{v}\left(v_{x}+\frac{i}{2}|u|^{2} v+\frac{i}{2} u\right)+\frac{i}{2} \bar{u} v\right] d y
\end{array}\right).
\end{align}
\end{subequations}
\end{lemma}
By using the transformation formulas for large $\zeta$, there is a new scattering data matrix $\widetilde{S}_g(\lambda)$ such that
\begin{equation}\label{Scattering Matrix}
\left[\widetilde{n}^{(+)}_{1}e^{ix(\lambda-\lambda^{-1})/4},\widetilde{n}^{(+)}_{2}e^{-ix(\lambda-\lambda^{-1})/4}\right]=\left[\widetilde{n}^{(-)}_{1}e^{ix(\lambda-\lambda^{-1})/4},\widetilde{n}^{(-)}_{2}e^{-ix(\lambda-\lambda^{-1})4}\right]\widetilde{S}_g(\lambda).
\end{equation}
The matrix $\widetilde{S}_g(\lambda)$ takes the form
\begin{equation}
\label{mtrx:S}
\widetilde{S}_g(\lambda)=\left(\begin{array}{cc}
\widetilde{\alpha}(\lambda) & \widetilde{\beta}(\lambda)  \\
\breve{\widetilde{\beta}}(\lambda)& \breve{\widetilde{\alpha}}(\lambda)
\end{array}\right)=\left(\begin{array}{cc}
a(\zeta) & \zeta^{-1}b(\zeta) \\
\zeta \breve{b}(\zeta) & \breve{a}(\zeta)
\end{array}\right).
\end{equation}
It is clear that
\begin{equation}
\label{eq:n-r}
\left\{\begin{array}{l}
\widetilde{\alpha}(\lambda)=\text{det}\left[\widetilde{n}^{(+)}_{1}e^{ix(\lambda-\lambda^{-1})/4}, \widetilde{n}^{(-)}_{2}e^{-ix(\lambda-\lambda^{-1})/4}\right]=\overline{\breve{\widetilde{\alpha}}(\overline{\lambda}) } \\\\
\breve{\widetilde{\beta}}(\lambda)=\text{det}\left[\widetilde{n}^{(-)}_{1}e^{ix(\lambda-\lambda^{-1})/4}, \widetilde{n}^{(+)}_{1}e^{ix(\lambda-\lambda^{-1})/4}\right]=-\lambda \overline{{\widetilde{\beta}}(\lambda)}.
\end{array}\right.
\end{equation}
thus for $\lambda\in \bbR$
\begin{equation}
\widetilde{S}_g(\lambda)=\twomat{\widetilde{\alpha}(\lambda)}{{\widetilde{\beta}}(\lambda)}{-\lambda \overline{{\widetilde{\beta}}(\lambda)}}{\overline{\widetilde{\alpha}(\overline{\lambda})}}
\end{equation}
and we obtain
\begin{equation}
\widetilde{\alpha}(\lambda)=\alpha(\lambda),\quad \breve{\widetilde{\beta}}(\lambda)=-\overline{\beta(\lambda)},\quad \breve{\beta}(\lambda)=-\overline{\widetilde{\beta}(\lambda)}, \quad \lambda\in {\mathbb R\backslash \{0\}}.
\end{equation}
and
\begin{subequations}
\begin{align}
    \lim_{\lambda \rightarrow 0} \alpha(\lambda)&=e^{\frac i4\int_{\mathbb R}\left(|u|^2+|v|^2\right)dy}=: \alpha_{0}\\
    \lim _{|\lambda| \rightarrow \infty} \alpha(\lambda)&=e^{-\frac i4\int_{\mathbb R}\left(|u|^2+|v|^2\right)dy}=: \alpha_{\infty}
\end{align}
\end{subequations}
We define the \textit{reflection} coefficient which enjoys the following useful relations :
\begin{subequations}
\begin{align}
\label{connection-1}
\widetilde{r}(\lambda):=\dfrac {{\widetilde{\beta}}(\lambda)}{\widetilde{\alpha}(\lambda)}&=\dfrac {\dfrac {{b}({\zeta})}{{\zeta}}}{\widetilde{\alpha}(\lambda)}=\dfrac{{{-\overline{\breve{\beta}(\lambda)}}}}{\alpha(\lambda)}\\
\nonumber
\\
\label{connection-2}
\breve{\widetilde{r}}(\lambda):=\dfrac {\breve{\widetilde{\beta}}(\lambda)}{\breve{\widetilde{\alpha}}(\lambda)}&=\dfrac {\zeta \breve{b}(\zeta)}{\overline{\widetilde{\alpha}(\lambda)}}=-\lambda\overline{\widetilde{r}(\lambda)}.
\end{align}
\end{subequations}
\subsubsection{Connecting different pairs of Jost functions}(see \cite[section 2.3]{PS})
 For the following two sets of Jost functions
\begin{equation}
\left\{n_{1}^{(\pm)}(\cdot ; \lambda), n_{2}^{(\pm)}(\cdot ; \lambda)\right\}, \quad\left\{\widetilde{n}_{1}^{(\pm)}(\cdot ; \lambda), \widetilde{n}_{2}^{(\pm)}(\cdot ; \lambda)\right\},
\end{equation}
we can directly read off that for every $\lambda \in \bbC\setminus \lbrace 0 \rbrace$ :
\begin{subequations}
\begin{align}
n_{1}^{(\pm)}(x ; \lambda)&=\left(\begin{array}{cc}
1 & 0 \\
u(x)-\lambda^{-1}v(x) & \lambda^{-1}
\end{array}\right) \widetilde{n}_{1}^{(\pm)}(x ; \lambda)
\\
\nonumber
\\
n_{2}^{(\pm)}(x ; \lambda)&=\left(\begin{array}{cc}
\lambda & 0 \\
u(x) \lambda-v(x) & 1
\end{array}\right) \widetilde{n}_{2}^{(\pm)}(x ; \lambda)
\end{align}
\end{subequations}
 and in the opposite direction
 \begin{subequations}
 \begin{align}
\widetilde{n}_{1}^{(\pm)}(x ; \lambda)&=\left(\begin{array}{cc}
1 & 0 \\
v(x)-\lambda u(x) & \lambda
\end{array}\right) n_{1}^{(\pm)}(x ; \lambda)
\\
\widetilde{n}_{2}^{(\pm)}(x ; \lambda)&=\left(\begin{array}{cc}
\lambda^{-1}& 0 \\
v(x) \lambda^{-1}-u(x) & 1
\end{array}\right) n_{2}^{(\pm)}(x ; \lambda)
\end{align}
\end{subequations}
From Lemma \ref{lm:asy-s} and Lemma \ref{lem:asy-l} we can deduce that
\begin{subequations}
\begin{align}
\label{asy-1}
 \lim _{|\lambda| \rightarrow \infty} \frac{n_{1}^{(\pm)}(x ; \lambda)}{\widetilde{n}_{1}^{\pm\infty}(x)}&=e_{1}+u(x) e_{2} \\ \lim _{|\lambda| \rightarrow \infty} \frac{n_{2}^{(\pm)}(x ; \lambda)}{\widetilde{n}_{2}^{\pm\infty}(x)}&=\bar{v}(x) e_{1}+(1+u(x) \bar{v}(x)) e_{2}
\\
\lim _{\lambda \rightarrow 0} \frac{\widetilde{n}_{1}^{(\pm)}(x ; \lambda)}{n_{1}^{\pm 0}(x)}&=e_{1}+v(x) e_{2}\\
\label{asy-4}
    \lim _{\lambda \rightarrow 0} \frac{\widetilde{n}_{2}^{(\pm)}(x ; \lambda)}{n_{2}^{\pm 0}(x)}&=\bar{u}(x) e_{1}+(1+\bar{u}(x) v(x)) e_{2}.
\end{align}
\end{subequations}
\begin{equation}
\lim _{\lambda \rightarrow 0} \beta(\lambda)=\lim _{\lambda \rightarrow 0} \widetilde{\beta}(\lambda)=\lim _{|\lambda| \rightarrow \infty} \beta(\lambda)=\lim _{|\lambda| \rightarrow \infty} \widetilde{\beta}(\lambda)=0.
\end{equation}
\begin{definition}
\label{def:generic}
We say that the potential $(u, v)$ admits an eigenvalue at $\lambda_{0} \in \mathbb{C}^{+}$ if $\alpha\left(\lambda_{0}\right)=0$ and a resonance at $\lambda_{0} \in \mathbb{R}$ if $\alpha\left(\lambda_{0}\right)=0$.
\end{definition}

\subsubsection{Residue condition}
If $a(\zeta_j)=0$ for some $\zeta_j\in \mathbb C^{++}$ then $\breve{a}(\bar{\zeta}_j)=0$ , and by symmetry, we have $a(-\zeta_j)=0$ and $\breve{a}(-\bar{\zeta}_j)=0$.
Set
\begin{equation}
\gamma_j=\zeta^{-2}_j,\quad \lambda_j=\zeta^2_j,\quad j=1,...,N
\end{equation}
Thus we have the linear dependence of the columns
\begin{subequations}
\label{list}
\begin{align}
\psi^{(+)}_{1}(x,\zeta_j)&=b_j\psi^{(-)}_{2}(x,\zeta_j)
\\
m^{(+)}_{1}(x,\zeta_j)&=b_j m^{(-)}_{2}(x,\zeta_j)e^{-ix(\zeta^2_j-\zeta^{-2}_j)/2}
\\
n^{(+)}_{1}(x;\lambda_j)&=\dfrac {b_j}{\zeta_j} n^{(-)}_{2}(x;\lambda_j)e^{-ix(\lambda_j-\lambda^{-1}_j)/2}
\\
\widetilde{n}^{(+)}_{1}(x;\lambda_j)&=b_j \zeta_j \widetilde{n}^{(-)}_{2}(x;\lambda_j)e^{-ix(\lambda_j-\lambda^{-1}_j)/2}
\\
-\bar{b}_j\psi^{(-)}_1(x;\bar{\zeta}_j)&=\psi^{(+)}_2(x;\bar{\zeta}_j)
\\
-\bar{b}_j m^{(-)}_1(x;\bar{\zeta}_j)&= m^{(+)}_2(x;\bar{\zeta}_j)e^{-ix(\bar{\zeta}^2_j-\bar{\zeta}^{-2}_j)/2}
\\
-\bar{b}_j n^{(-)}_{1}(x;\bar{\lambda}_j)&=\dfrac {1}{\bar{\zeta}_j} n^{(+)}_2(x;\bar{\lambda}_j)e^{-ix(\bar{\lambda}_j-\bar{\lambda}^{-1}_j)/2}
\end{align}
\end{subequations}
Using symmetry reduction we have that $\breve{a}'(\bar{\zeta}_j)=\overline{a'(\zeta_j)}$ so we can define norming constant
\begin{equation}
c_j=\dfrac {2\gamma_j}{b_j a'(\zeta_j)}.
\end{equation}
\subsubsection{Time evolution of the scattering data}
Using the Lax pair \eqref{eq:lax} and the global existence result obtained in section \ref{sec:global} one easily obtain
\begin{equation}
    a(t,\zeta)=a(\zeta),\quad b(t,\zeta)=b(\zeta)e^{-\frac i2(\zeta^2+\zeta^{-2})t}
\end{equation}

\subsection{Estimations on the reflection coefficient}We recall that
\begin{equation}
\breve{\widetilde{r}}(\lambda):=\dfrac { \breve{\widetilde{\beta}} (\lambda)}{\breve{\widetilde{\alpha}}(\lambda)},\quad \widetilde{r}(\lambda):=\dfrac {\widetilde{\beta} (\lambda)}{\widetilde{\alpha}(\lambda)}.
\end{equation}
We define the following function space:
\begin{equation}
\label{sp:r}
 H^{1,1}_0(\mathbb{R}):=H^{1,1}(\mathbb{R})\cap \lbrace f: \lim_{\lambda\to 0}f(\lambda)/\lambda=0 \rbrace
\end{equation}
where 
\begin{equation}
H^{1,1}(\mathbb{R})=\lbrace f: f, f'\in L^{2,1}(\mathbb{R}) \rbrace.
\end{equation}
In this subsection we will prove the following important proposition:
\begin{proposition}
\label{prop:r}
If $(u_0,v_0)\in \mathcal{I}\times \mathcal{I} $ \eqref{space:initial} and are generic in the sense of definition \ref{def:generic'}, then $\widetilde{r}(\lambda) \in {{H^{1,1}_0(\mathbb R)}}$.
\end{proposition}
\begin{proof}
The proposition will follow from Lemma \ref{lm:large-1}, Lemma \ref{lm:large-2}, Lemma \ref{lm: small-1} and Corollary \ref{r-infty}.
\end{proof}
The importance of proposition \ref{prop:r} is two-fold. Not only it plays an important role in the reconstruction of the potentials $u$ and $v$ but also it is the key ingredient in obtaining the error term in the long time asymptotics formulas. Because of the $1/\lambda$ factor in the spectral problem \eqref{spec: large}, we will divide the proof into two cases:
 \begin{itemize}
 \item[1.]For $|\lambda|>1$ we will directly study problem \eqref{spec: large} in the spirit of nonlinear Fourier transform of \cite{Zhou98}.
 \item[2.]For $|\lambda|<1$, to avoid singularity at the origin, we instead study problem \eqref{spec: small}.
 \end{itemize}
\begin{remark}
We need the following observations:
\begin{itemize}
\item[1.] From \eqref{connection-2} we have that
$$\lambda\overline{\widetilde{r}(\lambda)}= -\dfrac {\breve{\widetilde{\beta}}(\lambda)}{\breve{\widetilde{\alpha}}(\lambda)}=-\dfrac{\overline{\beta(\lambda)}}{\breve{\alpha}(\lambda)}$$
and this allows us to make estimates on the scattering data \eqref{mtrx:small} associated to spectral problem \eqref{spec: small}.
\item[2.] To consider the estimates of reflection coefficient when $\lambda\in (-1,1)$, we need the following change of variable: $\lambda \mapsto \gamma=\frac{1}{\lambda}$ and consequently suppose that $f_1(\lambda)=f_2(\gamma)=f_1(1/\lambda)$, then
\begin{equation}\label{variable change}
\int^1_0|f'_1(\lambda)|^2d\lambda=\int^{+\infty}_1 |f'_2(\gamma)|^2\gamma^2d\gamma.
\end{equation}
Thus in order to show that $\widetilde{r}(\lambda)\in H^1((-1,1))$ and $\gamma\widetilde{r}(\gamma)\in H^1_{\gamma}(\mathbb{R}\setminus[-1,1])$ it suffices to show that $\lambda\overline{\widetilde{r}(\lambda)}\in H^1(\mathbb{R}\setminus[-1,1]) $.
\item[3.] { $\gamma\widetilde{r}(\gamma)\in H^1_{\gamma}(\mathbb{R}\setminus[-1,1])$ implies that $\lim_{\gamma \to \infty}\gamma\widetilde{r}(\gamma)=0$ and this in turn implies that $\lim_{\lambda\to 0}\widetilde{r}(\lambda)/\lambda=0$.}
\end{itemize}
\end{remark}
Before proving proposition \ref{prop:r}, we define the following mixed norm for $1\leq p, q\leq \infty$:
\begin{equation}
\|f\|_{L^p_x\left(X; L^q_\lambda(Y)\right)}:=\left(\int_X\left(\int_Y|f|^qd\lambda\right)^{\frac pq}dx\right)^{\frac 1p}.
\end{equation}
\subsubsection{Estimates on large $\lambda$}From the integral equation below
\begin{eqnarray*}
\begin{aligned}
\widetilde{n}^{(\pm)}(x;\lambda)=
I+\int^x_{\pm\infty}e^{ iJ(\lambda)(x-y)\ad \sigma_3}\left(\left[\widetilde{Q}_{1}(y)+\lambda^{-1}\widetilde{Q}_{2}(y)\right] \widetilde{n}^{(\pm)}(y ; \lambda)\right)dy
\end{aligned}
\end{eqnarray*}
we can find the scattering matrix $\widetilde{S}_g(\lambda)$:
\begin{equation}
\begin{aligned}
\widetilde{S}_g(\lambda)=I-\int^\infty_{-\infty}e^{ -iJ(\lambda)y\ad \sigma_3}\left(\left[\widetilde{Q}_{1}(y)+\lambda^{-1}\widetilde{Q}_{2}(y)\right] \widetilde{n}^{(+)}(y ; \lambda)\right)dy.
\end{aligned}
\end{equation}
To study the $L^2$ mapping property of the direct scattering map, we first split
\begin{align*}
\label{Q-split}
\widetilde{Q}_1&=\widetilde{Q}_1^{\textrm{diag}}+\widetilde{Q}_1^{\textrm{off}}\\
                &=\left(\begin{array}{cc}
\frac{i}{4}\left(|u|^{2}+|v|^{2}\right) & 0 \\
0 & -\frac{i}{4}\left(|u|^{2}+|v|^{2}\right)
\end{array}\right)+ \left(\begin{array}{cc}
0& -\frac{i}{2} \bar{v} \\
v_{x}+\frac{i}{2}|u|^{2} v+\frac{i}{2} u & 0
\end{array}\right)
\end{align*}
and consequently from \eqref{ntilde-1}-\eqref{ntilde-2} we have that
\begin{equation}
\widetilde{n}^{+\infty}(x)=I+\int_{+\infty}^x \widetilde{Q}_1^{\textrm{diag}}(y)\widetilde{n}^{+\infty}(y)dy.
\end{equation}
where
\begin{equation}
\widetilde{n}^{+\infty}(x):=\left(\widetilde{n}^{+\infty}_1(x) e_1,\widetilde{n}^{+\infty}_2(x) e_2\right).
\end{equation}
Setting
\begin{equation}
\widetilde{n}^{(+)}_{*}(y;\lambda):=\left(\widetilde{n}^{+\infty}(y)\right)^{-1}\widetilde{n}^{(+)}(y ; \lambda)
\end{equation}
we proceed to rewrite
\begin{equation}
\begin{aligned}
\label{New Volterra}
\widetilde{n}^{(+)}_{*}(x;\lambda)=&I+\int^x_{+\infty}e^{iJ(\lambda)(x-y)\ad \sigma_3}\left(\left(\widetilde{n}^{+\infty}(x)\right)^{-1}\left[\widetilde{Q}_{1}^{\textrm{off}}(y)+\lambda^{-1}\widetilde{Q}_2(y)\right] \widetilde{n}^{+\infty}(y)\left[\widetilde{n}^{(+)}_{*}(y;\lambda)-I\right]\right)dy\\
&+\int^x_{+\infty}e^{iJ(\lambda)(x-y)\ad \sigma_3}\left(\left(\widetilde{n}^{+\infty}(x)\right)^{-1}\widetilde{Q}_{1}^{\textrm{diag}}(y) \widetilde{n}^{+\infty}(y)\left[\widetilde{n}^{(+)}_{*}(y;\lambda)-I\right]\right)dy\\
&+\int^x_{+\infty}e^{iJ(\lambda)(x-y)\ad \sigma_3}\left(\left(\widetilde{n}^{+\infty}(x)\right)^{-1}\left[\widetilde{Q}_{1}^{\textrm{off}}(y)+\lambda^{-1}\widetilde{Q}_2(y)\right] \widetilde{n}^{+\infty}(y)\right)dy\\
&=:I+\widetilde{K}_1\left[ \widetilde{n}_*^{(+)}-I\right]+\widetilde{K}_2\left[ \widetilde{n}_*^{(+)}-I\right]+\widetilde{K}_1\left[I\right]
\end{aligned}
\end{equation}
This implies
\begin{equation}
\label{eq:nstar}
\widetilde{n}^{(+)}_{*}(x;\lambda)-I=\left[\mathbf{1}-\widetilde{K}_1-\widetilde{K}_2\right]^{-1}\left[\widetilde{K}_1I\right](x;\lambda).
\end{equation}
Recall that $I_\infty=\mathbb{R}\setminus[-1,1]$, then standard Volterra theory implies that
\begin{equation}
\left\|\widetilde{n}_*^{(+)}(x, \lambda)\right\|_{L_x^{\infty}\left(\mathbb{R}^{+} ; L_z^{\infty}\left(I_{\infty}\right)\right)}<\infty.
\end{equation}
\begin{remark}
\label{rmk:star}
It is clear that $ \widetilde{n}_1^{+\infty}(x)$ and $ \widetilde{n}_2^{+\infty}(x)$ are independent of the $\lambda$ variable so the estimates on $\widetilde{n}^{(+)}_{*}(x;\lambda)$ are equivalent to estimates on $\widetilde{n}^{(+)}(x;\lambda)$ which will lead to estimates on $\widetilde{r}(\lambda)$ \eqref{eq:n-r}.
\end{remark}
Before proving the next lemma, we make the following observations:
To simplify notations, we set $\lambda'=4J(\lambda)$, then for $\lambda\in(1,\infty)$, $\lambda'\in (0,\infty)$, we then let ${\mathfrak {h}}_i(\lambda):=g_i(\lambda-\frac 1\lambda)=g_i(\lambda')$ and find that
\begin{equation}\label{Variable change}
\begin{aligned}
\int_{I_\infty}\left|{\mathfrak {h}}_i(\lambda)\right|^2d\lambda \leq \int_{\mathbb R}\left|g_i(\lambda')\right|^2\left(\frac 12+\frac {\lambda'}{2\sqrt{\lambda'^2+4}}\right)d\lambda'<\infty.
\end{aligned}
\end{equation}
As a consequence, we have for instance
\begin{equation}
\begin{aligned}
g_1(\lambda'):=\int_{\mathbb R}\bar{v}\widetilde{n}^{+\infty}_2(y)e^{-\frac i2 y\lambda'}dy\quad {\mathrm {and}} \quad g_2(\lambda'):=\int_{\mathbb R}\left(v_y+\frac i2|u|^2v+\frac i2 u\right)\widetilde{n}^{+\infty}_1(y)e^{\frac i2 y\lambda'}dy
\end{aligned}
\end{equation}
belong to $L^2_{\lambda}(I_\infty)$ by \textit{Plancherel}'s theorem.

\begin{lemma}
\label{lm:large-1}
If $(u_0,v_0)\in \mathcal{I}\times \mathcal{I} $ \eqref{space:initial} and are generic in the sense of definition \ref{def:generic'}, then $\widetilde{r}(\lambda),\lambda\widetilde{r}(\lambda) \in L^2(I_\infty)$.
\end{lemma}
\begin{proof}
We decompose $\widetilde{S}_g(\lambda)$ the scattering matrix \eqref{mtrx:S} into the following
\begin{equation}
\begin{aligned}
\widetilde{S}_g(\lambda)&=I -\int^\infty_{-\infty}e^{ -iJ(\lambda)y\ad \sigma_3}\left(\left[\widetilde{Q}_{1}(y) +\lambda^{-1}\widetilde{Q}_{2}(y)\right] \widetilde{n}^{+\infty}(y) \left(\widetilde{n}^{(+)}_{*}(y;\lambda)-I\right) \right) dy\\
&\qquad -\int^\infty_{-\infty}e^{ -iJ(\lambda)y\ad \sigma_3}\left(\left[\widetilde{Q}^\off_{1}(y)+\lambda^{-1}\widetilde{Q}_{2}(y)\right] \widetilde{n}^{+\infty}(y)\right)dy\\
   &\qquad -\int^\infty_{-\infty}e^{ -iJ(\lambda)y\ad \sigma_3}\left(\left[\widetilde{Q}^\diag_{1}(y)+\lambda^{-1}\widetilde{Q}_{2}(y)\right] \widetilde{n}^{+\infty}(y)\right)dy\\
&=I-\widetilde{S}_{g,1}(\lambda)-\widetilde{S}_{g,2}(\lambda)-\widetilde{S}_{g,3}(\lambda)
\end{aligned}
\end{equation}
The proof follows directly from the proof of \cite[lemma 3.1]{CLL}. More specifically, we have that
\begin{align}
\label{resolvent bound}
\left\|n^{+}_*-I\right\|_{L_x^{\infty}\left(\mathbb{R}^{+} ; L_\lambda^2(I_\infty)\right)} & =\left\|\left[\mathbf{1}-\widetilde{K}_1-\widetilde{K}_2\right]^{-1}\left[\widetilde{K}_1I\right](x;\lambda)\right\|_{L_x^{\infty}\left(\mathbb{R}^{+} ; L_\lambda^2(I_\infty)\right)} \\
\nonumber
& \lesssim e^{\|\widetilde{Q}_{1}+\lambda^{-1}\widetilde{Q}_{2}\|_{L_x^1\left(\mathbb{R}^{+} ; L_\lambda^\infty(I_\infty)\right)}} \|\widetilde{Q}^\off_{1}+\lambda^{-1}\widetilde{Q}_{2}\|_{L_x^\infty\left(\mathbb{R}^{+} ; L_\lambda^2(I_\infty)\right)}.
\end{align}
\end{proof}

\begin{lemma}
\label{lm:large-2}
If $(u_0,v_0)\in \mathcal{I}\times \mathcal{I} $ \eqref{space:initial} and are generic in the sense of definition \ref{def:generic'}, then $\widetilde{r}(\lambda),\lambda\widetilde{r}(\lambda) \in {{H^1(I_\infty)}}$.
\end{lemma}
\begin{proof}
Using the relation \eqref{Scattering Matrix} and letting $x=0$, we obtain
\begin{equation}
\dfrac {\partial}{\partial \lambda}\widetilde{S}_g(\lambda)=\left(\dfrac {\partial}{\partial \lambda}\left(\widetilde{n}^{(-)}(0,\lambda)\right)^{-1}\right)\widetilde{n}^{(+)}(0,\lambda)+\left(\widetilde{n}^{(-)}(0,\lambda)\right)^{-1} \dfrac {\partial}{\partial \lambda}\widetilde{n}^{(+)}(0,\lambda).
\end{equation}
By the standard \textit{Volterra} theory, $\left\|\widetilde{n}^{(\pm)}(x,\lambda)\right\|_{L^\infty_x L^\infty_\lambda}<\infty$. So we only need to show
\begin{equation}
\label{n'-lambda}
\lambda \dfrac {\partial}{\partial \lambda}\left(\widetilde{n}^{(\pm)}(0,\lambda)\right)^{-1} \in L^2_\lambda(I_\infty).
\end{equation}
Indeed we only need estimates on $\widetilde{n}^{(+)}(x,\lambda)$ for $x\geq 0$ and estimates on $\widetilde{n}^{(-)}(x,\lambda)$ for $x\leq 0$ follows from symmetry. Because of remark \ref{rmk:star}, we will deal with $\widetilde{n}_*^{(\pm)}(x,\lambda)$ instead.
From \eqref{New Volterra} we have
\begin{equation}
\begin{aligned}
\dfrac {\partial}{\partial \lambda}\widetilde{n}^{(+)}_*(x,\lambda)&=\left[\dfrac {\partial}{\partial \lambda}(\widetilde{K}_1 I)+\left[\dfrac {\partial \widetilde{K}_1}{\partial \lambda}\right]\left(\widetilde{n}^{(+)}_{*}-I\right)+\left[\dfrac {\partial \widetilde{K}_2}{\partial \lambda}\right]\left(\widetilde{n}^{(+)}_{*}-I\right)\right]\\
 &\quad+ \left[ \widetilde{K}_1\left(\dfrac {\partial \widetilde{n}^{(+)}_{*}}{\partial \lambda}\right) +\widetilde{K}_2\left(\dfrac {\partial \widetilde{n}^{(+)}_{*}}{\partial \lambda}\right)\right]\\
&=\widetilde{\xi}_1(x,\lambda)+\widetilde{\xi}_2(x,\lambda)+\widetilde{\xi}_3(x,\lambda)+ \left[ \widetilde{K}_1\left(\dfrac {\partial \widetilde{n}^{(+)}_{*}}{\partial \lambda}\right) +\widetilde{K}_2\left(\dfrac {\partial \widetilde{n}^{(+)}_{*}}{\partial \lambda}\right)\right].
\end{aligned}
\end{equation}
From the resolvent bound \eqref{resolvent bound}, we only need to show
$$\lambda\widetilde{\xi}_1,\lambda\widetilde{\xi}_2, \lambda\widetilde{\xi}_3\in L^{\infty}_x\left(\mathbb R^{+}; L^2_\lambda(I_\infty)\right).$$
Since they are matri-valued functions, we study one term from each of them since the proof of the rest are similar. For $\lambda\widetilde{\xi}_1$ we treat the following term:
\begin{equation*}
\begin{aligned}
\lambda h^{(+)}_1(x,\lambda):=&\lambda\frac {\partial}{\partial \lambda}\left(\int^x_{+\infty}\frac i{2\lambda}e^{\frac i2\int^{+\infty}_y\left(|u|^2+|v|^2\right)}\bar{u}(y)e^{i(x-y)(\lambda/2-1/(2\lambda))}dy\right)\\
=&\int^x_{+\infty}-\frac i{2\lambda}e^{\frac i2\int^{+\infty}_y\left(|u|^2+|v|^2\right)}\bar{u}(y)e^{i(x-y)(\lambda/2-1/(2\lambda))}dy\\
-&\int^x_{+\infty}\frac {x-y}{4}e^{\frac i2\int^{+\infty}_y\left(|u|^2+|v|^2\right)}\bar{u}(y)e^{i(x-y)(\lambda/2-1/(2\lambda))}dy\\
-&\int^x_{+\infty}\frac {x-y}{4\lambda^2}e^{\frac i2\int^{+\infty}_y\left(|u|^2+|v|^2\right)}\bar{u}(y)e^{i(x-y)(\lambda/2-1/(2\lambda))}dy\\
=:&h^{(+)}_{1,1}(x,\lambda)-h^{(+)}_{1,2}(x,\lambda)-h^{(+)}_{1,3}(x,\lambda).
\end{aligned}
\end{equation*}
We will only show that $h^{(+)}_{1,2}\in L^\infty_x\left(\mathbb R^+; L^2_\lambda(I_\infty)\right)$. The estimates for other terms are similar. Indeed, setting $\lambda'=\lambda-1/\lambda$ again, we have
\begin{equation}\label{L2 estimate}
\begin{aligned}
\left\|h^{(+)}_{1,2}(x,\lambda')\right\|_{L^2_{\lambda'}}&=\sup_{\phi\in C^{\infty}_0, \|\phi\|_{L^{2}}=1 }\left|\int_{\mathbb R}\phi(\lambda')\left(\int^x_{+\infty} (x-y)e^{\frac i2\int^{+\infty}_y\left(|u|^2+|v|^2\right)}\bar{u}(y)e^{i(x-y)\lambda'/2}dy\right)d\lambda'\right|\\
&\leq C \sup_{\phi\in C^{\infty}_0, \|\phi\|_{L^{2}}=1}\int^{+\infty}_x\left|\mathcal{F}^{-1}\left[ \phi(\cdot)\right]\left(\frac {x-y}2\right)\right||x-y||\bar{u}|dy\\
&\leq C \sup_{\phi\in C^{\infty}_0, \|\phi\|_{L^{2}}=1}\int^{+\infty}_x\left|\mathcal{F}^{-1}\left[\phi(\cdot)\right]\left(\frac {x-y}2\right)\right||y||\bar{u}|dy
\end{aligned}
\end{equation}
where $\mathcal{F}$ is Fourier transformation and we used the fact that $|x-y|<|y|$ for $0<x<y$. Finally an application of the Schwarz inequality gives
\begin{equation}
\left\|h^{(+)}_{1,2}(x,\lambda')\right\|_{L^\infty_x(\mathbb{R}^+; L^2_{\lambda'})}\lesssim\|u\|_{L^{2,1}}.
\end{equation}
Using the same argument of \eqref{Variable change}, we show that $h^{(+)}_{1,2}(x,\lambda)\in L^\infty_x(\mathbb R^+; L^2_{\lambda}(I_\infty))$. Next, we consider
\begin{equation}
\begin{aligned}
h^{(+)}_2(x,\lambda):=&\lambda\dfrac {\partial}{\partial \lambda}\left(\int^x_{+\infty}-\frac i2 e^{\frac i2\int^{+\infty}_y\left(|u|^2+|v|^2\right)}\bar{v}e^{i(x-y)(\lambda/2-1/(2\lambda))}dy\right)\\
=&\lambda\int^x_{+\infty}\frac 14 e^{\frac i2\int^{+\infty}_y\left(|u|^2+|v|^2\right)}\bar{v}(y)(x-y)e^{i(x-y)(\lambda/2-1/(2\lambda))}dy\\
&+\frac 1\lambda\int^x_{+\infty}\frac 14 e^{\frac i2\int^{+\infty}_y\left(|u|^2+|v|^2\right)}\bar{v}(y)(x-y)e^{i(x-y)(\lambda/2-1/(2\lambda))}dy\\
=:&\frac 14\left(h^{(+)}_{2,1}(x,\lambda)+h^{(+)}_{2,2}(x,\lambda)\right)
\end{aligned}
\end{equation}
We write $h^{(+)}_{2,1}$ as 
\begin{equation}
\begin{aligned}
h^{(+)}_{2,1}(x,\lambda)=&\int^x_{+\infty} e^{\frac i2\int^{+\infty}_y\left(|u|^2+|v|^2\right)}\bar{v}(y)(x-y)e^{i(x-y)(\lambda/2-1/(2\lambda))}dy\\
&+2i\int^x_{+\infty} e^{\frac i2\int^{+\infty}_y\left(|u|^2+|v|^2\right)}\bar{v}(y)(x-y)\dfrac {\partial}{\partial y}\left(e^{i(x-y)(\lambda/2-1/(2\lambda))}\right)dy\\
=& h^{(+)}_{2,2}(x,\lambda)\\
&\int^x_{+\infty} e^{\frac i2\int^{+\infty}_y\left(|u|^2+|v|^2\right)}\left[\left((|u|^2+|v|^2)\bar{v}+\bar{v}_y\right)(x-y)-\bar{v}(y)\right]e^{i(x-y)(\lambda/2-1/(2\lambda))}dy
\end{aligned}
\end{equation}
So we conclude that
\begin{equation}
\begin{aligned}
h^{(+)}_2(x,\lambda)=&\frac 1{2\lambda}\int^x_{+\infty} e^{\frac i2\int^{+\infty}_y\left(|u|^2+|v|^2\right)}\bar{v}(y)(x-y)e^{i(x-y)(\lambda/2-1/(2\lambda))}dy\\
-&\frac 14\int^x_{+\infty} e^{\frac i2\int^{+\infty}_y\left(|u|^2+|v|^2\right)}\left[\left((|u|^2+|v|^2)\bar{v}+\bar{v}_y\right)(x-y)-\bar{v}(y)\right]e^{i(x-y)(\lambda/2-1/(2\lambda))}dy\\
=:&\widetilde{\xi}^{(+)}_{2,1}(x,\lambda)+\widetilde{\xi}^{(+)}_{2,2}(x,\lambda)
\end{aligned}
\end{equation}
It is clear that $\widetilde{\xi}^{(+)}_{2,2}(x,\lambda)\in L^2_\lambda(I_\infty)$ by \eqref{L2 estimate}. We only need to show this term $\widetilde{\xi}^{(+)}_{2,1}(x,\lambda)\in L^2_\lambda(1,\infty)$. In fact, setting $\lambda'=\lambda-1/\lambda$, we have
\begin{equation}
\begin{aligned}
\left\|\widetilde{\xi}^{(+)}_{2,1}(x,\lambda')\right\|_{L^2_{\lambda'}}&=\sup_{\phi\in C^{\infty}_0, \|\phi\|_{L^{2}}=1 }\left|\int_{\mathbb R}\phi(\lambda')\left(\int^x_{+\infty} \dfrac {x-y}{\lambda'+\sqrt{(\lambda')^2+4}}e^{\frac i2\int^{+\infty}_y\left(|u|^2+|v|^2\right)}\bar{v}(y)e^{i(x-y)\lambda'/2}dy\right)d\lambda'\right|\\
&\leq C \sup_{\phi\in C^{\infty}_0, \|\phi\|_{L^{2}}=1}\int^{+\infty}_x\left|\mathcal{F}^{-1}\left[\dfrac {\phi(\cdot)}{(\cdot)+\sqrt{(\cdot)^2+4}}\right]\left(\frac {x-y}2\right)\right||x-y||\bar{v}|dy\\
&\leq C \sup_{\phi\in C^{\infty}_0, \|\phi\|_{L^{2}}=1}\int^{+\infty}_x\left|\mathcal{F}^{-1}\left[\dfrac {\phi(\cdot)}{(\cdot)+\sqrt{(\cdot)^2+4}}\right]\left(\frac {x-y}2\right)\right||y||\bar{v}|dy
\end{aligned}
\end{equation}
which implies $\widetilde{\xi}^{(+)}_{2,1}(x,\lambda)\in L^2_\lambda(I_\infty)$ and leads to $h^{(+)}_2(x,\lambda)\in L^\infty_x\left(\mathbb R^+; L^2_{\lambda}(I_\infty)\right)$. Finally we turn to
$\lambda\widetilde{\xi}^{(+)}_2(x,\lambda)$ and  $\lambda\widetilde{\xi}^{(+)}_3(x,\lambda)$.
We are only going to work with  $\widetilde{\xi}^{(+)}_3(x,\lambda)$. 
As before we focus on the following term:
\begin{equation}
F(\lambda,x):= \frac{i}{4}\int^x_{+\infty}\lambda(x-y)e^{\frac i2\int^{+\infty}_y(|u|^2+|v|^2)}\left(|u|^{2}+|v|^{2}\right)e^{i(x-y)(\lambda/2-1/(2\lambda))}\widetilde{n}^{(+)}_{*, 12}dy.
\end{equation}
We perform integration by parts to find:
\begin{equation}
\begin{aligned}
F(\lambda,x)&=2i\int^x_{+\infty}(x-y)e^{\frac i2\int^{+\infty}_y(|u|^2+|v|^2)}(|u|^2+|v|^2)\frac {\partial}{\partial y}\left(e^{i(x-y)(\lambda/2-1/(2\lambda))}\right)\widetilde{n}^{(+)}_{*, 12}dy\\
&\quad+\int^x_{+\infty}\frac 1\lambda(x-y)e^{\frac i2\int^{+\infty}_y(|u|^2+|v|^2)}(|u|^2+|v|^2)e^{i(x-y)(\lambda/2-1/(2\lambda))}\widetilde{n}^{(+)}_{*, 12}dy\\
&=\int^x_{+\infty}\frac 1\lambda(x-y)e^{\frac i2\int^{+\infty}_y(|u|^2+|v|^2)}(|u|^2+|v|^2)e^{i(x-y)(\lambda/2-1/(2\lambda))}\widetilde{n}^{(+)}_{*, 12}dy\\
&\quad+2i\int^x_{+\infty}e^{\frac i2\int^{+\infty}_y(|u|^2+|v|^2)}(|u|^2+|v|^2)e^{i(x-y)(\lambda/2-1/(2\lambda))}\widetilde{n}^{(+)}_{*, 12}dy\\
&\quad-2i\int^x_{+\infty}(x-y)e^{\frac i2\int^{+\infty}_y(|u|^2+|v|^2)}(|u|^2+|v|^2)_y e^{i(x-y)(\lambda/2-1/(2\lambda))}\widetilde{n}^{(+)}_{*, 12}dy\\
&\quad-\int^x_{+\infty}(x-y)e^{\frac i2\int^{+\infty}_y(|u|^2+|v|^2)}\left(|u|^2+|v|^2\right)e^{i(x-y)(\lambda/2-1/(2\lambda))}\widetilde{n}^{(+)}_{*, 12}dy\\
&\quad -2i\int^x_{+\infty}(x-y)e^{\frac i2\int^{+\infty}_y(|u|^2+|v|^2)}\left(|u|^2+|v|^2\right)e^{i(x-y)(\lambda/2-1/(2\lambda))}\left[\widetilde{n}^{(+)}_{*12}\right]_ydy
\end{aligned}
\end{equation}
Applying \textit{Minkowski}'s inequality and \textit{Cauchy-Schwarz} inequality we obtain
\begin{equation}
\begin{aligned}
\|F(\lambda,x)\|_{L^\infty_x(\mathbb{R}^+; L^2_\lambda(I_\infty))}\lesssim&  \left(\|u\|_{L^{2,1}}+\|v\|_{L^{2,1}} \right)\left\|\widetilde{n}^{(+)}_{*}-I\right\|_{L^2_yL^2_\lambda}\\
+&\left\|\int^x_{+\infty}(x-y)e^{\frac i2\int^{+\infty}_y(|u|^2+|v|^2)}\left(|u|^2+|v|^2\right)e^{i(x-y)(\lambda/2-1/(2\lambda))}\left[\widetilde{n}^{(+)}_{*12}\right]_ydy\right\|_{L^\infty_xL^2_\lambda}\\
=& \|F_1(\lambda,x)\|_{L^\infty_x(\mathbb{R}^+; L^2_\lambda(I_\infty))}+ \|F_2(\lambda,x)\|_{L^\infty_x(\mathbb{R}^+; L^2_\lambda(I_\infty))}
\end{aligned}
\end{equation}
For the first term we need to show that $\left\|\widetilde{n}^{(+)}_{*}-I\right\|_{L^2_x(\mathbb{R}^+; L^2_\lambda(I_\infty))}< \infty$.
Notice that
\begin{equation*}
\begin{aligned}
\|\widetilde{K}_1 I(y,\cdot)\|_{L^2_{\lambda'}}& \leq\left(\int_{x}^{+\infty}\langle y\rangle^{2 }\left(\left(\widetilde{n}^{+\infty}(x)\right)^{-1}\left[\widetilde{Q}_{1}^{\textrm{off}}(y)+\lambda^{-1}\widetilde{Q}_2(y)\right] \widetilde{n}^{+\infty}(y)\right) d y\right)^{1 / 2}\langle x\rangle^{-1} \\
& \leq\|\widetilde{Q}_{1}+\lambda^{-1}\widetilde{Q}_2\|_{L^{2, 1}(\mathbb{R})}\langle x\rangle^{-1}
\end{aligned},
\end{equation*}
we can conclude that
\begin{equation}\label{L22 estimate}
\begin{aligned}
\left\|\widetilde{n}^{(+)}_{*}-I\right\|_{L_{y}^{2}(\mathbb R^+; L_{\lambda}^{2}((I_\infty))} & \lesssim\|\widetilde{Q}_{1}+\lambda^{-1}\widetilde{Q}_2\|_{L^{2, 1}(\mathbb{R})}\left\|\left(1-\widetilde{K}_1-\widetilde{K}_2\right)^{-1}\right\|_{L_{y}^{\infty} L_{\lambda'}^{2} \rightarrow L_{y}^{\infty} L_{\lambda'}^{2}}\\
&\quad \times\left(\int\langle y\rangle^{-2 } d y\right)^{1 / 2} \\
& \lesssim\|\widetilde{Q}_{1}+\lambda^{-1}\widetilde{Q}_2\|_{L^{2, 1}(\mathbb{R})} e^{\|\widetilde{Q}_{1}+\lambda^{-1}\widetilde{Q}_2\|_{L_{y}^{1} L_{\lambda'}^{\infty}}}.
\end{aligned}
\end{equation}
Next we show that
\begin{equation}
\|F_2(\lambda,x)\|_{L^\infty_x(\mathbb{R}^+; L^2_\lambda(I_\infty))}<\infty.
\end{equation}
 We only show how to handle the terms $\widetilde{n}^{(+)}_{*11y}$ and $\widetilde{n}^{(+)}_{*12y}$ since the other two are similar. Direct calculations on \eqref{New Volterra} leads to
\begin{equation}
\begin{aligned}
\left[\widetilde{n}^{(+)}_{*11}(x,\lambda)\right]_x=& \int^x_{+\infty}\left(\left(\widetilde{n}^{+\infty}(x)\right)_x^{-1}\left[\widetilde{Q}_{1}(y)+\lambda^{-1}\widetilde{Q}_2(y)\right] \widetilde{n}^{+\infty}(y)\left[\widetilde{n}^{(+)}_{*}(y;\lambda)-I\right]\right)_{11}dy\\
&+ \left[ \left(\widetilde{n}^{+\infty}(x)\right)_x^{-1}\left[\widetilde{Q}_{1}(x)+\lambda^{-1}\widetilde{Q}_2(x)\right] \widetilde{n}^{+\infty}(x)\left[\widetilde{n}^{(+)}_{*}(y;\lambda)-I\right] \right]_{11}\\
&+\int^x_{+\infty}\left(\left(\widetilde{n}^{+\infty}(x)\right)_x^{-1}\left[\widetilde{Q}_{1}^{\textrm{off}}(y)+\lambda^{-1}\widetilde{Q}_2(y)\right] \widetilde{n}^{+\infty}(y)\right)_{11}dy\\
&+\left(\left(\widetilde{n}^{+\infty}(x)\right)^{-1}\left[\widetilde{Q}_{1}^{\textrm{off}}(x)+\lambda^{-1}\widetilde{Q}_2(x)\right] \widetilde{n}^{+\infty}(x)\right)_{11}
\end{aligned}
\end{equation}
where $\hat{Q}(y)=\widetilde{Q}_1(y)+\widetilde{Q}_2(y)$ We just look at the diagonal entries and find that
\begin{equation}
\widetilde{n}^{(+)}_{*11x}\in L^2_x\left(\mathbb{R}^+; L^2_\lambda(I_\infty)\right)
\end{equation}
and also from symmetry:
\begin{equation}
\widetilde{n}^{(+)}_{*22x}\in L^2_x\left(\mathbb{R}^+; L^2_\lambda(I_\infty)\right).
\end{equation}
After integrating by parts and some cancellation,
\begin{equation}
\begin{aligned}
\left[\widetilde{n}^{(+)}_{*12}(x,\lambda)\right]_x=& \int^x_{+\infty}e^{iJ(\lambda)(x-y)\ad \sigma_3}\left(\left(\widetilde{n}^{+\infty}(x)\right)_x^{-1}\left[\widetilde{Q}_{1}(y)+\lambda^{-1}\widetilde{Q}_2(y)\right] \widetilde{n}^{+\infty}(y)\left[\widetilde{n}^{(+)}_{*}(y;\lambda)-I\right]\right)_{12}dy\\
&+\int^x_{+\infty}e^{iJ(\lambda)(x-y)\ad \sigma_3}\left(\left(\widetilde{n}^{+\infty}(x)\right)_x^{-1}\left[\widetilde{Q}_{1}(y)+\lambda^{-1}\widetilde{Q}_2(y)\right] \widetilde{n}^{+\infty}(y)\right)_{12}dy\\
&+\int^x_{+\infty}e^{iJ(\lambda)(x-y)\ad \sigma_3}\left(\left(\widetilde{n}^{+\infty}(x)\right)^{-1}\left[\widetilde{Q}_{1}(y)+\lambda^{-1}\widetilde{Q}_2(y)\right] \widetilde{n}^{+\infty}(y)\right)_{y,12}dy\\
& +\int^x_{+\infty}e^{iJ(\lambda)(x-y)\ad \sigma_3}\left(\left(\widetilde{n}^{+\infty}(x)\right)^{-1}\left[\widetilde{Q}_{1}(y)+\lambda^{-1}\widetilde{Q}_2(y)\right] \widetilde{n}^{+\infty}(y)\left[\widetilde{n}^{(+)}_{*}(y;\lambda)-I\right]\right)_{y,12}dy\\
\end{aligned}
\end{equation}
To see $\widetilde{n}^{(+)}_{*12x}\in L^2_x(\mathbb{R}^+; L^2_\lambda(I_\infty))$, we simplify the notation above and let
\begin{equation}
\begin{aligned}
\left[\widetilde{n}^{(+)}_{*12}(x,\lambda)\right]_x=:& \mathcal{I}(x,\lambda)+\mathcal{K}\left[\widetilde{n}^{(+)}_{*12x}\right](x,\lambda).
\end{aligned}
\end{equation}
Following the same procedure used to prove \eqref{L22 estimate}, we conclude that
\begin{equation}
\left\|\widetilde{n}^{(+)}_{*12x}\right\|_{L^2_x(\mathbb{R}^+;L^2_\lambda(I_\infty))}< \infty.
\end{equation}
This implies that $\lambda\widetilde{\xi}^{(+)}_3\in L^\infty_x(\mathbb{R}^+;L^2_\lambda(I_\infty))$ and  thus we have proven \eqref{n'-lambda} and consequently
\begin{equation}
\lambda\dfrac {\partial \widetilde{S}_g(\lambda)}{\partial \lambda}\in L^2_\lambda(I_\infty).
\end{equation}
\end{proof}
\subsubsection{Estimates for small $\lambda$}
\begin{lemma}
\label{lm: small-1}
If $(u_0,v_0)\in \mathcal{I}\times \mathcal{I} $ \eqref{space:initial} and are generic in the sense of definition \ref{def:generic'}, then $\widetilde{r}(\lambda),\lambda\widetilde{r}(\lambda) \in {{H^{1}((-1,1))}}$.
\end{lemma}
\begin{proof}
Setting $\gamma:=\frac 1\lambda$, we have
\begin{equation}
S_g(\lambda)=\left(\begin{array}{cc}
\alpha(\lambda) & \breve{\beta}(\lambda) \\
\beta(\lambda) & \breve{\alpha}(\lambda)
\end{array}\right)=\left(\begin{array}{cc}
\alpha(\gamma^{-1}) & \breve{\beta}(\gamma^{-1}) \\
\beta(\gamma^{-1}) & \breve{\alpha}(\gamma^{-1})
\end{array}\right)=:S_{ng}(\gamma).
\end{equation}
We also recall \eqref{eq: n-s} and make change of variable $\lambda\mapsto\gamma$:
\begin{align}
\label{eq: n-s'}
{n}^{(\pm)}(x;\gamma):&=I+\int^x_{\pm\infty}e^{ -iJ(\gamma)(x-y)\ad \sigma_3}\left(\left[{Q}_{1}(y)+\frac{1}{\gamma} {Q}_{2}(y)\right] {n}^{(\pm)}(y ; \lambda)\right)dy.
\end{align}
We can use the similar analysis in the proof of lemma \ref{lm:large-2} to obtain
\begin{equation}
\gamma\frac {\partial S_{ng}(\gamma)}{\partial \gamma}\in L^2_{\gamma}(I_\infty).
\end{equation}
A combination of \eqref{connection-1}-\eqref{connection-2} and \eqref{variable change} completes the proof.
\end{proof}
An obvious conclusion of Lemma \ref{eq: n-s'} is the following:
\begin{corollary}
\label{r-infty}
If $(u_0,v_0)\in \mathcal{I}\times \mathcal{I} $ \eqref{space:initial} and are generic in the sense of definition \ref{def:generic'}, we have that
$$\lim_{\lambda\to 0}\dfrac{\widetilde{r}(\lambda)}{\lambda}=0.$$
\end{corollary}

\subsection{The Riemann-Hilbert problem and inverse scattering}
We derive the Riemann-Hilbert problem needed for the reconstruction of potentials.
Let us first define matrices $\widetilde{P}(x ; \lambda) \in \mathbb{C}^{2 \times 2}$ for every $x \in \mathbb{R}$ by
\begin{equation}
\widetilde{P}(x;\lambda)=\left\{\begin{array}{l}
\left[\widetilde{n}^{(+)}_{1}(x; \lambda), \dfrac {\widetilde{n}^{(-)}_{2}(x; \lambda)}{\widetilde{\alpha}(\lambda)}\right],\quad \lambda \in {\mathbb C}^{+} \\\\
\left[\dfrac {\widetilde{n}^{(-)}_{1}\left(x ; \lambda\right)}{\breve{\widetilde{\alpha}}(\lambda)}, \widetilde{n}^{(+)}_{2}(x;\lambda)\right],\quad \lambda \in {\mathbb C}^{-}
\end{array}\right.
\end{equation}
and two reflection coefficients

The following behavior of $\widetilde{P}(x,\lambda)$ for $|\lambda|>1$ in their analytic domains.
\begin{equation}
\widetilde{P}(x ; \lambda) \rightarrow\left[\begin{array}{cc}
\widetilde{n}_{1}^{+\infty}(x) & 0 \\
0 & \widetilde{n}_{2}^{+\infty}(x)
\end{array}\right]=: \widetilde{P}^{\infty}(x), \quad \text { as }|\lambda| \rightarrow \infty
\end{equation}
In order to normalize the boundary conditions, we define
\begin{equation}
\widetilde{M}(x ; \lambda):=\left[\widetilde{P}^{\infty}(x)\right]^{-1} \widetilde{P}(x ; \lambda), \quad \lambda \in \mathbb{C}.
\end{equation}
\begin{definition}
\label{def: data}
We define the following set of \textit{scattering data}:
\begin{equation}
\label{data}
\mathcal{S}=\left\{\widetilde{r}(\lambda), \left\{\lambda_j, \widetilde{c}_j\right\}_{j=1}^{N_1}\right\}.
\end{equation}
Here $\lambda_j=\rho_j e^{i \omega_j}$ with $ \rho_j>0$, $0<\omega_j<{\pi}$ and $\widetilde{c}_j$ are non-zero complex numbers playing the role of norming constants. We arrange eigenvalues $\lbrace \lambda_j \rbrace_{j=1}^{N_1}$ in the following way: for $\rho_j>0$, we have 
$$\rho_1<\rho_2<...<\rho_j<...<\rho_{\footnotesize{ N_1}}.$$
\end{definition}

The following Riemann-Hilbert problem is formulated for the function $\widetilde{M}(x ; \lambda)$:\\
\begin{problem}\label{RHP2-v}
For each $x\in \mathbb R$ and $\mathcal{S}\subset H_0^{1,1}(\mathbb{R}) \times \mathbb{C}^{2 N_1}$, find a $2\times 2$-matrix valued function $\widetilde{M}(x,\cdot)$ such that
\begin{enumerate}
\item[(1)] $\widetilde{M}(x ; \cdot)$ is piecewise analytic in $\mathbb{C} \backslash \mathbb{R}$ with continuous boundary values
\begin{equation}
\widetilde{M}_{\pm}(x ; \lambda)=\lim _{\varepsilon \downarrow 0} \widetilde{M}(x ; \lambda \pm i \varepsilon), \quad \lambda \in \mathbb{R}
\end{equation}
\item[(2)] $\widetilde{M}(x ; \lambda) \rightarrow I$ as $|\lambda| \rightarrow \infty$.\\
\item[(3)] The boundary values $\widetilde{M}_{\pm}(x ; \cdot)$ on $\mathbb{R}$ satisfy the jump relation
\begin{equation}
\widetilde{M}_{+}(x ; \lambda)=\widetilde{M}_{-}(x ; \lambda) \widetilde{V}(x ; \lambda), \quad \lambda \in \mathbb{R}
\end{equation}
where
\begin{align}
\widetilde{V}(x ; \lambda):&=\left[\begin{array}{cc}
1 & -\widetilde{r}(\lambda) e^{\frac{i}{2}\left(\lambda-\lambda^{-1}\right) x} \\
\breve{\widetilde{r}}(\lambda) e^{-\frac{i}{2}\left(\lambda-\lambda^{-1}\right) x} & 1- \breve{\widetilde{r}}(\lambda)\widetilde{r}(\lambda)
\end{array}\right]\\
&=\left[\begin{array}{cc}
1 & -\widetilde{r}(\lambda) e^{\frac{i}{2}\left(\lambda-\lambda^{-1}\right) x} \\
-\lambda\overline{\widetilde{r}(\lambda)} e^{-\frac{i}{2}\left(\lambda-\lambda^{-1}\right) x} & 1+\lambda |\widetilde{r}(\lambda)|^2
\end{array}\right]
\end{align}
\item[(4)]Residue conditions
\begin{subequations}
\begin{align}
\Res_{\lambda=\lambda_j}\widetilde{M}(x;\lambda)&=\lim_{\lambda\rightarrow \lambda_j}\widetilde{M}\left[\begin{array}{cc}
0 & \widetilde{c}_j e^{ix(\lambda_j-\lambda^{-1}_j)/2} \\
0 & 0
\end{array}\right]
\\
\Res_{\lambda=\widetilde{\lambda}_j}\widetilde{M}(x;\lambda)&=\lim_{\lambda\rightarrow \bar{\lambda}_j}\widetilde{M}\left[\begin{array}{cc}
0 & 0 \\
-\bar{\lambda}_j \overline{\widetilde{c}_j} e^{-ix(\bar{\lambda}_j-\bar{\lambda}^{-1}_j)/2} & 0
\end{array}\right]
\end{align}
\end{subequations}
hold, where
\begin{equation}
\widetilde{c}_j=\dfrac 2{b_j a'(\zeta_j)}.
\end{equation}
\end{enumerate}
\end{problem}
\begin{proposition}\cite[(4.10)]{PS}
If the Riemann Hilbert problem \ref{RHP2-v} is solvable and $\tdr\in H^{1,1}_0(\mathbb R)$ as in proposition \ref{prop:r}, we  then have the following reconstruction formulas:
\begin{align}
 \overline{v(x)} e^{\frac{i}{2} \int_{x}^{+\infty}\left(|u|^{2}+|v|^{2}\right) d y} &=\lim _{|\lambda| \rightarrow \infty} \lambda[\widetilde{M}(x ; \lambda)]_{12}.
\end{align}
Also (see \cite[Remark 6]{PS})as a consequence of $\widetilde{V}(x;0)=I$ and \eqref{asy-1}-\eqref{asy-4}, we have
\begin{equation}
\label{M_11}
\left[\widetilde{M}(x,0)\right]_{11}=\dfrac{n_1^{\infty}(x)}{\widetilde{n}_1^\infty(x)}=e^{\frac{i}{2} \int^{+ \infty}_{x}\left(|u|^{2}+|v|^{2}\right) d y}.
\end{equation}
\end{proposition}
\begin{definition}
\label{def:generic'}
We say that the initial condition $\mathcal{I}$ \eqref{def:generic} is generic if
\begin{itemize}
\item[1.]$\widetilde{\alpha}(\lambda)$ associated to $\mathcal{I}$ has only finitely many simple zeros in $\mathbb{C}^+$ and we assume there  exists a positive number $A$ such that $|\alpha(\lambda)| \geq A>0,\quad \lambda\in \mathbb R$.
\item[2.]  For $\lbrace \lambda_j \rbrace_{j=1}^{N_1}$ where  $z_j=\rho_j e^{i\omega_j}$, $ \rho_{j_1}\neq \rho_{j_2} $
for all $j$, $k$. This will avoid the unstable structure in which solitons travel in the same velocity.
\end{itemize}
\end{definition}

\subsection{Solvability of the RHP}
\begin{theorem}\label{Solvability}
Let $\mathcal{S}\subset H_0^{1,1}(\mathbb{R}) \times \mathbb{C}^{2 N_1}$, then there exists a unique solution for RHP problem \ref{RHP2-v}.
\end{theorem}
\begin{proof}
The proof is similar to the one in \cite{Liu}. For the convenience of the reader, we will sketch the main steps
\end{proof}
We begin by formulating precisely RHP $(\mathbb R,\widetilde{V})$ problem \ref{vector problem 1} and RHP $(\Sigma, \widetilde{V}_{*})$ problem \ref{vector problem 2}.
\begin{figure}
\caption{The Augmented Contour $\Lambda$}
\vspace{.5cm}
\label{figure-lambda}
\begin{tikzpicture}[scale=0.75]
\draw[ thick] (0,0) -- (-3,0);
\draw[ thick] (-3,0) -- (-5,0);
\draw[thick,->,>=stealth] (0,0) -- (3,0);
\draw[ thick] (3,0) -- (5,0);
\node[above] at 		(2.5,0) {$+$};
\node[below] at 		(2.5,0) {$-$};
\node[right] at (3.5 , 2) {$\Gamma_j$};
\node[right] at (3.5 , -2) {$\Gamma_j^*$};
\draw[->,>=stealth] (-0.4,2) arc(360:0:0.6);
\draw[->,>=stealth] (3.4,2) arc(360:0:0.4);
\draw[->,>=stealth] (-0.4,-2) arc(0:360:0.6);
\draw[->,>=stealth] (3.4,-2) arc(0:360:0.4);
\draw [red, fill=red] (-1,2) circle [radius=0.05];
\draw [red, fill=red] (3,2) circle [radius=0.05];
\draw [red, fill=red] (-1,-2) circle [radius=0.05];
\draw [red, fill=red] (3,-2) circle [radius=0.05];
\node[right] at (5 , 0) {$\bbR$};
\draw [red, fill=red] (0, 1) circle [radius=0.05];
\draw[->,>=stealth] (0.3, 1) arc(360:0:0.3);
\draw [red, fill=red] (0, -1) circle [radius=0.05];
\draw[->,>=stealth] (0.3, -1) arc(0:360:0.3);
\end{tikzpicture}
\end{figure}
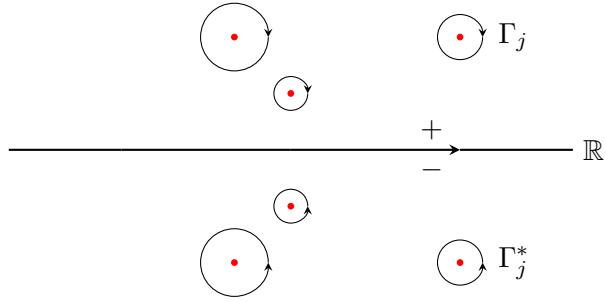
\begin{problem}\label{vector problem 1}
Fix $x\in \mathbb R$ and let $\left(\widetilde{r},\left\{\widetilde{c}_j,\lambda_j\right\}^n_{j=1}\right)$ be a set of scattering data such that $\widetilde{r}\in H^{1,1}_0(\mathbb R)$ and $\lambda_j\in \mathbb C^+$, $\widetilde{c}_j\in \mathbb C$. Find a vector-valued function $\widetilde{N}(x,\cdot)$ with the following properties:\\
\begin{enumerate}
\item[(i)] (Analyticity) $\widetilde{N}(x, \lambda)$ is a row vector-valued analytic function of $\lambda$ for $\lambda \in \mathbb{C} \backslash \Lambda$ where
\begin{equation}
\Lambda=\mathbb{R} \bigcup\left\{\Gamma_{1}, \ldots, \Gamma_{n}, \Gamma_{1}^{*}, \ldots, \Gamma_{n}^{*}\right\}
\end{equation}
\item[(ii)] (Normalization) There are two types of normalization:
\begin{eqnarray}
\begin{aligned}
&A.\; \widetilde{N}(x, \lambda)=(1,0)+\mathcal{O}\left(\lambda^{-1}\right)\quad as\; \lambda \rightarrow \infty,\\
&B.\; \widetilde{N}(x, \lambda)=\mathcal{O}\left(\lambda^{-1}\right)\quad as\; \lambda \rightarrow \infty.
\end{aligned}
\end{eqnarray}
\item[(iii)] (Jump condition) For each $\lambda \in \Lambda$, $\widetilde{N}$ has continuous boundary values $\widetilde{N}_{\pm}(\lambda)$ as $\lambda \rightarrow \Lambda$ \textit{non-tangentially} from $\mathbb{C}^{\pm}$. Moreover, the jump relation
\begin{eqnarray}
\begin{aligned}
\widetilde{N}_{+}(x, \lambda)=\widetilde{N}_{-}(x, \lambda) \widetilde{V}(\lambda)
\end{aligned}
\end{eqnarray}
holds, where for $\lambda \in \mathbb{R}$
\begin{equation}
\widetilde{V}(\lambda)=\left[\begin{array}{cc}
1 & -\widetilde{r}(\lambda) e^{\frac{i}{2}\left(\lambda-\lambda^{-1}\right) x} \\
-\lambda\overline{\widetilde{r}(\lambda)} e^{-\frac{i}{2}\left(\lambda-\lambda^{-1}\right) x} & 1+\lambda |\widetilde{r}(\lambda)|^2
\end{array}\right]
\end{equation}
and for $\lambda \in \Gamma_{i} \cup \Gamma_{i}^{*}$
\begin{equation}
\widetilde{V}(\lambda)=\left\{
\begin{array}{cc}
&\left[\begin{array}{cc}
1 & \dfrac {\widetilde{c}_j e^{ix(\lambda-\lambda^{-1})/2}}{\lambda-\lambda_j} \\
0 & 1
\end{array}\right] \quad \lambda \in \Gamma_{j} \\
&\left[\begin{array}{cc}
1 & 0 \\
\dfrac {\bar{\lambda}_j \overline{\widetilde{c}_j} e^{-ix(\lambda-\lambda^{-1})/2}}{\lambda-\bar{\lambda}_j} & 1
\end{array}\right] \quad \lambda \in \Gamma_{j}^{*}
\end{array}\right.
\end{equation}
\end{enumerate}
\end{problem}
\begin{figure}
\caption{The Augmented Contour $\Sigma'$}
\vspace{.5cm}
\label{figure-zeta}
\begin{tikzpicture}[scale=0.75]
\draw[thick,<-,>=stealth] (-3,0) -- (0,0);
\draw[ thick] (-3,0) -- (-5,0);
\draw[thick,->,>=stealth] (0,0) -- (3,0);
\draw[ thick] (3,0) -- (5,0);
\draw[thick,->,>=stealth] (0,3) -- (0,2);
\draw[ thick] (0,3) -- (0,5);
\draw[ thick] (0,2) -- (0,0);
\draw[thick,<-,>=stealth] (0,-2) -- (0,-3);
\draw[ thick] (0,-3) -- (0,-5);
\draw[ thick] (0,-2) -- (0,0);
\node[above] at 		(2.5,0) {$+$};
\node[below] at 		(2.5,0) {$-$};
\node[above] at 		(-2.5,0) {$-$};
\node[below] at 		(-2.5,0) {$+$};
\node[left] at 		     (0,2.5) {$-$};
\node[right] at 		(0,2.5) {$+$};
\node[left] at 		     (0,-2.5) {$+$};
\node[right] at 		(0,-2.5) {$-$};
\node[right] at (3.5 , 2) {$\gamma_j$};
\node[right] at (3.5 , -2) {$\gamma_j^*$};
\node[left] at (-3.5 , 2) {$-\gamma_j^*$};
\node[left] at (-3.5 , -2) {$-\gamma_j$};
\draw[->,>=stealth] (3.4,2) arc(360:0:0.4);
\draw[->,>=stealth] (3.4,-2) arc(0:360:0.4);
\draw[->,>=stealth] (-2.6,2) arc(0:360:0.4);
\draw[->,>=stealth] (-2.6,-2) arc(360:0:0.4);
\draw [green, fill=green] (-3,2) circle [radius=0.05];
\draw [red, fill=red] (3,2) circle [radius=0.05];
\draw [red, fill=red] (-3,-2) circle [radius=0.05];
\draw [green, fill=green] (3,-2) circle [radius=0.05];
\draw[->,>=stealth] (1.2,1) arc(360:0:0.2);
\draw[->,>=stealth] (1.2,-1) arc(0:360:0.2);
\draw[->,>=stealth] (-0.8,1) arc(0:360:0.2);
\draw[->,>=stealth] (-0.8,-1) arc(360:0:0.2);
\draw [green, fill=green] (-1,1) circle [radius=0.05];
\draw [red, fill=red] (1,1) circle [radius=0.05];
\draw [red, fill=red] (-1,-1) circle [radius=0.05];
\draw [green, fill=green] (1,-1) circle [radius=0.05];
\node[right] at (5 , 0) {$\Sigma_1$};
\node[above] at (0 , 5) {$\Sigma_2$};
\node[left] at (-5 , 0) {$\Sigma_3$};
\node[below] at (0 , -5) {$\Sigma_4$};
\end{tikzpicture}
\end{figure}
We proceed to define Schwartz-invariant contour $\Sigma'$ in Figure \ref{figure-zeta}
\begin{equation}
\begin{gathered}
\Sigma=\bigcup_{j=1}^{4} \Sigma_{j} \\
\Sigma^{\prime}=\Sigma \bigcup\left\{\pm \gamma_{i}\right\}_{i=1}^{n} \bigcup\left\{\pm \gamma_{i}^{*}\right\}_{i=1}^{n}
\end{gathered}
\end{equation}
\begin{problem}\label{vector problem 2}
Find a matrix-valued meromorphic function $\widetilde{N}_{*}(x,\zeta)$ with the following properties:
\begin{enumerate}
\item[(i)] (Analyticity) $\widetilde{N}_{*}(x, \zeta)$ is analytic in $\mathbb{C}\backslash \Sigma^{\prime}$ and has continuous boundary values $\widetilde{N}_{*\pm}$.\\
\item[(ii)] (Normalization)
\begin{eqnarray}
\begin{aligned}
&A. \widetilde{N}_{*}(x, \zeta)=I+\mathcal{O}\left(\zeta^{-1}\right)\quad as\quad \zeta \rightarrow \infty,\;or\\
&B. \widetilde{N}_{*}(x, \zeta)=\mathcal{O}\left(\zeta^{-1}\right)\quad as\quad \zeta \rightarrow \infty
\end{aligned}
\end{eqnarray}
\item[(iii)] (Jump Condition)On $\Sigma^{\prime}$, $\widetilde{N}_{*\pm}$ satisfy
\begin{equation}
\widetilde{N}_{*+}(x, \zeta)=\widetilde{N}_{*-}(x, \zeta) \widetilde{V}_{*}(\zeta)
\end{equation}
where on $\Sigma$
\begin{equation}
\widetilde{V}_{*}(\zeta)=\left[\begin{array}{cc}
1 & -\rho(\zeta) e^{\frac{i}{2}\left(\zeta^2-\zeta^{-2}\right) x} \\
\widetilde{\rho}(\zeta) e^{-\frac{i}{2}\left(\zeta^2-\zeta^{-2}\right) x} & 1-\rho(\zeta)\widetilde{\rho}(\zeta)
\end{array}\right]
\end{equation}
where
\begin{eqnarray}
\begin{aligned}
\rho(\zeta)=\zeta\widetilde{r}(\zeta^2),\quad \widetilde{\rho}(\zeta)=-\zeta\overline{\widetilde{r}(\zeta^2)}
\end{aligned}
\end{eqnarray}
and
\begin{equation}
\widetilde{V}_{*}(\zeta)=\left\{\begin{array}{cc}
\left(\begin{array}{cc}
1 & \frac{\widetilde{c}_{*j} e^{i x (\zeta^2-\zeta^{-2})/2}}{\zeta \mp \zeta_{j}} \\
0 & 1
\end{array}\right) & \zeta \in \pm \gamma_{j} \\
\left(\begin{array}{cc}
1 & 0 \\
\frac{\overline{\widetilde{c}_{*j}} e^{-i x (\zeta^2-\zeta^{-2})/2}}{\zeta \mp \bar{\zeta}_{j}} & 1
\end{array}\right) & \zeta \in \pm \gamma_{j}^{*}
\end{array}\right.
\end{equation}
where
\begin{eqnarray}
\begin{aligned}
\widetilde{c}_{*j}=\frac 12\widetilde{c}_{j},\;\zeta^2=\lambda\; {\mathrm {and}} \; \zeta^2_j=\lambda_j\; {\mathrm {for}}\;\zeta_j\in {\mathbb C^{++}}.
\end{aligned}
\end{eqnarray}
\end{enumerate}
\end{problem}
We then derive the Beals-Coifman integral equation for RHPs problem \ref{vector problem 1} and problem \ref{vector problem 2}. The unique solvability of Problem \ref{vector problem 1}  is equivalent to unique solvability of this integral equation. We set
\begin{equation}\label{BCRHP1}
\nu=\widetilde{N}^{+}\left(I+W_{x}^{+}\right)^{-1}=\widetilde{N}^{-}\left(I-W_{x}^{-}\right)^{-1},
\end{equation}
where
\begin{equation}
\begin{gathered}
W_{x}^{+}=\left(\begin{array}{cc}
0 & -\widetilde{r}(\lambda) e^{\frac{i}{2}\left(\lambda-\lambda^{-1}\right) x} \\
0 & 0
\end{array}\right), \quad W_{x}^{-}=\left(\begin{array}{cc}
0 & 0 \\
-\lambda\overline{\widetilde{r}(\lambda)} e^{-\frac{i}{2}\left(\lambda-\lambda^{-1}\right) x} & 0
\end{array}\right) \quad \lambda \in \mathbb{R} \\
W_{x}^{+}=\left(\begin{array}{cc}
0 & \dfrac {\widetilde{c}_j e^{ix(\lambda-\lambda^{-1})/2}}{\lambda-\lambda_j} \\
0 & 0
\end{array}\right), \quad W_{x}^{-}=\left(\begin{array}{cc}
0 & 0 \\
0 & 0
\end{array}\right) \quad \lambda \in \Gamma_{j} \\
W_{x}^{+}=\left(\begin{array}{cc}
0 & 0 \\
0 & 0
\end{array}\right), \quad W_{x}^{-}=\left(\begin{array}{cc}
0 & 0 \\
\dfrac {\bar{\lambda}_j \overline{\widetilde{c}_j} e^{-ix(\lambda-\lambda^{-1})/2}}{\lambda-\bar{\lambda}_j} & 0
\end{array}\right) \quad \lambda \in \Gamma_{j}^{*}
\end{gathered}
\end{equation}
From \eqref{BCRHP1} we have
\begin{equation}
\widetilde{N}^+ -\widetilde{N}^-=\nu(W^+_x+W^-_x).
\end{equation}
The \textit{Plemelj} formula and Type A normalization together give the following Beals-Coifman integral equation
\begin{equation}
\nu=(1,0)+\mathcal{C}_{W} \nu=(1,0)+C_{\Lambda}^{+}\left(\nu W_{x}^{-}\right)+C_{\Lambda}^{-}\left(\nu W_{x}^{+}\right).
\end{equation}
Similarly, for Type B normalization we have that
\begin{equation}\label{homonu}
\nu=\mathcal{C}_{W} \nu=C_{\Lambda}^{+}\left(\nu W_{x}^{-}\right)+C_{\Lambda}^{-}\left(\nu W_{x}^{+}\right).
\end{equation}
For $\lambda\in \mathbb R$ and Type A normalization we have
\begin{equation}\label{nu11}
\nu_{11}(x,\lambda)=1+\int^{+\infty}_{-\infty}\frac {-\nu_{12}(x,s)s\overline{\widetilde{r}(s)}e^{-i(s-s^{-1})x/2}}{s-\lambda-i0}\frac {ds}{2\pi i}+\sum^n_{j=1}\frac {1}{2\pi i}\int_{\Gamma^*_j}\dfrac {\nu_{12}(x,s)\bar{\lambda}_j \overline{\widetilde{c}_j} e^{-ix(s-s^{-1})/2}}{(s-\lambda)(s-\bar{\lambda}_j)}ds
\end{equation}
and
\begin{equation}\label{nu12}
\nu_{12}(x,\lambda)=\int^{\infty}_{-\infty}\frac {-\widetilde{r}(s)\nu_{11}(x,s)e^{i(s-s^{-1})x/2}}{s-\lambda+i0}\frac {ds}{2\pi i}+\sum^n_{j=1}\frac {1}{2\pi i}\int_{\Gamma_j}\frac {\nu_{11}(x,s)\widetilde{c}_je^{i(s-s^{-1})x/2}}{(s-\lambda)(s-\lambda_j)}ds
\end{equation}
An application of Cauchy's integral formula on \eqref{nu11} and \eqref{nu12} gives the following integral-algebraic equations:
\begin{equation}
\nu_{11}(x,\lambda)=1+\int^{+\infty}_{-\infty}\frac {-\nu_{12}(x,s)s\overline{\widetilde{r}(s)}e^{-i(s-s^{-1})x/2}}{s-\lambda-i0}\frac {ds}{2\pi i}+\sum^n_{j=1}\frac {\nu_{12}(x,\bar{\lambda}_j)\bar{\lambda}_j \overline{\widetilde{c}_j} e^{-ix(\bar{\lambda}_j-\bar{\lambda}^{-1}_j)/2}}{\bar{\lambda}_j-\lambda}
\end{equation}
and
\begin{equation}
\nu_{12}(x,\lambda)=\int^{\infty}_{-\infty}\frac {-\widetilde{r}(s)\nu_{11}(x,s)e^{i(s-s^{-1})x/2}}{s-\lambda+i0}\frac {ds}{2\pi i}+\sum^n_{j=1}\frac {-\nu_{11}(x,\lambda_j)\widetilde{c}_je^{i(\lambda_j-\lambda^{-1}_j)x/2}}{\lambda_j-\lambda}.
\end{equation}
To close the system above, we evaluate \eqref{nu11} and \eqref{nu12} at the eigenvalues to get
\begin{equation}
\nu^-_j=\nu_{12}(x,\bar{\lambda}_j)=\int^{\infty}_{-\infty}\frac {-\widetilde{r}(s)\nu_{11}(x,s)e^{i(s-s^{-1})x/2}}{s-\bar{\lambda}_j}\frac {ds}{2\pi i}+\sum^n_{k=1}\frac {-\nu_{11}(x,z_k)\widetilde{c}_ke^{i(\lambda_k-\lambda^{-1}_k)x/2}}{z_k-\bar{\lambda}_j}
\end{equation}
\begin{equation}\label{nu+}
\nu^+_j=\nu_{11}(x,\lambda_j)=1+\int^{+\infty}_{-\infty}\frac {-\nu_{12}(x,s)s\overline{\widetilde{r}(s)}e^{-i(s-s^{-1})x/2}}{s-\lambda_j}\frac {ds}{2\pi i}+\sum^n_{k=1}\frac {\nu_{12}(x,\bar{\lambda}_k)\bar{\lambda}_k \overline{\widetilde{c}_k} e^{-ix(\bar{\lambda}_k-\bar{\lambda}^{-1}_k)/2}}{\bar{\lambda}_k-\lambda_j}
\end{equation}
For Type B case, we only need to remove the $"1"$ term from the right hand side of equation \eqref{nu11} and \eqref{nu+}.The solution of RHP \ref{vector problem 1} with Type A normalization, in terms of
\begin{equation}
\nu(x,s)=(\nu_{11}(x,s),\nu_{12}(x,s))
\end{equation}
should then be
\begin{equation}
\begin{aligned}
\widetilde{N}(x, z)=(1,0) &+\frac{1}{2 \pi i} \int_{\Lambda} \frac{\nu(x, s)\left(W_{x}^{+}(s)+W_{x}^{-}(s)\right)}{s-z} d s \\
=(1,0) &+\frac{1}{2 \pi i} \int_{\mathbb{R}} \frac{\nu(x, s)\left(W_{x}^{+}(s)+W_{x}^{-}(s)\right)}{s-z} d s \\
&+\left(\sum^n_{j=1}\frac {\nu_{12}(x,\bar{\lambda}_j)\bar{\lambda}_j \overline{\widetilde{c}_j} e^{-ix(\bar{\lambda}_j-\bar{\lambda}^{-1}_j)/2}}{\bar{\lambda}_j-z}, \sum^n_{j=1}\frac {-\nu_{11}(x,\lambda_j)\widetilde{c}_je^{i(\lambda_j-\lambda^{-1}_j)x/2}}{\lambda_j-z}\right)
\end{aligned}
\end{equation}
In analogy to Problem \ref{vector problem 1}, we can deduce the following Beals-Coifman integral equation for \ref{vector problem 2}. For Type A normalization:
\begin{equation}
\mu=I+\mathcal{C}_{w} \mu=I+C_{\Sigma^{\prime}}^{+}\left(\mu w_{x}^{-}\right)+C_{\Sigma^{\prime}}^{-}\left(\mu w_{x}^{+}\right)
\end{equation}
where $I$ is the $2\times 2$ identity matrix. And for Type B normalization:
\begin{equation}\label{homomu}
\mu=\mathcal{C}_{w} \mu=C_{\Sigma^{\prime}}^{+}\left(\mu w_{x}^{-}\right)+C_{\Sigma^{\prime}}^{-}\left(\mu w_{x}^{+}\right).
\end{equation}
Componentwise, for $\zeta\in \Sigma$ we have
\begin{equation}
\begin{aligned}
\mu_{11}(x, \zeta)=&C_{\Sigma}^{+}\left[\mu_{12}(x, \cdot) \widetilde{\rho}(\cdot) e^{- i x((\cdot)^2-(\cdot)^{-2})/2}\right](\zeta) \\
&-\sum_{j=1}^{n}\left(\frac{\mu_{12}\left(x, \bar{\zeta}_{j}\right) \overline{\widetilde{c}_{*j}} e^{- i x (\bar{\zeta}_j^2-\bar{\zeta}^{-2}_j)/2}}{\zeta-\bar{\zeta}_j}+\frac{\mu_{12}\left(x, -\bar{\zeta}_{j}\right) \overline{\widetilde{c}_{*j}} e^{- i x (\bar{\zeta}_j^2-\bar{\zeta}^{-2}_j)/2}}{\zeta+\bar{\zeta}_j}\right)
\end{aligned}
\end{equation}
\begin{equation}
\begin{aligned}
\mu_{12}(x, \zeta)=&-C_{\Sigma}^{-}\left[\mu_{11}(x, \cdot) \rho(\cdot) e^{i x((\cdot)^2-(\cdot)^{-2})/2}\right](\zeta) \\
&+\sum_{j=1}^{n}\left(\frac{\mu_{11}\left(x, \zeta_{j}\right) \widetilde{c}_{*j} e^{ i x (\zeta_j^2-\zeta^{-2}_j)/2}}{\zeta-\zeta_j}+\frac{\mu_{11}\left(x, -\zeta_{j}\right) \widetilde{c}_{*j} e^{ i x (\zeta_j^2-\zeta^{-2}_j)/2}}{\zeta+\zeta_j}\right)
\end{aligned}
\end{equation}
\begin{equation}
\begin{aligned}
\mu_{21}(x, \zeta)=&C_{\Sigma}^{+}\left[\mu_{22}(x, \cdot) \widetilde{\rho}(\cdot) e^{- i x((\cdot)^2-(\cdot)^{-2})/2}\right](\zeta) \\
&-\sum_{j=1}^{n}\left(\frac{\mu_{22}\left(x, \bar{\zeta}_{j}\right) \overline{\widetilde{c}_{*j}} e^{- i x (\bar{\zeta}_j^2-\bar{\zeta}^{-2}_j)/2}}{\zeta-\bar{\zeta}_j}+\frac{\mu_{22}\left(x, -\bar{\zeta}_{j}\right) \overline{\widetilde{c}_{*j}} e^{- i x (\bar{\zeta}_j^2-\bar{\zeta}^{-2}_j)/2}}{\zeta+\bar{\zeta}_j}\right)
\end{aligned}
\end{equation}
\begin{equation}
\begin{aligned}
\mu_{22}(x, \zeta)=&-C_{\Sigma}^{-}\left[\mu_{21}(x, \cdot) \rho(\cdot) e^{i x((\cdot)^2-(\cdot)^{-2})/2}\right](\zeta) \\
&+\sum_{j=1}^{n}\left(\frac{\mu_{21}\left(x, \zeta_{j}\right) \widetilde{c}_{*j} e^{ i x (\zeta_j^2-\zeta^{-2}_j)/2}}{\zeta-\zeta_j}+\frac{\mu_{21}\left(x, -\zeta_{j}\right) \widetilde{c}_{*j} e^{ i x (\zeta_j^2-\zeta^{-2}_j)/2}}{\zeta+\zeta_j}\right)
\end{aligned}
\end{equation}
and in order to close the system, we have
\begin{equation}
\begin{aligned}
\mu_{11}\left(x, \pm \zeta_{j}\right)=&\int_{\Sigma} \frac{\mu_{12}(x, s) \widetilde{\rho}(s) e^{- i (s^{2}-s^{-2}) x/2}}{s \mp \zeta_{j}} \frac{d s}{2 \pi i} \\
&-\sum_{k=1}^{n}\left(\frac{\mu_{12}\left(x, \bar{\zeta}_{k}\right) \overline{\widetilde{c}_{*k}} e^{- i x (\bar{\zeta}_k^2-\bar{\zeta}^{-2}_k)/2}}{\pm \zeta_j-\bar{\zeta}_k}+\frac{\mu_{12}\left(x, -\bar{\zeta}_{k}\right) \overline{\widetilde{c}_{*k}} e^{- i x (\bar{\zeta}_k^2-\bar{\zeta}^{-2}_k)/2}}{\pm\zeta_j+\bar{\zeta}_k}\right)
\end{aligned}
\end{equation}
\begin{equation}
\begin{aligned}
\mu_{12}\left(x, \pm \bar{\zeta}_{j}\right)=&- \int_{\Sigma} \frac{\mu_{11}(x, s) \rho(s) e^{ i (s^{2}-s^{-2}) x/2}}{s \mp \bar{\zeta}_{j}} \frac{d s}{2 \pi i} \\
&+\sum_{k=1}^{n}\left(\frac{\mu_{11}\left(x, \zeta_{k}\right) \widetilde{c}_{*k} e^{ i x (\zeta_k^2-\zeta^{-2}_k)/2}}{\pm\bar{\zeta}_{j}-\zeta_k}+\frac{\mu_{11}\left(x, -\zeta_{k}\right) \widetilde{c}_{*k} e^{ i x (\zeta_k^2-\zeta^{-2}_k)/2}}{\pm\bar{\zeta}_{j}+\zeta_k}\right)
\end{aligned}
\end{equation}
\begin{equation}
\begin{aligned}
\mu_{21}\left(x, \pm \zeta_{j}\right)=&\int_{\Sigma} \frac{\mu_{22}(x, s) \widetilde{\rho}(s) e^{- i (s^{2}-s^{-2}) x/2}}{s \mp \zeta_{j}} \frac{d s}{2 \pi i} \\
&-\sum_{k=1}^{n}\left(\frac{\mu_{22}\left(x, \bar{\zeta}_{k}\right) \overline{\widetilde{c}_{*k}} e^{- i x (\bar{\zeta}_k^2-\bar{\zeta}^{-2}_k)/2}}{\pm \zeta_j-\bar{\zeta}_k}+\frac{\mu_{22}\left(x, -\bar{\zeta}_{k}\right) \overline{\widetilde{c}_{*k}} e^{- i x (\bar{\zeta}_k^2-\bar{\zeta}^{-2}_k)/2}}{\pm\zeta_j+\bar{\zeta}_k}\right)
\end{aligned}
\end{equation}
\begin{equation}
\begin{aligned}
\mu_{22}\left(x, \pm \bar{\zeta}_{j}\right)=&- \int_{\Sigma} \frac{\mu_{21}(x, s) \rho(s) e^{ i (s^{2}-s^{-2}) x/2}}{s \mp \bar{\zeta}_{j}} \frac{d s}{2 \pi i} \\
&+\sum_{k=1}^{n}\left(\frac{\mu_{21}\left(x, \zeta_{k}\right) \widetilde{c}_{*k} e^{ i x (\zeta_k^2-\zeta^{-2}_k)/2}}{\pm\bar{\zeta}_{j}-\zeta_k}+\frac{\mu_{21}\left(x, -\zeta_{k}\right) \widetilde{c}_{*k} e^{ i x (\zeta_k^2-\zeta^{-2}_k)/2}}{\pm\bar{\zeta}_{j}+\zeta_k}\right).
\end{aligned}
\end{equation}
 We now show the equivalence of two RHPs. We begin with the following change of variable formulas: for $\zeta\in \Sigma$
\begin{equation}
\mu_{11}(x, \zeta)=\nu_{11}\left(x, \zeta^{2}\right), \quad \mu_{12}(x, \zeta)=\zeta \nu_{12}\left(x, \zeta^{2}\right)
\end{equation}
and for $\zeta_i\in \mathbb C^{++}$
\begin{equation}
\mu_{11}\left(x, \bar{\zeta}_{i}\right)=\nu_{11}\left(x, \bar{\zeta}_{i}^{2}\right), \quad \mu_{12}\left(x, \zeta_{i}\right)=\zeta_{i} \nu_{12}\left(x, \zeta_{i}^{2}\right)
\end{equation}
It is easy to see that
\begin{equation}
\mu_{11}(x,-\zeta)=\mu_{11}(x, \zeta), \quad \mu_{12}(x,-\zeta)=-\mu_{12}(x, \zeta).
\end{equation}
We have the three lemmas from \cite{Liu}:
\begin{lemma}\label{lemma1}
For $\nu=(\nu_{11},\nu_{12})$ a solution of the homogeneous Beals-Coifman equation \eqref{homonu}, then $\nu_{11},\nu_{12}\in L^2_\lambda(\mathbb R)$ implies that $\mu_{11}\rho$ and $\mu_{12}\widetilde{\rho}$ belongs to $L^2_\zeta(\Sigma)\cap L^1_\zeta(\Sigma)$.
\end{lemma}
\begin{lemma}\label{lemma2}
For $\nu=(\nu_{11},\nu_{12})$ a solution of the homogeneous Beals-Coifman equation \eqref{homonu}, define
\begin{equation}
\mu(x, \zeta)=\twomat{\mu_{11}(x, \zeta) }{\mu_{12}(x, \zeta)}{\mu_{21}(x, \zeta) }{\mu_{22}(x, \zeta)}=\Twomat{\nu_{11}(x, \zeta^{2})}{\zeta \nu_{12}\left(x, \zeta^{2}\right)}{-\zeta \overline{\nu_{12}\left(x, \zeta^{2}\right)} }{\overline{\nu_{11}(x, \zeta^{2})}}
\end{equation}
and for $j=1...n$
\begin{equation}
\left(\begin{array}{cc}
\mu_{11}\left(x, \pm \bar{\zeta}_{j}\right) & \mu_{12}\left(x, \pm \zeta_{j}\right) \\
\mu_{21}\left(x, \pm \bar{\zeta}_{j}\right) & \mu_{22}\left(x, \pm \zeta_{j}\right)
\end{array}\right)=\Twomat{\nu_{11}\left(x, \bar{\zeta}_{j}^{2}\right)}{\pm \zeta_{j} \nu_{12}\left(x, \zeta_{j}^{2}\right)}{\mp \bar{\zeta}_{j} \overline{\nu_{12}\left(x, \zeta_{j}^{2}\right)}}{\overline{\nu_{11}\left(x, \bar{\zeta}_{j}^{2}\right)}}
\end{equation}
then $\mu$ solves integral equation \eqref{homomu}.
\end{lemma}
\begin{lemma}\label{lemma3}
From $\nu_{11},\nu_{12}\in L^2_\lambda(\mathbb R)$ that solve equation \eqref{homonu}, we can construct solution $\widetilde{N}_{*}$ to RHP \ref{vector problem 2} with Type B normalization such that $\widetilde{N}_{*\pm}(x,\cdot)\in\partial C_{\Sigma}(L^2)$.
\end{lemma}

To prove the solvability of RHP \ref{vector problem 1}, we first notice that $I-C_{W}$ is  a \textit{Fredholm} operator of the index zero\cite{ZhouXin-1}. Thus, using the \textit{Fredholm} alternative, we only need to show that the zero solution to the homogeneous equation \eqref{homonu} is unique in $L^2_{\zeta}(\mathbb R)$. And the following proposition is an application of the \textit{vanishing} lemma from \cite[proposition 9.3]{ZhouXin-1}:
\begin{proposition}
The solution to RHP problem \ref{vector problem 2} with Type B normalization is identically zero. 
\end{proposition}
This proposition together with lemma \ref{lemma1}, \ref{lemma2} and \ref{lemma3} and the construction of $\nu_{21}, \nu_{22}$ from \eqref{jost:ntilde} implies that the homogeneous Beal-Coifman integral equation given by \eqref{homonu} has only the trivial solution.

The following lemma is a standard result from \cite{ZhouXin-1}.
\begin{lemma}
If $\widetilde{N}$ is a solution to Problem \ref{vector problem 1} with $\widetilde{N}-\mathbf{1}\in \partial C_{\Sigma}(L^2)$, then $\widetilde{N}$ is unique.
\end{lemma}
\subsubsection{Reconstruction of the potential $u$}
\begin{proposition}
Suppose that the reflection coefficients $\widetilde{r}(\lambda)\in {{H^{1,1}_0(\mathbb R)}}$ as in proposition \ref{prop:r}, then from the solution to problem \ref{RHP2-v} we can derive
\begin{equation}
\label{recon-1}
\overline{u(x)}  =\lim _{|\lambda| \rightarrow 0} [\widetilde{M}(x, t ; \lambda)]_{12}
\end{equation}
\end{proposition}

\begin{proof}
By construction of \eqref{eq:n-l}, we can see that $\widetilde{M}(x ; \lambda)$ solves the equation:
\begin{equation}
\begin{aligned}
\left(\widetilde{P}^{\infty}(x) \widetilde{M}(x ; \lambda)\right)_x=&\left[\widetilde{P}^{\infty}(x) \widetilde{M}(x ; \lambda), iJ(\lambda) \sigma_3\right]\\
&+\left(\widetilde{Q}_{1}(y)+\lambda^{-1}\widetilde{Q}_{2}(y)\right)\widetilde{P}^{\infty}(x)\widetilde{M}(x ; \lambda)
\end{aligned}
\end{equation}
for $\lambda\in \bbC\setminus \lbrace 0 \rbrace$. To control the behavior of $\widetilde{M}(\lambda; x)$ near the origin, we appeal to the singular integral representation of $\widetilde{M}(\lambda; x)$:
\begin{align}
\label{M: sing int}
\widetilde{M}(\lambda; x, t) &=I+ \dfrac{1}{2\pi i}\int_{\Sigma} \dfrac{\widetilde{\nu} (W_x^+ + W_x^-)  }{z-\lambda} dz.
   \end{align}
where $\widetilde{\nu}$ is the usual \textit{Beals-Coifman} solution associated to $\widetilde{M}$. We first show that the limit in \eqref{recon-1} exists and
\begin{equation}
   \widetilde{M}_0(x, t) =I+  \dfrac{1}{2\pi i}\int_{\Sigma} \dfrac{\widetilde{\nu} (W_x^+ + W_x^-)  }{ z} dz.  \end{equation}
This is a simple consequence of the dominated convergence theorem. Since the limit is taken non-tangentially, for simplicity, we let
$\lambda=i\sigma$ where $\sigma>0$, then
\begin{align*}
\left\vert \dfrac{\widetilde{r}(z)}{z-\lambda} \right\vert &= \left\vert \dfrac{(z + i\sigma)\widetilde{r}(z)}{z^2+\sigma^2} \right\vert\\
                         &\leq \left\vert \dfrac{\widetilde{r}(z)}{z} \right\vert + \left\vert \dfrac{\sigma \widetilde{r}(z)}{z^2+\sigma^2} \right\vert\\
                          &\leq  \left\vert \dfrac{\widetilde{r}(z)}{z} \right\vert + \left\vert \dfrac{\sigma z }{z^2+\sigma^2}\right\vert \left\vert \dfrac{\widetilde{r}(z)}{z} \right\vert\\
                          & \leq \dfrac{3}{2}  \left\vert \dfrac{\widetilde{r}(z)}{z} \right\vert.
\end{align*}
The last term is in $L^1$ by (1) and (2) of Proposition \ref{prop:r} applied to the regions with $|\lambda|$ large and  $|\lambda|$  small respectively. Following the same procedure in \cite[Theorem 4]{CVZ99},  we deduce that
\begin{align*}
\left[ \widetilde{P}^{\infty}(x) \widetilde{M}(\lambda; x, t)\right]_x =& \lambda\left[\frac{i}{2}\sigma_3,  \widetilde{P}^{\infty}(x)\widetilde{ M}(\lambda; x,t)\right] - \lambda^{-1} \left[\frac{i}{2}\sigma_3,    \widetilde{P}^{\infty}(x) \widetilde{M}(\lambda; x,t)\right] \\
                 & + \widetilde{Q}_1 \widetilde{P}^{\infty}(x)\widetilde{ M}(\lambda; x,t)  + \lambda^{-1} \widetilde{Q}_2  \widetilde{P}^{\infty}(x)  \widetilde{M}(\lambda; x,t).
\end{align*}
If we multiply both sides by $\lambda$ and assume that \begin{equation}
\label{lim-m}
\lim_{\lambda\to 0} \lambda \widetilde{M}(\lambda; x,t)_x=0,
\end{equation}
we can establish
$$\widetilde{Q}_2 \widetilde{P}^{\infty}(x) \widetilde{M}_0=\left[\frac{i}{2}\sigma_3,    \widetilde{P}^{\infty}(x) \widetilde{M}_0\right] $$
from which \eqref{recon-1} can be directly read off. For \eqref{lim-m}, a direct computation gives:
\begin{align*}
\lambda \frac{\partial\widetilde{M}(\lambda; x,t)}{\partial x}&= \dfrac{1}{2\pi i}\int_{\bbR} \dfrac{\lambda \widetilde{\nu}_x (W_x^+ + W_x^-)  }{z-\lambda} dz + \dfrac{1}{4\pi }\int_{\bbR} \dfrac{ \lambda z \widetilde{\nu} W_x^+   }{z-\lambda} dz-\dfrac{1}{4\pi }\int_{\bbR} \dfrac{ \lambda z\widetilde{\nu} W_x^-   }{z-\lambda}dz  \\
              & \quad  - \dfrac{1}{4\pi }\int_{\bbR} \dfrac{\lambda\widetilde{\nu} W_x^+   }{ z( z-\lambda )} dz + \dfrac{1}{4\pi }\int_{\bbR} \dfrac{\lambda\widetilde{\nu} W_x^-  }{ z( z-\lambda )} dz\\
              & \quad + \lambda \times [\text{ $x$-derivative of the discrete part }].
\end{align*}
Here we can ignore derivative of the discrete part since this part obviously vanishes as $\lambda\to 0$. The first two integrals hold and have zero limit since $\widetilde{\nu}_x \in L^2(\bbR)$ (see \cite[Lemma 6.2.2]{Liu}) and $\widetilde{r}\in L^{2,1}(\bbR) $ by proposition \ref{prop:r}. For the last integral, we notice that
$$ \dfrac{1}{4\pi }\int_{\bbR} \dfrac{\lambda\widetilde{\nu} W^{\pm}_{\ttheta}   }{ z( z-\lambda )} dz=  \dfrac{1}{4\pi } \left( \int_{\bbR} \dfrac{\widetilde{\nu} W^{\pm}_{\ttheta}  }{z-\lambda} dz-  \int_{\bbR} \dfrac{\widetilde{\nu} W^{\pm}_{\ttheta}  }{z } dz  \right)$$
which again goes to zero by the dominated convergence theorem similarly to the argument above with the explicit formulas of $W^{\pm}_{\ttheta}$ and the fact $(\nu-I)\in L^2_\lambda (\bbR)$.

\end{proof}
\section{Conjugation}
From this section on-wards we will assume that the initial condition $(u_0, v_0)\in \mathcal{I}\times \mathcal{I}$ as in definition \ref{def:generic'}. Along a characteristic line $x=\text{v}t$ for $-1<\text{v}<1$ we have the signature table for the following phase function:
\begin{equation}
\label{phasefunction}
\widetilde{\theta}(t,x;\lambda)=\frac t2(\lambda+\lambda^{-1})+\frac x2(\lambda-\lambda^{-1})
\end{equation}
Take $\lambda=\xi+i\eta$ into $Re(i\widetilde{\theta})$, we have
\begin{equation}
Re(i\widetilde{\theta}(t,x;\lambda))=\frac{1}{2}\left[-\left(1+\frac{x}{t}\right) \eta t+\left(1-\frac{x}{t}\right) \frac{\eta t}{\xi^2+\eta^2}\right]
\end{equation}
Let $\pm\widetilde{z}_0:=\pm \sqrt{\frac {t-x}{x+t}}$ be the positive root of  $Re(i\widetilde{\theta})=0$.
\begin{figure}[H]
\label{sig-table}
\caption{Signature Table}
\begin{tikzpicture}[scale=0.8]
\draw [->] (-4,0)--(3,0);
\draw (4,0)--(3,0);
\draw [->] (0,-4)--(0,3);
\draw (0,3)--(0,4);
 

 \draw	[fill, red]  (1, 2)		circle[radius=0.06];	    
 \draw	[fill, red]  (1, -2)		circle[radius=0.06];	    
 	    
  \draw	[fill, red]  (2, 2.5)		circle[radius=0.06];	    
 \draw	[fill, red]  (2, -2.5)		circle[radius=0.06];	    
 \draw	[fill, red]  (-1, 1.5)		circle[radius=0.06];	    
 \draw	[fill, red]  (-1, -1.5)		circle[radius=0.06];	    
 \draw (0, 0)[dashed] circle[radius=2.236];

   \node [below] at (1.9,0) {\footnotesize $\widetilde{z}_0$};
    \node [below] at (-1.9,0) {\footnotesize $-\widetilde{z}_0$};
    \node[above]  at (0, 0.5) {\footnotesize $\text{Re}(i\widetilde{\theta})>0$};
    
      \node[below]  at (0, -0.5) {\footnotesize $\text{Re}(i\theta)<0$};
     
     \node[above]  at (0, 3) {\footnotesize $\text{Re}(i\widetilde{\theta})<0$};
     \node[below]  at (0, -3) {\footnotesize $\text{Re}(i\widetilde{\theta})>0$};
    \end{tikzpicture}
 \begin{center}
  \begin{tabular}{ccc}
Solitons ({\color{red} $\bullet$}) 
\end{tabular}
 \end{center}
\end{figure}
\begin{figure}
\caption{The Augmented Contour $\Lambda$}
\vspace{.5cm}
\label{figure-lambda}
\begin{tikzpicture}[scale=0.9]
 \draw[ thick] (0,0) -- (-3,0);
\draw[ thick] (-3,0) -- (-5,0);
\draw[thick,->,>=stealth] (0,0) -- (3,0);
\draw[ thick] (3,0) -- (5,0);
\node[above] at 		(2.5,0) {$+$};
\node[below] at 		(2.5,0) {$-$};
\node[right] at (3.5 , 2) {$\Gamma_j$};
\node[right] at (3.5 , -2) {$\Gamma_j^*$};
\draw[->,>=stealth] (3.4,2) arc(360:0:0.4);

\draw[->,>=stealth] (3.4,-2) arc(0:360:0.4);

\draw [red, fill=red] (3,2) circle [radius=0.05];

\draw [red, fill=red] (3,-2) circle [radius=0.05];
\node[right] at (5 , 0) {$\bbR$};

\draw (0,0) [dashed] circle [ radius=1.83];

\draw [red, fill=red] (0, 1) circle [radius=0.05];
\draw[->,>=stealth] (0.3, 1) arc(360:0:0.3);
\draw [red, fill=red] (0, -1) circle [radius=0.05];
\draw[->,>=stealth] (0.3, -1) arc(0:360:0.3);
\node[above] at 		(0, 1.25) {$\Gamma_k$};
\node[below] at 		(0, -1.25) {$\Gamma_k^*$};
 \node [below] at (1.9,0) {\footnotesize $\tz_0$};
    \node [below] at (-1.9,0) {\footnotesize $-\tz_0$};
     \draw	[fill, red]  (1.529, 1.006)		circle[radius=0.05];	  
     \draw[->,>=stealth] (1.729, 1.006) arc(360:0:0.2);
    \draw	[fill, red]  (1.529, -1.006)		circle[radius=0.05];	  
     \draw[->,>=stealth] (1.729, -1.006) arc(0:360:0.2);
    \node[above]  at (0, 0) {\footnotesize $\text{Re}(i\widetilde{\theta})>0$};
    
      \node[below]  at (0, 0) {\footnotesize $\text{Re}(i\widetilde{\theta})<0$};
    
     \node[above]  at (0, 2) {\footnotesize $\text{Re}(i\widetilde{\theta})<0$};
     \node[below]  at (0, -2) {\footnotesize $\text{Re}(i\widetilde{\theta})>0$};
     \node[above] at 		(1.9, 1.1) {$\Gamma_\ell$};
\node[below] at 		(2.1, -1.0) {$\Gamma_\ell^*$};
\end{tikzpicture}
\begin{center}
\begin{tabular}{ccc}
	Soliton ({\color{red} $\bullet$}) 
\end{tabular}
\end{center}
\end{figure}
In the figure above, we have chosen 
$$\mathrm{v}_\ell=\dfrac{x}{t}= \dfrac { 1-\rho^2_\ell}  {1+\rho^2_\ell }. $$
Recall definition \ref{def: data}, here we set $ {\mathcal{Z}}= \lbrace  \lambda_j \rbrace_{j=1}^{N_1} \ni \lambda_\ell=\rho_\ell e^{i\omega_\ell}$ with $1\leq \ell\leq N_1$. Define the set
\begin{equation}
\label{B-set}
\mathcal{B}_\ell=\lbrace \lambda_j: ~   \rho_j <\rho_\ell  \rbrace .
\end{equation}
and
\begin{equation}
\label{Up}
\Upsilon=\text{min}\lbrace \text{min}_{z,z'\in \calZ} |z-z'|, \quad \text{dist}(\calZ, \bbR)   \rbrace.
\end{equation}
For all $\lambda_j\in \mathcal{B}_\ell $, 
$ \text{Re}(i\theta(\lambda_j))>0$.
define the scalar functions 
$\widetilde{\delta}(\lambda)$ which solves the scalar \textit{Riemann-Hilbert} problem \ref{lambda-RHP} below:
\begin{problem}\label{lambda-RHP}
Given $\pm \widetilde{z}_0\in \mathbb R$ and $\widetilde{r}\in H^{1,1}_0(\mathbb R)$ find a scalar function $\widetilde{\delta}(\lambda)=\widetilde{\delta}(\lambda,\widetilde{z}_0)$, meromorphic for $\lambda\in\mathbb C\backslash [-\widetilde{z}_0,\widetilde{z}_0]$ with the following properties:
\begin{enumerate}
\item[(1)] $\widetilde{\delta}(\lambda)\rightarrow 1$ as $\lambda\rightarrow +\infty$.
\item[(2)] $\widetilde{\delta}(\lambda)$ has continuous boundary values $\widetilde{\delta}_{\pm}(\lambda)=\lim_{\epsilon\downarrow 0}\widetilde{\delta}(\widetilde\pm i\epsilon)$ for $\widetilde\in \mathbb R$.
\item[(3)] $\widetilde{\delta}_{\pm}$ obey the jump relation
\begin{equation}
\widetilde{\delta}_{+}(\lambda)=
\begin{cases}
\widetilde{\delta}_{-}(\lambda)\left(1+\lambda|\widetilde{r}(\lambda)|^{2}\right), & \lambda \in\left(-\widetilde{z}_{0}, \widetilde{z}_{0}\right) \\
\widetilde{\delta}_{-}(\lambda), & \lambda \in \mathbb{R} \backslash\left(-\widetilde{z}_{0}, \widetilde{z}_{0}\right)
\end{cases}
\end{equation}
\item[(4)] $\widetilde{\delta}(\lambda)$ has simple poles at $\lambda_j \in B_{\ell}$.
\end{enumerate}
\end{problem}
{
\begin{lemma}\label{delta}
Suppose $\widetilde{r}\in H^{1,1}_0(\mathbb R)$. Then we have the following conclusions:
\begin{enumerate}
\item [(1)] \ref{lambda-RHP} has the unique solution
\begin{eqnarray}
\begin{aligned}
\widetilde{\delta}(\lambda)=\left(\prod_{\lambda_j\in B_{\ell}}\frac {\lambda-\bar{\lambda}_j}{\lambda-\lambda_j}\right)
e^{\chi(\lambda)}
\end{aligned}
\end{eqnarray}
where
\begin{eqnarray}
\begin{aligned}
\chi(\lambda)=\frac 1{2\pi i}\int^{\widetilde{z}_0}_{-\widetilde{z}_0}\dfrac {\log(1+s|\widetilde{r}|^2)}{s-\lambda}ds
\end{aligned}
\end{eqnarray}
Define 
\begin{align*}
\kappa(\lambda)&=\frac {1}{2\pi}\log\left(1+\lambda|\tdr(\lambda)|^2\right)\\
\kappa^\pm_0&=\kappa(\pm\widetilde{z}_0)=\frac {1}{2\pi}\log\left(1\pm\widetilde{z}_0|\tdr(\pm\widetilde{z}_0)|^2\right)\\
\widetilde{\delta}_0^{ \pm}&=\left(\prod_{\lambda_j\in B_{\ell}}\frac {\pm\tz_0-\bar{\lambda}_j}{\pm\tz_0-\lambda_j}\right)\exp \left\{ \pm \frac{1}{i} \int_0^{ \pm \tz_0} \frac{\kappa(s) \mp s \cdot \kappa_0^{ \pm}}{s \mp \tz_0} d s \mp \frac{1}{i} \int_0^{\mp \tz_0} \frac{\kappa(s)}{s \mp \tz_0} d s \mp i \kappa_0^{ \pm}\right\}
\end{align*}
Here we have chosen the branch of the logarithm with $-\pi<\arg(\lambda)<\pi$.
\item[(2)] $\widetilde{\delta}(\lambda)= 1+O(\lambda^{-1})$ as $|\lambda| \rightarrow \infty$.
\item[(3)] Along any ray of the form $\{\lambda\in \mathbb C|\lambda=\pm \widetilde{z}_0+e^{i\phi}\mathbb R^+\}$ with $0<\phi<\pi$ or $\pi<\phi<2\pi$ one has
\begin{align*}
\left|\widetilde{\delta}(\lambda)-\left( {\lambda-\widetilde{z}_0}\right)^{-i\kappa^+} \widetilde{\delta}_{0}^+\right|& \leq C_{\tdr}\left|\lambda - \widetilde{z}_{0}\right|^{\frac 12}\\
\left|\widetilde{\delta}(\lambda)-\left( -\lambda-\widetilde{z}_0\right)^{i\kappa^-} \widetilde{\delta}_{0}^-\right|& \leq C_{\tdr}\left|\lambda + \widetilde{z}_{0}\right|^{\frac 12}
\end{align*}
and the implied constant depends on r through its $H^{1,1}(\mathbb R)$-norm and is independent of $\pm \widetilde{z}_0\in \mathbb R$.
\end{enumerate}
\end{lemma}
\begin{proof}
The proof of $(1)-(2)$ are formal computations. To prove $(4)$, see \cite[proposition 4.2]{Saalmann-1}.
\end{proof}
}
We now define a new unknown function matrix $\widetilde{M}^{(1)}$ using our $\widetilde{\delta}(\lambda)$:
\begin{equation}\label{modified-tildeM}
\widetilde{M}^{(1)}=\widetilde{M}(x,t;\lambda)\widetilde{\delta}(\lambda)^{\sigma_3}
\end{equation}
\begin{proposition}
The function $\widetilde{M}^{(1)}$ defined by \eqref{modified-tildeM} satisfies the following Riemann-Hilbert problem.
\end{proposition}
\begin{problem}\label{modifiedRHP}
Find a meromorphic function $\widetilde{M}^{(1)}$ on $\mathbb C\backslash \Sigma$ with the following properties:
\begin{enumerate}
\item[(1)] $\widetilde{M}^{(1)}(x,t;\lambda)\rightarrow I$ as $|\lambda|\rightarrow \infty$,
\item[(2)] $\widetilde{M}^{(1)}(x,t;\lambda)$ is analytic for $\mathbb C\backslash \Sigma$ with continuous boundary values $\widetilde{M}^{(1)}_{\pm}(x,t;\lambda)$.
\item[(3)] On $\mathbb R$, the jump relation
\begin{equation}
\widetilde{M}^{(1)}_{+}(x,t;\lambda)=\widetilde{M}^{(1)}_{-}(x,t;\lambda)\widetilde{V}^{(1)}(x,t;\lambda)
\end{equation}
holds, where
\begin{equation}
\widetilde{V}^{(1)}(x,t;\lambda)=\widetilde{\delta}_{-}(\lambda)^{-\sigma_3}V(x,t;\lambda)\widetilde{\delta}_{+}(\lambda)^{\sigma_3}
\end{equation}
The jump matrix $V^{(1)}(x,t;\lambda)$ is factorized as
\begin{equation}
\widetilde{V}^{(1)}(x,t;\lambda)=\left\{
\begin{array}{cc}
&\left(\begin{array}{cc}
1 & 0\\
-\lambda\bar{\widetilde{r}}\widetilde{\delta}^2e^{-i\widetilde{\theta}} & 1
\end{array}\right)
\left(\begin{array}{cc}
1 & -\widetilde{r}\widetilde{\delta}^{-2}e^{i\widetilde{\theta}}\\
0 & 1
\end{array}\right)
\qquad\qquad\qquad \lambda\in \mathbb R \backslash [-\widetilde{z}_0,\widetilde{z}_0],\\
&\left(\begin{array}{cc}
1 & -\frac {\widetilde{r}e^{i\widetilde{\theta}}}{\widetilde{\delta}^2_-(1+\lambda|\widetilde{r}|^2)} \\
0 & 1
\end{array}\right)
\left(\begin{array}{cc}
1 & 0 \\
\frac{-\lambda \bar{\widetilde{r}}\widetilde{\delta}^2_-e^{-i\widetilde{\theta}}}{1+\lambda|\widetilde{r}|^2 } & 1
\end{array}\right)
\qquad \lambda \in (-\widetilde{z}_0,\widetilde{z}_0).
\end{array}\right.
\end{equation}
\item[(4)](Residue condition) $\widetilde{M}^{(1)}(x,t;\lambda)$ has simple poles at each $\lambda_j$:
\begin{equation}
\Res_{\lambda=\lambda_j} \widetilde{M}^{(1)}(x,t;\lambda)=\lim_{\lambda \rightarrow \lambda_j}\widetilde{M}^{(1)}\left(\begin{array}{ccc}
0 & 0 & \\
\widetilde{c}^{-1}_je^{-i\widetilde{\theta}}\left[\left(\frac 1{\widetilde{\delta}}\right)'\right]^{-2} & 0 &
\end{array}\right) \quad \lambda_j \in B^+_{l}
\end{equation}
\begin{equation}
\Res_{\lambda=\bar{\lambda}_j} \widetilde{M}^{(1)}(x,t;\lambda)=\lim_{\lambda \rightarrow \bar{\lambda}_{j}}\widetilde{M}^{(1)}\left(\begin{array}{ccc}
0 & -\bar{\lambda}^{-1}_j\bar{\widetilde{c}}^{-1}_je^{i\widetilde{\theta}}\left[\widetilde{\delta}'(\bar{\lambda}_j)\right]^{-2} & \\
0 & 0 &
\end{array}\right) \quad \bar{\lambda}_j \in B^-_{l}
\end{equation}
\begin{equation}
\Res_{\lambda=\lambda_{l}}\widetilde{M}^{(1)}(x,t;\lambda)=\lim_{\lambda \rightarrow \lambda_l}\widetilde{M}^{(1)}\left(\begin{array}{ccc}
0 & \widetilde{c}_l e^{i\widetilde{\theta}}\widetilde{\delta}(\lambda_l)^{-2} & \\
0 & 0 &
\end{array}\right) \quad \lambda_l \in B^+_{r}
\end{equation}
\begin{equation}
\Res_{\lambda=\bar{\lambda}_l}\widetilde{M}^{(1)}(x,t;\lambda)=\lim_{\lambda \rightarrow \bar{\lambda}_l}\widetilde{M}^{(1)}\left(\begin{array}{ccc}
0 & 0 & \\
-\bar{\lambda}_l \bar{\widetilde{c}}_l e^{-i\widetilde{\theta}}[\widetilde{\delta}(\bar{\lambda}_l)]^2 & 0 &
\end{array}\right) \quad \bar{\lambda}_l \in B^-_{r}
\end{equation}
\end{enumerate}
\end{problem}
\section{Contour deformation}
We now perform contour deformation on Problem \ref{modifiedRHP}, following the standard procedure outlined in \cite{DZ93} and also \cite{BJM} in the presence of discrete spectrum. Since the phase function \eqref{phasefunction} has two critical points at $\pm \tz_0$, our new contour is chosen to be
\begin{equation}
\Sigma^{(2)}=\Sigma_{1} \cup \Sigma_{2} \cup \Sigma_{3} \cup \Sigma_{4} \cup \Sigma_{5} \cup \Sigma_{6} \cup \Sigma_{7} \cup \Sigma_{8}
\end{equation}
shown in Figure \ref{fig:contour-def} (also see \cite[Figure 8]{CVZ99})
\begin{figure}[H]
\label{fig:contour-def}
\caption{Deformation from $\mathbb{R}$ to $\Sigma^{(2)}$}
\vskip 15pt
\begin{tikzpicture}[scale=0.9]

\draw 	[->,thick,>=stealth] 	(-3.564, 0)-- (-5.296, 0);	
\draw 	[thick] 	 (-5.296, 0)--(-7.028,0 );

\draw 	[thick] 	(3.564, 0)-- (5.296, 0);	
\draw 	[->,thick,>=stealth] 	(7.028,0 )-- (5.296, 0);

\draw 	[->,thick,>=stealth] 	(-3.564, 0)-- (-1.732, 0);	
\draw 	[thick] 	 (-1.732, 0)--(0,0 );

\draw 	[->,thick,>=stealth] 	(0, 0)-- (1.732, 0);	
\draw 	[thick] 	 (1.732, 0)--(3.564,0 );

\draw[thick] 			(3.564 ,0) -- (5.296, 1);								
\draw[->,thick,>=stealth] 	(   7.028 ,2  )--	(5.296, 1);

\draw[->,thick,>=stealth] 			(1.732 ,1) -- (3.564 ,0) ;
\draw[->,thick,>=stealth] 			(0,0)--(1.732 ,1);

\draw[thick] 			(3.564 ,0) -- (5.296, -1);								
\draw[->,thick,>=stealth] 	(   7.028 , -2  )--	(5.296, -1);

\draw[thick] 			(1.732 ,-1) -- (3.564 ,0) ;
\draw[->,thick,>=stealth] 			(0,0)--(1.732 ,-1);

\draw[->,thick,>=stealth] 			 (-3.564 ,0)--(-1.732 ,-1) ;
\draw[thick] 			(-1.732 ,-1)--(0,0);
\draw[->,thick,>=stealth] 			 (-3.564 ,0)--(-1.732 ,1) ;
\draw[thick] 			(-1.732 , 1)--(0,0);

\draw[->,thick,>=stealth] 			(-3.564 ,0) -- (-5.296, 1);							
\draw[thick] 	(   -7.028 ,2  )--	(-5.296, 1);
\draw[->,thick,>=stealth] 			(-3.564 ,0) -- (-5.296, -1);							
\draw[thick] 	(   -7.028 , -2  )--	(-5.296, -1);

\draw	[fill]							(-3.564 ,0)		circle[radius=0.1];	
\draw	[fill]							(3.564, 0)		circle[radius=0.1];
\draw							(0,0)		circle[radius=0.1];
\node[below] at (-3.564,-0.1)			{$-\tz_0$};
\node[below] at (3.564,-0.1)			{$\tz_0$};
\node[right] at (5, 1.6)					{$\Sigma_1$};
\node[left] at (-5, 1.6)					{$\Sigma_2$};
\node[left] at (-5,-1.6)					{$\Sigma_3$};
\node[right] at (5,-1.6)				{$\Sigma_4$};
\node[left] at (-2.4, 1)					{$\Sigma_5^+$};
\node[left] at (-2.4,-1)					{$\Sigma_7^+$};
\node[left] at (-0.5, 1)					{$\Sigma_5^-$};
\node[left] at (-0.5,-1)					{$\Sigma_7^-$};
\node[right] at (0.5,1)					{$\Sigma_6^-$};
\node[right] at (0.5,-1)					{$\Sigma_8^-$};
\node[right] at (2.4, 1)					{$\Sigma_6^+$};
\node[right] at (2.4,-1)					{$\Sigma_8^+$};

\node[above] at (0, 1.5)				{$\Omega_1$};
\node[below] at (0, -1.5)				{$\Omega_2$};

\node[right] at (5, 0.5)				{$\Omega_3$};
\node[right] at (5, -0.5)				{$\Omega_7$};
\node[right] at (1.73, 0.5)				{$\Omega_8$};
\node[right] at (1.73, -0.5)				{$\Omega_4$};
\node[left] at (-1.73, 0.5)				{$\Omega_9$};
\node[left] at (-1.73, -0.5)				{$\Omega_5$};
\node[left] at (-5, 0.5)				{$\Omega_6$};
\node[left] at (-5, -0.5)				{$\Omega_{10}$};

\end{tikzpicture}

\end{figure}
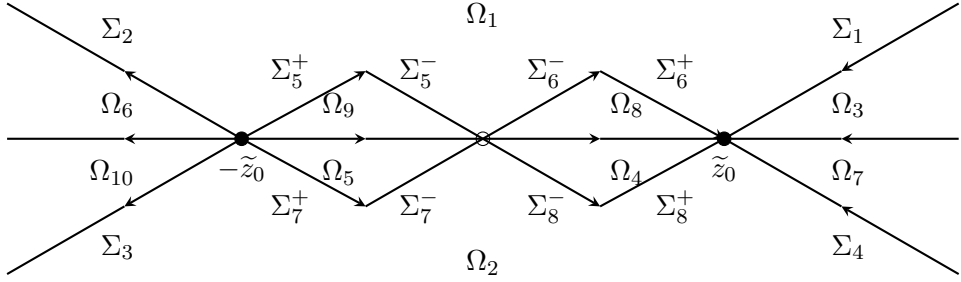
In order to separate the discrete spectral points from the $\bar{\partial}$ part, we define the following smooth cutoff function:
\begin{equation}
\Xi_{\mathcal{Z}}(\lambda)= \begin{cases}1 & \operatorname{dist}\left(\lambda, \mathcal{Z} \cup \mathcal{Z}^{*}\right) \leq \Upsilon / 3 \\ 0 & \operatorname{dist}\left(\lambda, \mathcal{Z} \cup \mathcal{Z}^{*}\right)>2 \Upsilon / 3\end{cases}
\end{equation}
where
\begin{equation}
\Upsilon=\min \left\{\min _{\lambda, \lambda^{\prime} \in \mathcal{Z}}\left|\lambda-\lambda^{\prime}\right|, \quad \operatorname{dist}(\mathcal{Z}, \mathbb{R})\right\}
\end{equation}
We now introduce another matrix-valued function $\widetilde{M}^{(2)}$:
\begin{equation}
\widetilde{M}^{(2)}(\lambda)=\widetilde{M}^{(1)}(\lambda)\mathcal{R}^{(2)}(\lambda).
\end{equation}
Here $\mathcal{R}^{(2)}$ will be chosen to remove the jump on the real axis and bring about new analytic jump matrices with the desired exponential decay along the contour $\Sigma^{(2)}$. Straight forward computation gives
\begin{equation}
\begin{aligned}
\widetilde{M}_{+}^{(2)} &=\widetilde{M}_{+}^{(1)} \mathcal{R}_{+}^{(2)} \\
&=\widetilde{M}_{-}^{(1)} \widetilde{V}^{(1)} \mathcal{R}_{+}^{(2)} \\
&=\widetilde{M}_{-}^{(2)}\left(\mathcal{R}_{-}^{(2)}\right)^{-1} \widetilde{V}^{(1)} \mathcal{R}_{+}^{(2)}
\end{aligned}
\end{equation}
We want to make sure that the following condition is satisfied
\begin{equation}
\left(\mathcal{R}_{-}^{(2)}\right)^{-1} \widetilde{V}^{(1)} \mathcal{R}_{+}^{(2)}=I
\end{equation}
where $\mathcal{R}_{\pm}^{(2)}$ are the boundary values of $\mathcal{R}^{(2)}(\lambda)$ as $\pm\mathrm{Im}(\lambda)\downarrow 0$. In this case the jump matrix associated to $\widetilde{M}^{(2)}_{\pm}$
will be the identity matrix on $\mathbb{R}$.\\
We can easily check that the function $e^{-i\widetilde{\theta}}$ is exponentially decreasing on $\Sigma_k$, $k=3,4,5,6$ and increasing on $\Sigma_j$, $j=1,2,7,8$ while the reverse is true for $e^{i\widetilde{\theta}}$. We define $\mathcal{R}^{2}$ as follows (Figure \ref{fig R-2+}-\ref{fig R-2-}):
 the functions $R_i$, $i=1,2...8$ satisfy
 \begin{align}
\label{R1}
R_1(\lambda)	&=	\begin{cases}
							\widetilde{r}(\lambda)\widetilde{\delta}^{-2}(\lambda)			
								&	\lambda \in (\tz_0,\infty)\\[10pt]
						\widetilde{r}(\widetilde{z}_0) \left( {\lambda-\widetilde{z}_0}\right)^{2i\kappa^+} (\widetilde{\delta}_{0}^+)^{-2}(1-\Xi_\calZ)
								&	\lambda	\in \Sigma_1,
					\end{cases}\\[10pt]
\label{R2}
R_2(\lambda)	&=	\begin{cases}
						\widetilde{r}(\lambda)\widetilde{\delta}^{-2}(\lambda)		
								&	\lambda \in (-\infty, -\tz_0)\\[10pt]
					\widetilde{r}(-\widetilde{z}_0) \left( -{\lambda-\widetilde{z}_0}\right)^{-2i\kappa^-} (\widetilde{\delta}_{0}^-)^{-2}(1-\Xi_\calZ)
								&	\lambda	\in \Sigma_2,
					\end{cases}\\[10pt]
\label{R3}
R_3(\lambda)	&=	\begin{cases}
						-\lambda\overline{\widetilde{r}(\lambda)}\widetilde{\delta}(\lambda)^2		
								&	\lambda \in (-\infty, -\tz_0)\\[10pt]
						\widetilde{z}_0\overline{\widetilde{r}(-\widetilde{z}_0)}\left( -{\lambda-\widetilde{z}_0}\right)^{2i\kappa^-} (\widetilde{\delta}_{0}^-)^{2}(1-\Xi_\calZ)
								&	\lambda	\in \Sigma_3,
					\end{cases}\\[10pt]
\label{R4}
R_4(\lambda)	&=	\begin{cases}
						-\lambda\overline{\widetilde{r}(\lambda)}\widetilde{\delta}(\lambda)^2		
								&	\lambda \in (\tz_0, \infty)\\[10pt]
						-\widetilde{z}_0\overline{\widetilde{r}(\widetilde{z}_0)}\left( {\lambda-\widetilde{z}_0}\right)^{-2i\kappa^+} (\widetilde{\delta}_{0}^+)^{2}(1-\Xi_\calZ)
								&	\lambda	\in \Sigma_4,
					\end{cases}
\end{align}
\begin{align}
\label{R5}
R_5(\lambda)	&=	\begin{cases}
						\frac {\lambda\overline{\widetilde{r}(\lambda)}\widetilde{\delta}^2_+}{1+\lambda|\widetilde{r}|^2}		
								& \lambda \in (-\widetilde{z}_0,0)\\[10pt]
					    \frac {-\widetilde{z}_0\overline{\widetilde{r}(-\widetilde{z}_0)}\left( -{\lambda-\widetilde{z}_0}\right)^{2i\kappa^-} (\widetilde{\delta}_{0}^-)^{2}}{1-\widetilde{z}_0|\widetilde{r}(-\widetilde{z}_0)|^2}(1-\Xi_\calZ)
						& \lambda \in \Sigma_5^+\\
						0 & \lambda \in \Sigma_5^-
					\end{cases}
					\\[10pt]
\label{R8+}
R_6(\lambda)	&=	\begin{cases}
				    \frac {\lambda\overline{\widetilde{r}(\lambda)}\widetilde{\delta}^2_+}{1+\lambda|\widetilde{r}|^2}		
								& \lambda \in (0, \widetilde{z}_0)\\[10pt]
					 \frac {\widetilde{z}_0\overline{\widetilde{r}(\widetilde{z}_0)}\left( {\lambda-\widetilde{z}_0}\right)^{-2i\kappa^+} (\widetilde{\delta}_{0}^+)^{2}}{1+\widetilde{z}_0|\widetilde{r}(\widetilde{z}_0)|^2}(1-\Xi_\calZ)
						& \lambda \in \Sigma_6^+\\
						0 & \lambda \in \Sigma_6^-
					\end{cases}
					\\[10pt]
\label{R7-}
R_{7}(\lambda)	&=	\begin{cases}
						-\frac {\widetilde{r}\widetilde{\delta}^{-2}_-}{(1+\lambda|\widetilde{r}|^2)}		
								& \lambda \in (-\widetilde{z}_0,0)\\[10pt]
					-\frac {\widetilde{r}(-\widetilde{z}_0)\left( -{\lambda-\widetilde{z}_0}\right)^{-2i\kappa^-} (\widetilde{\delta}_{0}^-)^{-2}  }{(1-\widetilde{z}_0|\widetilde{r}(-\widetilde{z}_0)|^2)}(1-\Xi_\calZ)
                                  & \lambda \in \Sigma_7^+\\
                                  0  & \lambda \in \Sigma_7^-
						\end{cases}
						\\[10pt]
\label{R8-}
R_8(\lambda)	&=	\begin{cases}
						-\frac {\widetilde{r}\widetilde{\delta}^{-2}_-}{(1+\lambda|\widetilde{r}|^2)}		
								& \lambda \in (0, \widetilde{z}_0)\\[10pt]
					-\frac {\widetilde{r}(\widetilde{z}_0)\left( {\lambda-\widetilde{z}_0}\right)^{2i\kappa^+} (\widetilde{\delta}_{0}^+)^{-2}}{(1+\widetilde{z}_0|\widetilde{r}(\widetilde{z}_0)|^2)}(1-\Xi_\calZ)
                                  & \lambda \in \Sigma_8^+\\
                                  0 &  \lambda \in \Sigma_8^-
                                  \end{cases}
\end{align}

{
\SixMatrix{The Matrix  $\calR^{(2)}$ for Region I, near $\widetilde{z}_0$}{fig R-2+}
	{\twomat{1}{R_1 e^{i\widetilde{\theta}}}{0}{1}}
	{\twomat{1}{0}{R_6 e^{-i\widetilde{\theta}}}{1}}
	{\twomat{1}{R_8 e^{i\widetilde{\theta}}}{0}{1}}
	{\twomat{1}{0}{R_4 e^{-i\widetilde{\theta}}}{1}}
}

{
\sixmatrix{The Matrix  $\calR^{(2)}$ for Region I, near $-\widetilde{z}_0$}{fig R-2-}
	{\twomat{1}{0}{R_5 e^{-i\widetilde{\theta}}}{1}}
	{\twomat{1}{R_2 e^{i\widetilde{\theta}}}{0}{1}}
	{\twomat{1}{0}{R_3 e^{-i\widetilde{\theta}}}{1}}
	{\twomat{1}{R_7 e^{i\widetilde{\theta}}} {0}{1}}
}
Each $R_i(z)$ in $\Omega_i$ is constructed in such a way that the jump matrices on the contour and $\dbar R_i(z)$ enjoys the property of exponential decay as $t\to \infty$.
We formulate Problem \ref{modifiedRHP} into a mixed RHP-$\dbar$ problem. In the following sections we will separate this mixed problem into a localized RHP and a pure $\dbar$ problem whose long-time contribution to the asymptotics of solution is  smaller than the leading term.\\
The following lemma (\cite[Proposition 2.1]{DM08}) will be used in the error estimates of
$\bar \partial$-problem in Section \ref{sec:dbar}.
We first denote the entries that appear in \eqref{R1}--\eqref{R8-} by
\begin{align*}
p_1(\lambda)=p_2(\lambda)	&=	\widetilde{r}(\lambda),	&
p_3(\lambda)=p_4(\lambda)	&=	-\lambda\overline{\widetilde{r}(\lambda)},&\\
p_{5}(\lambda)=p_{6}(\lambda)	&=\frac {\lambda\overline{\widetilde{r}(\lambda)}}{1+\lambda|\widetilde{r}|^2},& p_{7}(\lambda)=p_{8}(\lambda)	&= -\frac {\widetilde{r}(\lambda)}{1+\lambda|\widetilde{r}|^2}.
\end{align*}
For technical purpose, we further split the regions $\Omega_4$, $\Omega_5$, $\Omega_8$ and $\Omega_9$ in Figure \ref{fig:contour-2'}:
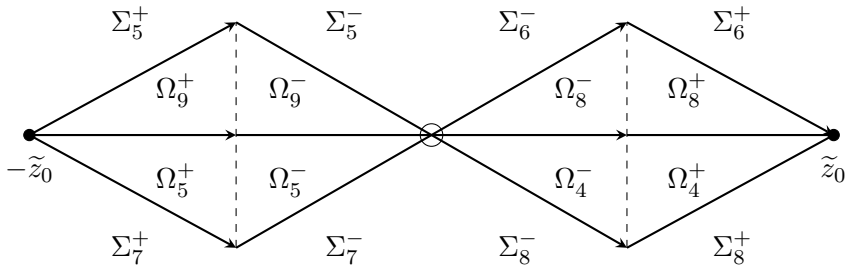
\begin{figure}[H]
\caption{Splitting the regions }
\vskip 15pt
\begin{tikzpicture}[scale=1.5]

\draw 	[->,thick,>=stealth] 	(-3.564, 0)-- (-1.732, 0);	
\draw 	[thick] 	 (-1.732, 0)--(0,0 );

\draw 	[->,thick,>=stealth] 	(0, 0)-- (1.732, 0);	
\draw 	[thick] 	 (1.732, 0)--(3.564,0 );

\draw[->,thick,>=stealth] 			(1.732 ,1) -- (3.564 ,0) ;
\draw[->,thick,>=stealth] 			(0,0)--(1.732 ,1);

\draw[thick] 			(1.732 ,-1) -- (3.564 ,0) ;
\draw[->,thick,>=stealth] 			(0,0)--(1.732 ,-1);

\draw[->,thick,>=stealth] 			 (-3.564 ,0)--(-1.732 ,-1) ;
\draw[thick] 			(-1.732 ,-1)--(0,0);
\draw[->,thick,>=stealth] 			 (-3.564 ,0)--(-1.732 ,1) ;
\draw[thick] 			(-1.732 , 1)--(0,0);

\draw	[fill]							(-3.564 ,0)		circle[radius=0.05];	
\draw	[fill]							(3.564, 0)		circle[radius=0.05];
\draw							(0,0)		circle[radius=0.1];

\node[below] at (-3.564,-0.1)			{$-\tz_0$};
\node[below] at (3.564,-0.1)			{$\tz_0$};

\node[left] at (-2.4, 1)					{$\Sigma_5^+$};
\node[left] at (-2.4,-1)					{$\Sigma_7^+$};
\node[left] at (-0.5, 1)					{$\Sigma_5^-$};
\node[left] at (-0.5,-1)					{$\Sigma_7^-$};
\node[right] at (0.5,1)					{$\Sigma_6^-$};
\node[right] at (0.5,-1)					{$\Sigma_8^-$};
\node[right] at (2.4, 1)					{$\Sigma_6^+$};
\node[right] at (2.4,-1)					{$\Sigma_8^+$};

\node[right] at (2, 0.4)				{$\Omega_8^+$};
\node[right] at (2, -0.4)				{$\Omega_4^+$};
\node[left] at (-2, 0.4)				{$\Omega_9^+$};
\node[left] at (-2, -0.4)				{$\Omega_5^+$};

\node[right] at (1, 0.4)				{$\Omega_8^-$};
\node[right] at (1, -0.4)				{$\Omega_4^-$};
\node[left] at (-1, 0.4)				{$\Omega_9^-$};
\node[left] at (-1, -0.4)				{$\Omega_5^-$};

\draw [dashed]  (-1.732, -1)--(-1.732, 1);
\draw [dashed]  (1.732, -1)--(1.732, 1);

\end{tikzpicture}
\label{fig:contour-2'}
\end{figure}
\begin{lemma}
\label{dbar.Ri}
Suppose $\widetilde{r} \in H^{1,1}_0(\bbR)$. There exist functions $R_i$ on $\Omega_3, \Omega_6, \Omega_7, \Omega_{10}$ and $\Omega_4^+, \Omega_5^+, \Omega_8^+, \Omega_9^+$ satisfying \eqref{R1}--\eqref{R8-}, so that
$$
|\dbar R_i(\lambda)| \lesssim
	 |p_i'(\Real(\lambda))| + |\lambda-\xi|^{-1/2}  +\dbar \left( \Xi_\calZ(\lambda) \right) , 	
				$$
where $\xi=\pm \widetilde{z}_0$. And for functions $R_i$ on $ \Omega_4^-, \Omega_5^+, \Omega_8^+, \Omega_9^+$  we have
$$
|\dbar R_i(\lambda)| \lesssim
	 |p_i'(\Real(\lambda))| + |\lambda|^{-1/2}  +\dbar \left( \Xi_\calZ(\lambda) \right) .
			$$
All the implied constants are uniform for $\widetilde{r} $ in a bounded subset of $H^{1,1}_0(\bbR)$.
\end{lemma}
\begin{proof}
We first prove the lemma for $R_1$. Define $f_1(\lambda)$ on $\Omega_3$ by
$$ f_1(\lambda) = p_1(\widetilde{z}_0) \left( {\lambda-\widetilde{z}_0}\right)^{2i\kappa^+} (\widetilde{\delta}_{0}^+)^{-2} \widetilde{\delta}(\lambda)^{2} $$
and let
\begin{equation}
\label{interpol}
\ R_1(\lambda) = \left( f_1(\lambda) + \left[ p_1(\Real(\lambda)) - f_1(\lambda) \right] \mathcal{K}(\phi) \right) \widetilde{\delta}(\lambda)^{-2}
(1-\Xi_\calZ)
\end{equation}
where $\phi = \arg (\lambda-\xi)$ and $\mathcal{K}$ is a smooth function on $(0, \pi/6)$ with
$$
\mathcal{K}(\phi)=
	\begin{cases}
			1
			&	\phi \in [0, \pi/24], \\
			0
			&	\phi \in[\pi/12, \pi/6]
	\end{cases}
$$
 It is easy to see that $R_1$ as constructed has the boundary values \eqref{R1}.
Writing $\lambda-\tz_0= \rho e^{i\phi}$, we have
$$ \dbar = \frac{1}{2}\left( \frac{\dee}{\dee x} + i \frac{\dee}{\dee y} \right)
			=	\frac{1}{2} e^{i\phi} \left( \frac{\dee}{\dee \rho} + \frac{i}{\rho} \frac{\dee}{\dee \phi} \right).
$$
We calculate
\begin{align*}
\dbar R_1 (\lambda) & = \left( \frac{1}{2}  p_1'(\Real \lambda) \mathcal{K}(\phi)  ~ \widetilde{\delta}(\lambda)^{-2} -
		\left[ p_1(\Real \lambda) - f_1(\lambda) \right]\widetilde{\delta}(\lambda)^{-2}  \frac{ie^{i\phi}}{|\lambda-\xi|}  \mathcal{K}'(\phi) \right)\\
		 & \quad \times \left( 1-\Xi_\calZ  \right) e^{i\widetilde{\theta}}-\left( f_1(\lambda) + \left[ p_1(\Real(\lambda)) - f_1(\lambda) \right] \mathcal{K}(\phi) \right) \widetilde{\delta}(\lambda)^{-2} \dbar \left( \Xi_\calZ(\lambda) \right) e^{i\widetilde{\theta}} .
\end{align*}
Given that $\Xi(\lambda)$ is infinitely smooth and compactly supported, it follows from Lemma \ref{delta} (4) that
$$
 \left|\left( \dbar R_1 \right)(\lambda)  \right| \lesssim
\left[  |p_1'(\Real \lambda)| + |\lambda-\tz_0|^{-1/2} +\dbar \left( \Xi_\calZ(\lambda) \right) \right] |e^{i\widetilde{\theta}}|
$$
where the implied constants depend on $\norm{r}{H^{1,1}}$ and the smooth function $\mathcal{K}$. We then move to region $\Omega_8^-$.
We set
$$ f_8(\lambda) = p_8(0)\left( {\lambda-\widetilde{z}_0}\right)^{-2i\kappa^+} (\widetilde{\delta}_{0}^+)^{2}\widetilde{\delta}(\lambda)^{2}=0$$
and calculate
\begin{align*}
\dbar R_8 (\lambda) &=\dbar  \left( p_8(\Real(\lambda) ) \mathcal{K}(\phi)  \widetilde{\delta}(\lambda)^{-2} (1-\Xi_\calZ)  \right) \\
                     & = \left( \frac{1}{2}  p_8'(\Real (\lambda)) \mathcal{K}(\phi) \widetilde{\delta}(\lambda)^{-2} -
		p_8(\Real(\lambda)) \widetilde{\delta}(\lambda)^{-2}  \frac{ie^{i\phi}}{| \lambda |}  \mathcal{K}'(\phi) \right) \left( 1-\Xi_\calZ  \right)\\
		 &\quad - p_8(\Real(\lambda))  \mathcal{K}(\phi)  \widetilde{\delta}(\lambda)^{-2} \dbar \left( \Xi_\calZ(\lambda) \right) .
\end{align*}
The estimates in the remaining regions are analogous.
\end{proof}
\begin{remark}
We note that the interpolation defined through \eqref{interpol} introduces a new jump on $\Sigma^{(3)}_9$ with  the jump matrix given by
\begin{equation}
\label{jump v9}
v_9=\begin{cases} I, & \lambda \in \left(-i ( \widetilde{z}_0/2) \tan(\pi/24) , i ( \widetilde{z}_0/2) \tan(\pi/24) \right)\\
\\
  \unitlower{ (R_5^+-R_5^-)e^{-i\widetilde{\theta}} },  & \lambda\in \left(i( \widetilde{z}_0/2) \tan(\pi/24)-\widetilde{z}_0/2, i\widetilde{z}_0/(2\sqrt{3} ) -\widetilde{z}_0/2\right) \\
        \\
          \unitlower{ (R_6^--R_6^+)e^{-i\widetilde{\theta}} },  & \lambda\in \left(i ( \widetilde{z}_0/2)  \tan(\pi/24)+\widetilde{z}_0/2,  i\widetilde{z}_0/(2\sqrt{3} )+\widetilde{z}_0/2 \right) \\
        \\
        \unitupper{ (R_7^+ -R_7^-)e^{i\widetilde{\theta}} }  & \lambda\in  \left( -i( \widetilde{z}_0/2)  \tan(\pi/24)-\widetilde{z}_0/2, - i\widetilde{z}_0/(2\sqrt{3} )-\widetilde{z}_0/2 \right) \\
        \\
         \unitupper{ (R_8^--R_8^+)e^{i\widetilde{\theta}} },  & \lambda\in  \left ( -i( \widetilde{z}_0/2) \tan(\pi/24)+\widetilde{z}_0/2, -i\widetilde{z}_0/(2\sqrt{3} )+\widetilde{z}_0/2 \right)
\end{cases}
\end{equation}
\end{remark}

\begin{figure}[H]
\caption{$\Sigma'^{(2)}$}
\vskip 15pt
\begin{tikzpicture}[scale=0.8]

\draw[thick] 			(3.564 ,0) -- (5.296, 1);								
\draw[->,thick,>=stealth] 	(   7.028 ,2  )--	(5.296, 1);

\draw[->,thick,>=stealth] 			(1.732 ,1) -- (3.564 ,0) ;
\draw[->,thick,>=stealth] 			(0,0)--(1.732 ,1);

\draw[thick] 			(3.564 ,0) -- (5.296, -1);								
\draw[->,thick,>=stealth] 	(   7.028 , -2  )--	(5.296, -1);

\draw[thick] 			(1.732 ,-1) -- (3.564 ,0) ;
\draw[->,thick,>=stealth] 			(0,0)--(1.732 ,-1);

\draw[->,thick,>=stealth] 			 (-3.564 ,0)--(-1.732 ,-1) ;
\draw[thick] 			(-1.732 ,-1)--(0,0);
\draw[->,thick,>=stealth] 			 (-3.564 ,0)--(-1.732 ,1) ;
\draw[thick] 			(-1.732 , 1)--(0,0);

\draw[->,thick,>=stealth] 			(-3.564 ,0) -- (-5.296, 1);							
\draw[thick] 	(   -7.028 ,2  )--	(-5.296, 1);
\draw[->,thick,>=stealth] 			(-3.564 ,0) -- (-5.296, -1);							
\draw[thick] 	(   -7.028 , -2  )--	(-5.296, -1);

\draw[->,thick,>=stealth] 			(-1.732 , -1) -- (-1.732, 0);	
\draw	[thick]	(-1.732 , 0) -- (-1.732, 1);	
\draw[->,thick,>=stealth] 			(1.732 , -1) -- (1.732, 0);	
\draw	[thick]	(1.732 , 0) -- (1.732, 1);

\draw	[fill]							(-3.564 ,0)		circle[radius=0.1];	
\draw	[fill]							(3.564, 0)		circle[radius=0.1];
\draw							(0,0)		circle[radius=0.1];

\node[below] at (-3.564,-0.1)			{$-\tz_0$};
\node[below] at (3.564,-0.1)			{$\tz_0$};

\node[right] at (5, 1.6)					{$\Sigma_1'^{(2)} $};
\node[left] at (-5, 1.6)					{$\Sigma_2'^{(2)}$};
\node[left] at (-5,-1.6)					{$\Sigma_3'^{(2)}$};
\node[right] at (5,-1.6)				{$\Sigma_4'^{(2)}$};
\node[left] at (-1,1.4)					{$\Sigma_5'^{(2)}$};
\node[left] at (-1,-1.4)					{$\Sigma_7'^{(2)}$};
\node[right] at (1,1.4)					{$\Sigma_6'^{(2)}$};
\node[right] at (1,-1.4)					{$\Sigma_8'^{(2)}$};
\node[right] at (-1.732, 0) {$\Sigma_9'^{(2)} $};
\node[right] at (1.732, 0) {$\Sigma_{10}'^{(2)} $};

\end{tikzpicture}
\label{fig:contour-2}
\end{figure}
The unknown $\widetilde{M}^{(2)}$ satisfies a mixed $\dbar$-RHP. We first identify the jumps of $\widetilde{M}^{(2)}$ along the contour $\Sigma^{(2)}$. Recall that $\widetilde{M}^{(1)}$ is analytic along the contour,  the jumps are determined entirely by
$\mathcal{R}^{(2)}$, see \eqref{R1}--\eqref{R8-}. Away from $\Sigma^{(2)}$, using the triangularity of $\mathcal{R}^{(2)}$, we have that
\begin{equation}
\label{N2.dbar}
 \dbar \widetilde{M}^{(2)} = \widetilde{M}^{(2)} \left( \calR^{(2)} \right)^{-1} \dbar \calR^{(2)} = \widetilde{M}^{(2)} \dbar \calR^{(2)}.
 \end{equation}
 \begin{remark}
 \label{radius}
By construction of $\mathcal{R}^{(2)}$ (see \eqref{R1}-\eqref{R8-} and \eqref{interpol}) and the choice of the radius of the circles in the set $ \Gamma$, the right multiplication of $\mathcal{R}^{(2)}$ to $\widetilde{M}^{(1)}$ will not change the jump conditions on circles in the set $ \Gamma$. Thus over circles in the set $ \Gamma$, $\widetilde{M}^{(2)}$ has the same jump matrices as given by (4) of Problem \ref{modifiedRHP}.
\end{remark}

\begin{problem}
\label{prob:MTM.RHP.dbar}
Given $\widetilde{r} \in H^{1,1}_0(\bbR)$, find a matrix-valued function $\widetilde{M}^{(2)}(\lambda;x,t)$ on $ \bbC \setminus \left( \Sigma'^{(2)} \cup \Gamma \right)$ with the following properties:
\begin{enumerate}
\item		$\widetilde{M}^{(2)}(\lambda;x,t) \rarr I$ as $ \lambda \rarr \infty$ in $ \bbC \setminus \left( \Sigma'^{(2)} \cup \Gamma \right);$
\item		$\widetilde{M}^{(2)}(\lambda;x,t)$ is continuous for $\lambda \in  \bbC \setminus\left( \Sigma'^{(2)} \cup \Gamma \right)$
			with continuous boundary values
			$\widetilde{M}^{(2)}_\pm(\lambda;x,t) $
			(where $\pm$ is defined by the orientation in Figure \ref{fig:contour-def});
\item		The jump relation $\widetilde{M}^{(2)}_+(\lambda;x,t)=\widetilde{M}^{(2)}_-(\lambda;x,t)	
			\widetilde{V}^{(2)}(\lambda)$ holds, where
			$\widetilde{V}^{(2)}(\gamma)$ is given in Figure \ref{fig:jumps-1}-\ref{fig:jumps-2} and \eqref{jump v9} and part (4) of Problem \ref{modifiedRHP}.
\item		The equation
			$$
			\dbar \widetilde{M}^{(2)} = \widetilde{M}^{(2)} \, \dbar \calR^{(2)}
			$$
			holds in $\bbC \setminus \left( \Sigma'^{(2)} \cup \Gamma \right)$, where
			$$
			\dbar \calR^{(2)}=
				\begin{doublecases}
					\Twomat{0}{(\dbar R_1) e^{i\widetilde{\theta}}}{0}{0}, 	& \lambda \in \Omega_3	&&
					\Twomat{0}{0}{(\dbar R_5)e^{-i\widetilde{\theta}}}{0}	,	& \lambda \in \Omega_9	\\
					\\
					\Twomat{0}{(\dbar R_2)e^{i\widetilde{\theta}}}{0}{0},	&	\lambda \in \Omega_6	&&
					\Twomat{0}{0}{(\dbar R_6)e^{-i\widetilde{\theta}}}{0}	,	&	\lambda	\in \Omega_8	 \\
					\\
					\Twomat{0}{0}{(\dbar R_3) e^{-i\widetilde{\theta}}}{0}, 	& \lambda \in \Omega_{10}	&&
					\Twomat{0}{(\dbar R_7)e^{i\widetilde{\theta}}} {0}  {0}	,	& \lambda \in \Omega_5	\\
					\\
					\Twomat{0} {0} {(\dbar R_4)e^{-i\widetilde{\theta}}}{0},	&	\lambda \in \Omega_7	&&
					\Twomat{0}{(\dbar R_8)e^{i\widetilde{\theta}}}{0}{0}	,	&	\lambda	\in \Omega_4 \\
					\\
					0	&\hspace{-5pt} \lambda\in \Omega_1\cup\Omega_2	
				\end{doublecases}
			$$
\end{enumerate}
\end{problem}
The following picture is an illustration of the jump matrices of RHP Problem \ref{prob:MTM.RHP.dbar}. For brevity we ignore the discrete scattering data.

\begin{figure}[H]
\caption{Jump Matrices  $\widetilde{V}^{(2)}$  for $\widetilde{M}^{(2)}$ near $\widetilde{z}_0$ }
\vskip 15pt
\begin{tikzpicture}[scale=0.7]
\draw[dashed] 				(-6,0) -- (6,0);							
\draw [thick] 	(0,0 )-- (1.732, 1);						
\draw [->,thick,>=stealth]   (3.464,2 ) -- (1.732, 1);
\draw [thick] 	(0,0 )-- (1.732, -1);						
\draw  [->,thick,>=stealth]  (3.464,-2 ) -- (1.732, -1);
\draw [thick] 	(0,0 )-- (-1.732, 1);						
\draw [->,thick,>=stealth] (-3.464,2 ) -- (-1.732, 1);
\draw [thick] 	(0,0 )-- (-1.732, -1);						
\draw [->,thick,>=stealth]  (-3.464, -2 ) -- (-1.732, -1);
\draw[fill]						(0,0)	circle[radius=0.075];		
\node [below] at  			(0,-0.15)		{$\widetilde{z}_0$};
\node[right] at					(3.2,3)		{$\unitupper{R_1 e^{i\widetilde{\theta}}}$};
\node[left] at					(-3.2,3)		{$\unitlower{-R_6 e^{-i\widetilde{\theta}}}$};
\node[left] at					(-3.2,-3)		{$\unitupper{R_8 e^{i\widetilde{\theta}}}$};
\node[right] at					(3.2,-3)		{$\unitlower{-R_4 e^{-i\widetilde{\theta}}}$};
\node[left] at					(2.5, 1.8)		{$\Sigma_1$};
\node[right] at					(-2.5,1.8)		{$\Sigma_6$};
\node[right] at					(-2.5,-1.8)		{$\Sigma_8$};
\node[left] at					(2.5,-1.8)		{$\Sigma_4$};
\end{tikzpicture}
\label{fig:jumps-1}
\end{figure}

\begin{figure}[H]
\caption{Jump Matrices  $\widetilde{V}^{(2)}$  for $\widetilde{M}^{(2)}$ near $-\widetilde{z}_0$}
\vskip 15pt
\begin{tikzpicture}[scale=0.7]
\draw[dashed] 				(-6,0) -- (6,0);							
\draw [->,thick,>=stealth] 	(0,0 )-- (1.732, 1);						
\draw  [thick]  (3.464,2 ) -- (1.732, 1);
\draw [->,thick,>=stealth] 	(0,0 )-- (1.732, -1);						
\draw  [thick]  (3.464,-2 ) -- (1.732, -1);
\draw [->,thick,>=stealth] 	(0,0 )-- (-1.732, 1);						
\draw  [thick]  (-3.464,2 ) -- (-1.732, 1);
\draw [->,thick,>=stealth] 	(0,0 )-- (-1.732, -1);						
\draw  [thick]  (-3.464, -2 ) -- (-1.732, -1);
\draw[fill]						(0,0)	circle[radius=0.075];		
\node [below] at  			(0,-0.15)		{$-\widetilde{z}_0$};
\node[right] at					(3.2,3)		{$\unitlower{-R_5 e^{-i\widetilde{\theta}}}$};
\node[left] at					(-3.2,3)		{$\unitupper{R_2 e^{i\widetilde{\theta}}}$};
\node[left] at					(-3.2,-3)		{$\unitlower{-R_3 e^{-i\widetilde{\theta}}}$};
\node[right] at					(3.2,-3)		{$\unitupper{R_7 e^{i\widetilde{\theta}}}$};
\node[left] at					(2.5, 1.8)		{$\Sigma_5$};
\node[right] at					(-2.5, 1.8)		{$\Sigma_2$};
\node[right] at					(-2.5,-1.8)		{$\Sigma_3$};
\node[left] at					(2.5,-1.8)		{$\Sigma_7$};
\end{tikzpicture}
\label{fig:jumps-2}
\end{figure}
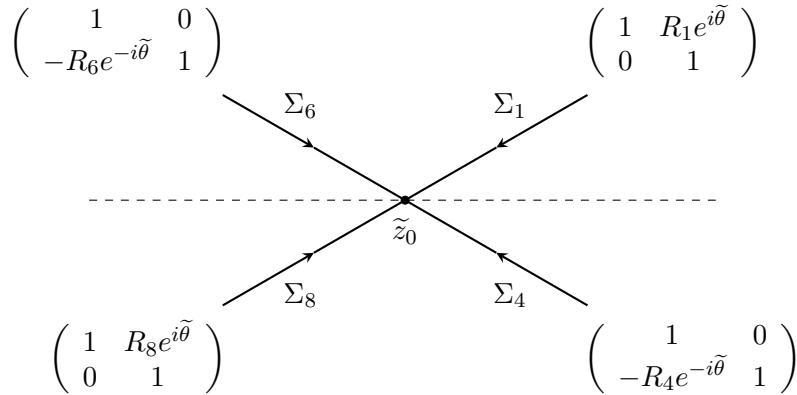
\section{The Localized Riemann-Hilbert Problem}
\label{sec:local}
We perform the following factorization of $\widetilde{M}^{(2)}$
\begin{equation}
\label{factor-LC}
\widetilde{M}^{(2)} = \widetilde{M}^{(3)} \widetilde{M}^{\RHP}.
\end{equation}
Here we require that $\widetilde{M}^{(3)} $ to be the solution of the pure $\dbar$-problem and it does not have any jumps, and $ \widetilde{M}^{\RHP}$ solution of the localized  RHP Problem \ref{MTM.RHP.local} below
with the jump matrix $\widetilde{V}^\RHP=\widetilde{V}^{(2)}$. The current section focuses on finding $ \widetilde{M}^{\RHP}$.
\begin{problem}
\label{MTM.RHP.local}
Find a $2\times 2$ matrix-valued function $\widetilde{M}^\RHP(\lambda; x,t)$, analytic on $\bbC \setminus (\Sigma'^{(2)}\cup\Gamma)$,
with the following properties:
\begin{enumerate}
\item	$\widetilde{M}^\RHP(\lambda;x,t) \rarr I$ as $|\lambda| \rarr \infty$ in $\bbC \setminus (\Sigma'^{(2)}\cup\Gamma)$, where $I$ is the $2\times 2$ identity matrix;
\item	$\widetilde{M}^\RHP(\lambda; x,t)$ is analytic for $\lambda \in \bbC \setminus  (\Sigma'^{(2)}\cup\Gamma)$ with continuous boundary values $\widetilde{M}^\RHP_\pm$
		on $\Sigma^{(3)}\cup\Gamma $;
\item	The jump relation $\widetilde{M}^\RHP_+(\lambda;x,t) = \widetilde{M}^\RHP_-(\lambda; x,t ) \widetilde{V}^\RHP(\lambda)$ holds on $\Sigma'^{(2)}\cup\Gamma $, where
		\begin{equation*}	
		\widetilde{V}^\RHP(\lambda) =	\widetilde{V}^{(2)}(\lambda).
		\end{equation*}
\end{enumerate}
\end{problem}

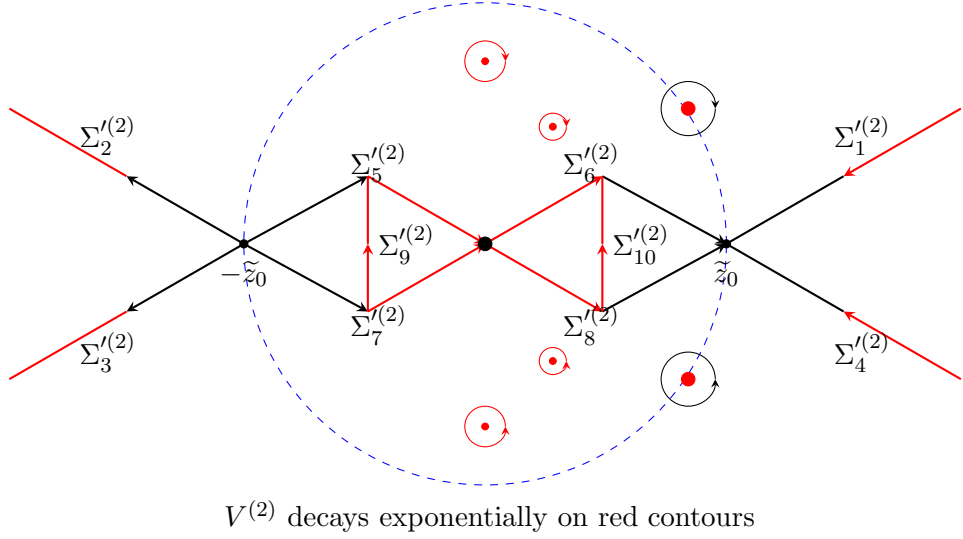
\begin{figure}[H]
\caption{ $\Sigma'^{(2)}\cup\Gamma$ }

\vskip 15pt
\begin{tikzpicture}[scale=0.9]

\draw[thick] 		(3.564 ,0) -- (5.296, 1);								
\draw[->,thick,>=stealth] [red]	(   7.028 ,2  )--	(5.296, 1);

\draw[->,thick,>=stealth] 			(1.732 ,1) -- (3.564 ,0) ;
\draw[->,thick,>=stealth] 	 [red]			(0,0)--(1.732 ,1);

\draw[thick] 			(3.564 ,0) -- (5.296, -1);								
\draw[->,thick,>=stealth]  [red]		(   7.028 , -2  )--	(5.296, -1);

\draw[->,thick,>=stealth] 			(1.732 ,-1) -- (3.564 ,0) ;
\draw[->,thick,>=stealth] 		 [red]		(0,0)--(1.732 ,-1);

\draw[->,thick,>=stealth] 			 (-3.564 ,0)--(-1.732 ,-1) ;
\draw[->,thick,>=stealth] 	 [red]			(-1.732 ,-1)--(0,0);
\draw[->,thick,>=stealth] 			 (-3.564 ,0)--(-1.732 ,1) ;
\draw[->,thick,>=stealth] 		 [red]		(-1.732 , 1)--(0,0);

\draw[->,thick,>=stealth] 			(-3.564 ,0) -- (-5.296, 1);							
\draw[thick]  [red]		(   -7.028 ,2  )--	(-5.296, 1);
\draw[->,thick,>=stealth] 			(-3.564 ,0) -- (-5.296, -1);							
\draw[thick]  [red]		(   -7.028 , -2  )--	(-5.296, -1);

\draw[->,thick,>=stealth]  [red]				(-1.732 , -1) -- (-1.732, 0);	
\draw	[thick]  [red]		(-1.732 , 0) -- (-1.732, 1);	
\draw[->,thick,>=stealth] 	 [red]			(1.732 , -1) -- (1.732, 0);	
\draw	[thick]	 [red]	 (1.732 , 0) -- (1.732, 1);

\draw [blue][dashed] (0,0) circle[radius=3.564];

\draw	[fill]							(-3.564 ,0)		circle[radius=0.06];	
\draw	[fill]							(3.564, 0)		circle[radius=0.06];
\draw		[fill]					(0,0)		circle[radius=0.1];

\node[below] at (-3.564,-0.1)			{$-\widetilde{z}_0$};
\node[below] at (3.564,-0.1)			{$\widetilde{z}_0$};

\node[right] at (5, 1.6)					{$\Sigma_1'^{(2)} $};
\node[left] at (-5, 1.6)					{$\Sigma_2'^{(2)}$};
\node[left] at (-5,-1.6)					{$\Sigma_3'^{(2)}$};
\node[right] at (5,-1.6)				{$\Sigma_4'^{(2)}$};
\node[left] at (-1,1.2)					{$\Sigma_5'^{(2)}$};
\node[left] at (-1,-1.2)					{$\Sigma_7'^{(2)}$};
\node[right] at (1,1.2)					{$\Sigma_6'^{(2)}$};
\node[right] at (1,-1.2)					{$\Sigma_8'^{(2)}$};
\node[right] at (-1.732, 0) {$\Sigma_9'^{(2)} $};
\node[right] at (1.732, 0) {$\Sigma_{10}'^{(2 )} $};

 \draw [red, fill=red] (1, 1.732) circle [radius=0.05];
\draw[->,>=stealth] [red] (1.2, 1.732) arc(360:0:0.2);
\draw [red, fill=red] (1,-1.732) circle [radius=0.05];
\draw[->,>=stealth] [red] (1.2, -1.732) arc(0:360:0.2);
\draw [red, fill=red]  (0, 2.7) circle [radius=0.05];
\draw[->,>=stealth] [red] (0.3, 2.7) arc(360:0:0.3);
\draw [red, fill=red] (0, -2.7) circle [radius=0.05];
\draw[->,>=stealth] [red]  (0.3, -2.7) arc(0:360:0.3);
\draw[->,>=stealth] (3.4,2) arc(360:0:0.4);
\draw[->,>=stealth] (3.4,-2) arc(0:360:0.4);
\draw [red, fill=red] (3,2) circle [radius=0.1];
\draw [red, fill=red] (3,-2) circle [radius=0.1];
\end{tikzpicture}
\label{fig:contour-d}
\begin{center}
  \begin{tabular}{ccc}
$V^{(2)}$ decays exponentially on red contours
\end{tabular}
 \end{center}
\end{figure}

We will eventually construct a classical solution to problem \ref{MTM.RHP.local}. We make the following observations.  On figure \ref{fig:contour-d}, for some fixed $\eps>0$, we define
\begin{align*}
L_\eps &=\lbrace \lambda: \lambda=u \widetilde{z}_0 e^{i\pi/6},  0 \leq u\leq 1/ \sqrt{3} \rbrace\\
           &\cup   \lbrace \lambda: \lambda=u \widetilde{z}_0 e^{i 5\pi/6},  0\leq u\leq  1/ \sqrt{3} \rbrace \\
           & \cup \lbrace \lambda: \lambda= \widetilde{z}_0+ u \widetilde{z}_0 e^{i\pi/6}, \eps\leq u <+\infty\rbrace\\
             &\cup   \lbrace \lambda: \lambda=-\widetilde{z}_0+u \widetilde{z}_0 e^{5i\pi/6},  \eps\leq u < +\infty \rbrace \\
           \Sigma'&= \left(\Sigma^{(3)}\setminus (L_\eps\cup {L_\eps^*}\cup \Sigma^{(3)}_9 \cup\Sigma^{(3)}_{10} )\right) \cup \left( \pm \lambda_\ell  \right)  \cup \left( \pm \lambda^*_\ell\right) .
\end{align*}
\begin{figure}[H]
\caption{ $\Sigma'$}
\vskip 15pt
\begin{tikzpicture}[scale=1.2]

\draw[thick]		(2,0) -- (2.866, 0.5);								

\draw[thick] 		(-2,0) -- (-2.866, 0.5);					

\draw[thick] 		(-2,0) -- (-2.866, -0.5);					

\draw[thick]		(2,0) -- (2.866,-0.5);								

\draw[thick] 	(-1.134, 0.5)--(-2, 0) ; 
\draw[thick] 	(-1.134, -0.5)--(-2, 0) ;

\draw[thick]	(2,0) -- (1.134, 0.5);
\draw[thick]	(2,0) -- (1.134, -0.5);

\draw	[fill]							(-2,0)		circle[radius=0.1];	
\draw	[fill]							(2,0)		circle[radius=0.1];
\node[below] at (-2,-0.1)			{$-\widetilde{z}_0$};
\node[below] at (2,-0.1)			{$\widetilde{z}_0$};

\draw [blue] [dashed] (0,0) circle[radius=2];	

\draw [red, fill=red] (1, 1.732) circle [radius=0.07];
\draw[->,>=stealth]  (1.2, 1.732) arc(360:0:0.2);

\draw [red, fill=red] (1,-1.732) circle [radius=0.07];
\draw[->,>=stealth] (1.2, -1.732) arc(0:360:0.2);

  \node[above]  at (0, 0) {\footnotesize $\text{Re}(i\widetilde{\theta})>0$};

      \node[below]  at (0, 0) {\footnotesize $\text{Re}(i\widetilde{\theta})<0$};

     \node[above]  at (0, 2.5) {\footnotesize $\text{Re}(i\widetilde{\theta})<0$};
     \node[below]  at (0, -2.5) {\footnotesize $\text{Re}(i\widetilde{\theta})>0$};
     \draw[dashed ] (0,0) -- (-3,0);
\draw[ dashed] (-3,0) -- (-5,0);
\draw[->,>=stealth] [dashed](0,0) -- (3,0);
 \draw[ dashed] (3,0) -- (5,0);
 \node[right] at (5,0) {$\bbR$};
\end{tikzpicture}
\label{fig:sigma'}
\end{figure}
Here $\Sigma'$ is the black portion of the contour $\Sigma'^{(2)}\cup\Gamma $ given in Figure \ref{fig:contour-d}. Our goal is to build explicitly solvable models out of this contour. Now we decompose $\widetilde{W}^{ ( 2 )} = \widetilde{V}^{ ( 2 )} -I $ into two parts:
\begin{equation}
\widetilde{W}^{ ( 2 )}=\widetilde{W}^e+\widetilde{W}'
\end{equation}
where $\widetilde{W}'=\widetilde{W}^{ ( 2 )}\restriction_{  \Sigma' }$ and $\widetilde{W}^e=\widetilde{W}^{ ( 2 )}\restriction_{   \left( \Sigma^{(2)}\cup\Gamma \right) \setminus\Sigma'   }$.

Near $\pm \widetilde{z}_0$, we write
$$\text{Re} i\widetilde{\theta}(\lambda; x, t)=t(\text{Im}\lambda)\left(\frac 1{1+\widetilde{z}^2_0}\right) \left(   \dfrac{\widetilde{z}_0^2}{(\text{Re}\lambda )^2 + (\text{Im}\lambda )^2 } -1\right)$$
and set
\begin{equation}
\label{tau}
\tau=\dfrac{t \widetilde{z}_0}{1+\widetilde{z}_0^2}.
\end{equation}
On $L_\eps$, away from $\pm \tz_0$, for $i=1,  2, 7, 8$ we estimate:
\begin{equation}
\label{Ri-decay}
\left\vert R_i e^{i\ttheta} \right\vert \leq C_r e^{-C \eps \tau},
\end{equation}
Similarly, on ${L_\eps^*}$  for $j=3,  4, 5, 6$
\begin{equation}
\label{Rj-decay}
\left\vert R_j e^{-i\ttheta} \right\vert  \leq C_r e^{-C \eps \tau}.
\end{equation}

Due to the construction of $\mathcal{K}(\phi)$ and $v_9$
on $\Sigma'^{(2)}_9$ and $\Sigma^{(2)}_{10}$, one obtains
\begin{equation}
\label{R9-decay}
\left\vert v_9 -I \right\vert \lesssim e^{-c\tau}.
\end{equation}
This together with the results above imply
\begin{equation}
\label{expo}
|w^e|\lesssim e^{-c\tau}
\end{equation}

\subsection{Construction of local parametrices}
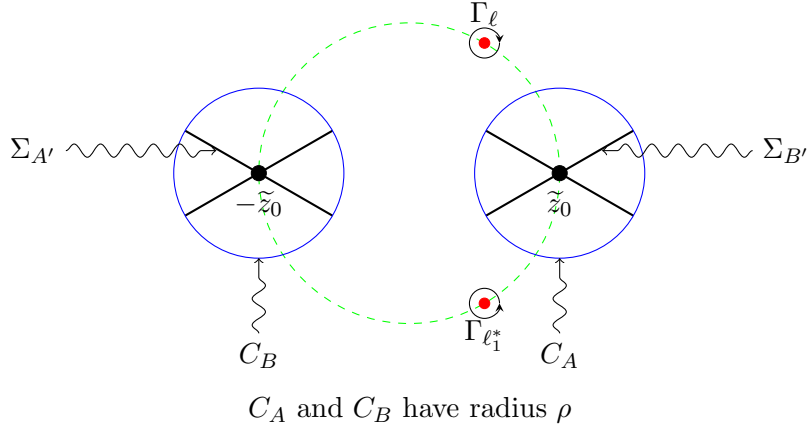
\begin{figure}[H]
\caption{ $\Sigma'=\Sigma_{A'}\cup \Sigma_{B'}\cup \Gamma_\ell \cup \Gamma^*_\ell$}
\vskip 15pt
\begin{tikzpicture}[ photon/.style={decorate,decoration={snake,post length=0.8mm}} ]

\draw [green] [dashed] (0,0) circle [radius=2];

\draw  [thick]   (2.9794, 0.5655)-- (1.02, -0.5655);
\draw  [thick]   (2.9794, -0.5655)-- (1.02, 0.5655);
\draw  [thick]   (-2.9794, 0.5655)-- (-1.02, -0.5655);
\draw  [thick]   (-2.9794, -0.5655)-- (-1.02, 0.5655);
\draw	[fill]						(-2,0)		circle[radius=0.1];	
\draw	[fill]					(2,0)		circle[radius=0.1];
\draw		[blue]				(-2,0)		circle[radius=1.131];	
\draw		[blue]					(2,0)		circle[radius=1.131];
\node[below] at (-2,-0.1)			{$-\widetilde{z}_0$};
\node[below] at (2,-0.1)			{$\widetilde{z}_0$};
\node[left] at (-4.56, 0.3)					{$\Sigma_{A'}$};
\node[right] at (4.56, 0.3)				{$\Sigma_{B'}$};

\draw [red, fill=red] (1, 1.732) circle [radius=0.07];
\draw[->,>=stealth]  (1.2, 1.732) arc(360:0:0.2);
\draw [red, fill=red] (1,-1.732) circle [radius=0.07];
\draw[->,>=stealth] (1.2, -1.732) arc(0:360:0.2);

\draw[->,photon] ( 4.565,0.3) --  (2.565, 0.3);
\draw[->,photon] ( -4.565,0.3) --  (-2.565, 0.3);
\draw[->,photon] (-2, -2.131) --  (-2, -1.131);
\draw[->,photon] (2, -2.131) --  (2, -1.131);
\node [below] at  (-2, -2.131)  {$C_B$};
\node [below] at  (2, -2.131) {$C_A$};
\node [above] at (1, 1.82) {$\Gamma_{\ell}$};
\node [below] at (1, -1.82)  {$\Gamma_{\ell^*_1}$};
\end{tikzpicture}
\begin{center}
  \begin{tabular}{ccc}
$C_A$ and $C_B$ have radius $\rho$
\end{tabular}
\end{center}
\label{fig:contour-2}
\end{figure}
In this subsection we construct some local parametrices which are needed to obtain $\widetilde{M}^\RHP$. To achieve this, we build the solutions of the following three exactly solvable RHPs:
\begin{problem}
\label{prob:MTM.sol}
Find a matrix-valued function $\widetilde{M}^{(sol)}(\lambda;x,t)$ on $\bbC \setminus\Sigma$ with the following properties:
\begin{enumerate}
\item		$\widetilde{M}^{(sol)}(\lambda;x,t) \rarr I$ as $|\lambda| \rarr \infty$;
\item		$\widetilde{M}^{(sol)}(\lambda;x,t)$ is analytic for $\lambda \in  \bbC \setminus ( \gamma_\ell \cup \gamma^*_\ell ) $
			with continuous boundary values
			$\widetilde{M}^{(sol)}_\pm(\lambda;x,t)$;
\item On $ \lambda_\ell \cup \lambda^*_\ell$, let $\widetilde{\delta}(\lambda)$ be the solution to Problem \ref{delta} and we have the following residue conditions
\begin{equation}
\Res_{\lambda=\lambda_{j}}\widetilde{M}^{(sol)}(x,t;\lambda)=\lim_{\lambda \rightarrow \lambda_{j}}\widetilde{M}^{(sol)}\left(\begin{array}{ccc}
0 & \widetilde{\delta}^{-2}(\lambda)\widetilde{c}_j e^{i\widetilde{\theta}} & \\
0 & 0 &
\end{array}\right) \quad \lambda_j \in B_{r}
\end{equation}
\begin{equation}
\Res_{\lambda=\lambda^*_{j}}\widetilde{M}^{(sol)}(x,t;\lambda)=\lim_{\lambda \rightarrow \lambda^*_{j}}\widetilde{M}^{(sol)}\left(\begin{array}{ccc}
0 & 0 & \\
-\lambda^*_j\bar{\widetilde{c}}_j\widetilde{\delta}^2(\lambda)e^{-i\widetilde{\theta}} & 0 &
\end{array}\right) \quad \lambda^*_j \in B_{r}
\end{equation}
\end{enumerate}
\end{problem}
\begin{problem}
\label{prob:MTM.A}
Find a matrix-valued function $\widetilde{M}^{A'}(\lambda;x,t)$ on $\bbC \setminus\Sigma_A'$ with the following properties:
\begin{enumerate}
\item		$\widetilde{M}^{A'}(\lambda;x,t) \rarr I$ as $ \lambda \rarr \infty$.
\item		$\widetilde{M}^{A'}(\lambda;x,t)$ is analytic for $\lambda \in  \bbC \setminus \Sigma_A' $
			with continuous boundary values
			$\widetilde{M}^{A'}_\pm(\lambda;x,t)$.
\item On $ \Sigma_A'$ we have the following jump conditions
$$\widetilde{M}^{A'}_+(\lambda;x,t)=\widetilde{M}^{A'}_-(\lambda;x,t)	
			\tv^{A'}(\lambda),$$
			where $\tv^{A'}=\tv^{(2)}\restriction _{\Sigma_A' }$.
\end{enumerate}
\end{problem}			

\begin{problem}
\label{prob:MTM.B}
Find a matrix-valued function $\widetilde{M}^{B'}(\lambda;x,t)$ on $\bbC \setminus\Sigma_B'$ with the following properties:
\begin{enumerate}
\item		$\widetilde{M}^{B'}(\lambda;x,t) \rarr I$ as $ \lambda \rarr \infty$;
\item		$\widetilde{M}^{B'}(\lambda;x,t)$ is analytic for $\lambda \in  \bbC \setminus \Sigma_B' $
			with continuous boundary values
			$\widetilde{M}^{B’}_\pm(\lambda;x,t)$;
\item On $ \Sigma_B'$ we have the following jump conditions
$$\widetilde{M}^{B'}_+(\lambda;x,t)=\widetilde{M}^{B’}_-(\lambda;x,t)	
			\tv^{B'}(\lambda),$$
			where $\tv^{B’}=\tv^{(2)}\restriction_{\Sigma_B' }$.
\end{enumerate}
\end{problem}	
We first study the solution to Problem \ref{prob:MTM.sol}. Since this problem consists of only discrete data, using Beals-Coifman formula, we have a closed system:
\begin{equation}
\label{BC-Int-solitons}
\begin{pmatrix}
{ \mu_{11}(\lambda_j ) } & {\mu_{12}( {\lambda^*_j} )}\\
{\mu_{21}(\lambda_j )} & {\mu_{22}( \lambda^*_j )}
\end{pmatrix}
= I +\begin{pmatrix}
        { \dfrac{\mu_{12}(\lambda^*_j)\widetilde{c}^*_j \lambda^*_j  \left[ \widetilde{\delta}(\lambda^*_j)\right]^{2} e^{-i\widetilde{\theta}(\lambda^*_j )} }{\lambda_j-\lambda^*_j }  } &
        { \dfrac{\mu_{11}( \lambda_j) \widetilde{c}_j \left[ \widetilde{\delta}(\lambda_j)\right]^{-2} e^{i\widetilde{\theta}(\lambda_j)} }{\lambda^*_j- \lambda_j}   }\\
        {\dfrac{\mu_{22}(\lambda^*_j)\widetilde{c}^*_j \lambda^*_j  \left[ \widetilde{\delta}(\lambda^*_j)\right]^{2} e^{-i\widetilde{\theta}(\lambda^*_j )} }{\lambda_j-\lambda^*_j }} &
        { \dfrac{\mu_{21}( \lambda_j) \widetilde{c}_j \left[ \widetilde{\delta}(\lambda_j)\right]^{-2} e^{i\widetilde{\theta}(\lambda_j)} }{\lambda^*_j- \lambda_j} },
        \end{pmatrix}
\end{equation}
Using Beals-Coifman's theorem, we arrive at the following expressions
\begin{equation}\label{BC-solitons}
\tm^{(sol)}(x,t,\lambda)=I+A
\begin{pmatrix}
{\dfrac {(\lambda_j-\lambda^*_j)\lambda^*_j\left|\widetilde{\delta} (\lambda^*_j)\right|^4|\widetilde{c}_j|^2e^{-2{\Im}\widetilde{\theta}(\lambda_j)}}{\lambda-\lambda^*_j}} &
\dfrac {(\lambda_j-\lambda^*_j)^2\widetilde{\delta}^{-2}(\lambda_j)\widetilde{c}_j e^{i\widetilde{\theta}(\lambda_j)}}{\lambda-\lambda_j}\\
\dfrac {-(\lambda_j-\lambda^*_j)^2\widetilde{\delta}^2(\lambda^*_j )\widetilde{c}^*_j\lambda^*_je^{-i\widetilde{\theta}(\lambda^*_j)}}{\lambda-\lambda^*_j} &
\dfrac {(\lambda^*_j-\lambda_j)\lambda^*_j\left|\widetilde{\delta} (\lambda^*_j)\right|^4|\widetilde{c}_j|^2e^{-2\Im \widetilde{\theta}(\lambda_j)}}{\lambda-\lambda_j},
\end{pmatrix}
\end{equation}
where $A^{-1}=(\lambda_j-\lambda^*_j)^2-\lambda^*_j\left|\widetilde{\delta} (\lambda^*_j)\right|^4|\widetilde{c}_j|^2 e^{-2{\Im}\widetilde{\theta}(\lambda_j)}$. Using the reconstruction formula of $u,v$, we have
\begin{eqnarray}
\begin{aligned}
&u=\overline{\widetilde{M}_{12}(x,t,0)}=\dfrac {(\lambda_j-\lambda^*_j)^2\widetilde{\delta}^{-2}(\lambda^*_j)\widetilde{c}^*_j e^{-i\widetilde{\theta}(\lambda^*_j)}}{-\lambda^*_j\left((\lambda_j-\lambda^*_j)^2-\lambda_j\left|\widetilde{\delta} (\lambda^*_j)\right|^4|\widetilde{c}_j|^2 e^{-2{\Im}\widetilde{\theta}(\lambda_j)}\right)}\\
&v=\widetilde{M}_{21}(x,t,0)=\dfrac {(\lambda_j-\lambda^*_j)^2\widetilde{\delta}^2(\lambda^*_j )\widetilde{c}^*_je^{-i\widetilde{\theta}(\lambda^*_j)}}{(\lambda_j-\lambda^*_j)^2-\lambda^*_j\left|\widetilde{\delta} (\lambda^*_j)\right|^4|\widetilde{c}_j|^2 e^{-2{\Im}\widetilde{\theta}(\lambda_j)}}.
\end{aligned}
\end{eqnarray}
Recall that
\begin{eqnarray*}
\begin{aligned}
\lambda_j=\xi_j+i\eta_j.
\end{aligned}
\end{eqnarray*}
We split $\widetilde{\theta}(\lambda_j;x,t)$ into real and imaginary parts,
\begin{eqnarray*}
\begin{aligned}
\widetilde{\theta}(\lambda_j;x,t)=\widetilde{\theta}_{\mathbb R}(\lambda_j;x,t)+i\widetilde{\theta}_{\mathbb I}(\lambda_j;x,t),
\end{aligned}
\end{eqnarray*}
with 
\begin{eqnarray*}
\begin{aligned}
&\widetilde{\theta}_{\mathbb R}(\xi_j,\eta_j;x,t)=\frac 12\left[\left(\xi_j-\frac {\xi_j}{\xi^2_j+\eta^2_j}\right)x+\left(\xi_j+\frac {\xi_j}{\xi^2_j+\eta^2_j}\right)t\right]\\
&\widetilde{\theta}_{\mathbb I}(\xi_j,\eta_j;x,t)=\frac 12\left[\left(\eta_j+\frac {\eta_j}{\xi^2_j+\eta^2_j}\right)x+\left(\eta_j-\frac {\eta_j}{\xi^2_j+\eta^2_j}\right)t\right].
\end{aligned}
\end{eqnarray*}
Therefore,
\begin{eqnarray*}
\begin{aligned}
&(\lambda_j-\lambda^*_j)^2-\lambda_j\left|\widetilde{\delta} (\lambda^*_j)\right|^4|\widetilde{c}_j|^2 e^{-2{\Im}\widetilde{\theta}(\lambda_j)}\\
=&(2i\eta_j)^2-(\xi_j+i\eta_j)\left|\widetilde{\delta} (\lambda^*_j)\right|^4|\widetilde{c}_j|^2 e^{-2{\Im}\widetilde{\theta}(\lambda_j)}\\
=&-4\eta_j\left|\widetilde{\delta} (\lambda^*_j)\right|^2|\widetilde{c}_j|(\xi^2_j+\eta^2_j)^{\frac 14}e^{-{\Im}\widetilde{\theta}(\lambda_j)+\frac i2\arctan(\frac {\eta_j}{\xi_j})}\cosh\left(\widetilde{\theta}_{\mathbb I}(\lambda_j)-\frac i2\arctan(\frac {\eta_j}{\xi_j})+\log\left[\frac {2\eta_j}{(\xi^2_j+\eta^2_j)^{\frac 14}|\widetilde{\delta}(\lambda^*_j)|^2|\widetilde{c}_j|}\right]\right)
\end{aligned}
\end{eqnarray*}
Finally,
\begin{align}
\label{u-sol}
u&=\dfrac {-\eta_j(\xi^2+\eta^2_j)^{-\frac 34}\widetilde{c}^*_je^{-i\widetilde{\theta}_{\mathbb R}(\lambda_j)+\frac i2\arctan(\frac {\eta_j}{\xi_j})}}{|\widetilde{\delta}(\lambda^*_j)|^2\widetilde{\delta}^2(\lambda^*_j)|\widetilde{c}_j|}{\text {sech}}\left(\widetilde{\theta}_{\mathbb I}(\lambda_j)-\frac i2\arctan(\frac {\eta_j}{\xi_j})+\log\left[\frac {2\eta_j}{(\xi^2_j+\eta^2_j)^{\frac 14}|\widetilde{\delta}(\lambda^*_j)|^2|\widetilde{c}_j|}\right]\right)\\
\label{v-sol}
v&=\dfrac {-4\eta^2_j\widetilde{\delta}^2(\lambda^*_j)\widetilde{c}^*_je^{-i\widetilde{\theta}_{\mathbb R}(\lambda_j)-\widetilde{\theta}_{\mathbb I}(\lambda_j)}}{-4\eta_j\left|\widetilde{\delta} (\lambda^*_j)\right|^2|\widetilde{c}_j|(\xi^2_j+\eta^2_j)^{\frac 14}e^{-\widetilde{\theta}_{\mathbb I}(\lambda_j)-\frac i2\arctan(\frac {\eta_j}{\xi_j})}\cosh\left(\widetilde{\theta}_{\mathbb I}(\lambda_j)+\frac i2\arctan(\frac {\eta_j}{\xi_j})+\log\left[\frac {2\eta_j}{(\xi^2_j+\eta^2_j)^{\frac 14}|\widetilde{\delta}(\lambda^*_j)|^2|\widetilde{c}_j|}\right]\right)}\\
\nonumber
&=\dfrac {\eta_j\widetilde{\delta}(\lambda^*_j)\widetilde{c}^*_je^{-i\widetilde{\theta}_{\mathbb R}(\lambda_j)+\frac i2\arctan(\frac {\eta_j}{\xi_j})}}{\widetilde{\delta}(\lambda_j)|\widetilde{c}_j|(\xi^2_j+\eta^2_j)^{\frac 14}}{\text {sech}}\left(\widetilde{\theta}_{\mathbb I}(\lambda_j)+\frac i2\arctan(\frac {\eta_j}{\xi_j})+\log\left[\frac {2\eta_j}{(\xi^2_j+\eta^2_j)^{\frac 14}|\widetilde{\delta}(\lambda^*_j)|^2|\widetilde{c}_j|}\right]\right)
\end{align}
In order to obtain the asymptotic expansions of solutions to \ref{prob:MTM.A} and \ref{prob:MTM.B},  we need the following matrix-valued function:
\begin{equation}
\label{eq: para}
\mathcal{P}=\begin{cases}
\twomat{\mathcal{P}_-}{0 }{0}{\mathcal{P}_-^{-1}}, \quad &\left\vert z + z_0 \right\vert<\rho \\
\twomat{\mathcal{P}_+}{0 }{0}{\mathcal{P}_+^{-1}}, \quad & \left\vert z - z_0 \right\vert <\rho \\
\twomat{1}{0}{0}{1},  \quad & \left\vert z \pm z_0 \right\vert \geq \rho
\end{cases}
\end{equation}
where 
\begin{align*}
\mathcal{P}_-&= (8 \tau)^{i \kappa / 2} e^{ i \tau}(-\lambda-\tz_0)^{-i\kappa^-} (-\zeta_-)^{-i\kappa^-} e^{i \zeta_-^{2} / 4}e^{i\ttheta/2} \\
\mathcal{P}_+&= (8 \tau)^{-i \kappa / 2} e^{- i \tau}(\lambda-\tz_0)^{i\kappa^+} \zeta_+^{i\kappa^+} e^{-i \zeta_+^{2} / 4}e^{i\ttheta/2} 
\end{align*}
with 
$$\zeta _\mp=    \sqrt{\dfrac{2t   }{ (1+\widetilde{z}_0^2)\widetilde{z}_0  } }  (\lambda \pm \widetilde{z}_0). $$
Then we further set 
\begin{equation}
\label{c}
 m^\RHP :=  \tilde{m_p}\mathcal{P}^{-1}
\end{equation}
where 
\begin{align}
\tilde{m_p}\restriction  \lbrace  z: \left\vert z + z_0 \right\vert < \rho \rbrace &=m^{A'} \twomat{\mathcal{P}_-}{0 }{0}{\mathcal{P}_-^{-1}}  := m^{A} ,\\
\tilde{m_p}\restriction  \lbrace  z: \left\vert z - z_0 \right\vert < \rho \rbrace &=m^{B'} \twomat{\mathcal{P}_+}{0 }{0}{\mathcal{P}_+^{-1}}  := m^{B}.
\end{align}
Denote
\begin{align*}
\delta^0_A &= (8\tau)^{i\kappa/2} e^{i\tau} \widetilde{\delta}_0^-\\
\delta^0_B &= (8\tau)^{-i\kappa/2} e^{-i\tau} \widetilde{\delta}_0^+
\end{align*}
\begin{figure}[H]
\caption{$\Sigma_A,\Sigma_B$}
\vskip 15pt
\begin{tikzpicture}[scale=0.6]
\draw[dashed] 				(-4,0) -- (4,0);							
\draw [->,thick,>=stealth] 	(0,0 )-- (1.732, 1);						
\draw  [thick]  (3.464,2 ) -- (1.732, 1);
\draw [->,thick,>=stealth] 	(0,0 )-- (1.732, -1);						
\draw  [thick]  (3.464,-2 ) -- (1.732, -1);
\draw [->,thick,>=stealth] 	(0,0 )-- (-1.732, 1);						
\draw  [thick]  (-3.464,2 ) -- (-1.732, 1);
\draw [->,thick,>=stealth] 	(0,0 )-- (-1.732, -1);						
\draw  [thick]  (-3.464, -2 ) -- (-1.732, -1);

\draw[fill]						(0,0)	circle[radius=0.075];		
\node [below] at  			(0,-0.15)		{$0$};
\node[left] at					(2.5, 2.3)		{$\Sigma_A^1$};
\node[right] at					(-2.5,2.3)		{$\Sigma_A^2$};
\node[right] at					(-2.5,-2.3)		{$\Sigma_A^3$};
\node[left] at					(2.5,-2.3)		{$\Sigma_A^4$};
\end{tikzpicture}
\qquad
\begin{tikzpicture}[scale=0.6]
\draw[dashed] 				(-4,0) -- (4,0);							
\draw[dashed] 				(-4,0) -- (4,0);							
\draw [thick] 	(0,0 )-- (1.732, 1);						
\draw [->,thick,>=stealth]   (3.464,2 ) -- (1.732, 1);
\draw [thick] 	(0,0 )-- (1.732, -1);						
\draw  [->,thick,>=stealth]  (3.464,-2 ) -- (1.732, -1);
\draw [thick] 	(0,0 )-- (-1.732, 1);						
\draw [->,thick,>=stealth] (-3.464,2 ) -- (-1.732, 1);
\draw [thick] 	(0,0 )-- (-1.732, -1);						
\draw [->,thick,>=stealth]  (-3.464, -2 ) -- (-1.732, -1);

\draw[fill]						(0,0)	circle[radius=0.075];		
\node [below] at  			(0,-0.15)		{$0$};
\node[left] at					(2.5, 2.3)		{$\Sigma_B^1$};
\node[right] at					(-2.5,2.3)		{$\Sigma_B^2$};
\node[right] at					(-2.5,-2.3)		{$\Sigma_B^3$};
\node[left] at					(2.5,-2.3)		{$\Sigma_B^4$};
\end{tikzpicture}
\label{fig:jumps-A-B}
\end{figure}
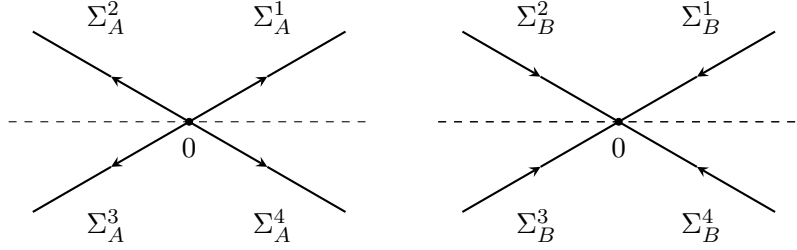
 Let $\Sigma_A$ and $\Sigma_B$ denote the contours
\begin{equation*}
\lbrace \lambda= u e^{\pm i\pi/6} : -\infty<u<\infty \rbrace
\end{equation*}
with the same orientation as those of $\Sigma_{A'}$ and $\Sigma_{B'}$ respectively.
$m^A$ solves the following Riemann-Hilbert problem
\begin{equation}
\left\{\begin{array}{ll}m_{+}^{A}(\zeta) & =m_{-}^{A}(\zeta) v^{A}(\zeta), \quad \zeta \in \Sigma_{A} \\\\ m^{A}(\zeta) &  = I-\frac{m_{1}^{A}}{\zeta}+O\left(\zeta^{-2}\right), \quad \zeta \rightarrow \infty\end{array}\right.
\end{equation}
We have from the list of entries stated in \eqref{R1}-\eqref{R8-} the rescaled jump matrices on $\Sigma_A$ and $\Sigma_B$ respectively :
\begin{equation}
\label{jump-vA}
v^{A}=\begin{cases}
\Twomat{1}{0}{-\dfrac{(\delta^0_A)^{2} \widetilde{z}_0\overline{\widetilde{r}(-\widetilde{z}_0)}}{1 - \widetilde{z}_0|\widetilde{r}(-\widetilde{z}_0)|^2} (-\zeta_-)^{ -2i\kappa^-} e^{ i \zeta_-^2/2}  }{1}, \quad \zeta_- \in \Sigma_A^1 \\
 \Twomat{1}{(\delta^0_A)^{-2}\widetilde{r}(-\widetilde{z}_0)(-\zeta_-)^{2i\kappa^-} e^{ -i \zeta_-^2/2}}{0  }{1}, \quad \zeta_- \in \Sigma_A^2 \\
\Twomat{1}
{0 }{(\delta^0_A)^{2}\widetilde{z}_0\overline{\widetilde{r}(-\widetilde{z}_0)}(-\zeta_-)^{- 2i\kappa^-} e^{ i \zeta_-^2/2}}{1}, \quad \zeta_- \in \Sigma_A^3 \\
\Twomat{1}{-\dfrac {(\delta^0_A)^{-2}\widetilde{r}(-\widetilde{z}_0)}{1-\widetilde{z}_0|\widetilde{r}(-\widetilde{z}_0)|^2}(-\zeta_-)^{ 2i\kappa^-} e^{ -i \zeta_-^2/2}  } {0}{1}, \quad \zeta_- \in \Sigma_A^4 \\
\end{cases}
\end{equation}

\begin{equation}
\label{jump-vB}
v^{B}=\begin{cases}
 \Twomat{1}{(\delta^0_B)^{-2}\widetilde{r}(\widetilde{z}_0)\zeta_+^{- 2i\kappa^+} e^{ i \zeta_+^2/2}}{0  }{1}, \quad \zeta_+ \in \Sigma_B^1 \\
\Twomat{1}{0}{-\dfrac{(\delta^0_B)^{2} \widetilde{z}_0\overline{\widetilde{r}(\widetilde{z}_0)}}{1 +  \widetilde{z}_0|\widetilde{r}(\widetilde{z}_0)|^2} \zeta_+^{ 2i\kappa^+} e^{ -i \zeta_+^2/2}  }{1}, \quad \zeta_+ \in \Sigma_B^2 \\
\Twomat{1}{-\dfrac {(\delta^0_B)^{-2}\widetilde{r}(\widetilde{z}_0)}{1+\widetilde{z}_0|\widetilde{r}(\widetilde{z}_0)|^2}\zeta_+^{- 2i\kappa^+} e^{ i \zeta_+^2/2}  } {0}{1}, \quad \zeta_+ \in \Sigma_B^3 \\
 \Twomat{1}
{0 }{(\delta^0_B)^{2}\widetilde{z}_0\overline{\widetilde{r}(\widetilde{z}_0)}\zeta_+^{ 2i\kappa^+} e^{ -i \zeta_+^2/2}}{1}, \quad \zeta_+ \in \Sigma_B^4 .
\end{cases}
\end{equation}
$m^B$ solves the following Riemann-Hilbert problem
\begin{equation}
\left\{\begin{array}{ll}m_{+}^{B}(\zeta) & =m_{-}^{B^{0}}(\zeta) v^{B}(\zeta), \quad \zeta \in \Sigma_{B} \\\\ m^{B}(\zeta) & = I-\frac{m_{1}^{B}}{\zeta}+O\left(\zeta^{-2}\right), \quad \zeta \rightarrow \infty\end{array}\right.
\end{equation}
The explicit form of $m^{B^0}_1$ is given as follows :
\begin{equation}
\label{explicit-B0}
m^{B^0}_1=\twomat{0}{-i\beta^B_{21}}{i\beta^B_{12}}{0}
\end{equation}
where
$$\beta^B_{12}=\dfrac{\sqrt{2\pi } e^{i\pi/4} e^{-\pi \kappa/2 } }{\widetilde{r}(\widetilde{z}_0) \Gamma(-i\kappa)}, \qquad \beta^B_{21}=\dfrac {\kappa}{\beta^B_{12}}$$
and $\Gamma(z)$ is the \textit{Gamma} function.
Recall that on $\Sigma_B$, $\zeta=  \sqrt{\dfrac{2t   }{ (1+\widetilde{z}_0^2)\widetilde{z}_0  } }  (\lambda - \widetilde{z}_0) $, thus we have
\begin{equation}
\label{differ-m-B}
\left\vert   \dfrac{m_1^B}{\zeta}-\dfrac{m_1^{B^0}}{\zeta}   \right\vert \lesssim \dfrac{1}{t(\lambda-\widetilde{z}_0)}.
\end{equation}
In the same fashion, the explicit form of $m^{A^0}_1(\zeta)$ is
\begin{equation}
\label{explicit-B0}
m^{A^0}_1=\twomat{0}{i\beta^A_{12}}{-i\beta^A_{21}}{0}
\end{equation}
where
\begin{align*}
\beta^A_{12}=\dfrac{\sqrt{2\pi } e^{i\pi/4} e^{-\pi \kappa/2 } }{-\widetilde{z}_0\overline{\widetilde{r}(-\widetilde{z}_0)} \Gamma(-i\kappa)}, \qquad \beta^A_{21}=\dfrac {\kappa}{\beta^A_{12}}
\end{align*}
We also have an analogue of \eqref{differ-m-B} for $m_1^{A^0}$:
\begin{equation}
\label{differ-m-A}
\left\vert   \dfrac{m_1^A}{\zeta}-\dfrac{m_1^{A^0}}{\zeta}   \right\vert \lesssim \dfrac{1}{t(\lambda+\widetilde{z}_0)}.
\end{equation}
\begin{proposition}
\label{solution-A-B}
Setting $\zeta=  \sqrt{\dfrac{2t   }{ (1+\widetilde{z}_0^2)\widetilde{z}_0  } }  (\lambda + \widetilde{z}_0)  $, the solution to RHP Problem  \ref{prob:MTM.A}  $\tm^{A'}$ admits the following expansion:
\begin{equation}
\label{expansion-A}
\tm^{A'}(\lambda(\zeta) ;x,t)=I +\dfrac{1}{\zeta}\twomat{0}{i ( -\delta^0_A)^{-2} \beta^A_{12}}{-i ( \delta^0_A)^{2}\beta^A_{21} }{0} +\mathcal{O}(\tau^{-1}).
\end{equation}
Similarly, in the case of $\zeta=  \sqrt{\dfrac{2t   }{ (1+\widetilde{z}_0^2)\widetilde{z}_0  } }  (\lambda - \widetilde{z}_0)  $, the solution to RHP Problem  \ref{prob:MTM.B}  $\tm^{B'}$ admits the following expansion:
\begin{equation}
\label{expansion-B}
\tm^{B'}(\lambda(\zeta) ;x,t)=I +\dfrac{1}{\zeta}\twomat{0}{-i ( \delta^0_B)^{-2} {\beta}^B_{21}}{i ( \delta^0_B)^{2}{\beta}^B_{12} }{0} +\mathcal{O}(\tau^{-1}).
\end{equation}
\end{proposition}
\subsection{The derivation of $\tm^\RHP$}
Now we construct $\tm^\RHP$ defined by problem \ref{MTM.RHP.local}.
In Figure \ref{fig:contour-2}, we let $\rho$ be the radius of the circle $C_A$ ($C_B$) centered at $\widetilde{z}_0$ ($-\widetilde{z}_0$). We seek a solution of the form
\begin{equation}
\label{parametrix}
\tm^\RHP(\lambda)=\begin{cases}
E(\lambda)\tm^{(out)}(\lambda) \quad  &\left\vert \lambda\pm \widetilde{z}_0 \right\vert>\rho \\
E(\lambda)\tm^{(out)}(\lambda)  \tm^{A'}(\lambda) \quad &\left\vert \lambda + \widetilde{z}_0 \right\vert\leq\rho \\
E(\lambda)\tm^{(out)}(\lambda)  \tm^{B'}(\lambda) \quad & \left\vert \lambda - \widetilde{z}_0 \right\vert\leq \rho
\end{cases}
\end{equation}
where $\tm^{(out)}$ solves the discrete Riemann-Hilbert problem whose jump condition are given by problem \ref{modifiedRHP} with $\widetilde{r}\equiv 0$.
Since $ \tm^{A'}$ and $ \tm^{B'}$ solve  Problem \ref{prob:MTM.A} and Problem \ref{prob:MTM.B} respectively and discrete RHPs have explicit solutions, we can construct the solution $\tm^\RHP(\lambda)$ if we find $E(\lambda)$. Indeed,  $E$ solves the following Riemann-Hilbert problem on $\Sigma_E$ given in Figure
\ref{fig:contour-E} :
\begin{problem}
\label{prob: E-2}
Find a matrix-valued function $E (\lambda)$ on $\bbC \setminus \Sigma_E $ with the following properties:
\begin{enumerate}
\item		$E(\lambda) \rarr I$ as $ \lambda \rarr \infty$,
\item		$E (\lambda) $ is analytic for $\lambda \in  \bbC \setminus \Sigma_E$
			with continuous boundary values
			$E_{\pm}(\lambda)$.
\item On $ C_A\cup C_B $, we have the following jump conditions
\begin{equation}
\label{jump:E}
E_{+}(\lambda)=E_{-}(\lambda) \tv^{ (E) }(\lambda)
\end{equation}
			where
			\begin{equation}
			\label{v-E}
			 \tv^{ (E) }(\lambda)=\begin{cases}
			\tm^{(out)} (\lambda)\tm^{A'}(\lambda(\zeta) )[\tm^{(out )} (\lambda)]^{-1}, \quad \lambda\in C_A\\
			\tm^{(out )} (\lambda)\tm^{B'}(\lambda(\zeta) )[\tm^{( out )} (\lambda)]^{-1}, \quad \lambda\in C_B\\
			\tm^{(out )} (\lambda)v^{(2)} [\tm^{( out )} (\lambda)]^{-1}, \quad \lambda\in \Sigma_E\setminus\left(  C_A\cup C_B\right)
			\end{cases}
			\end{equation}
\end{enumerate}
\end{problem}
\begin{figure}[H]
\caption{$\Sigma_E$}
\vskip 15pt

\begin{tikzpicture}[ photon/.style={decorate,decoration={snake,post length=0.8mm}} ][scale=0.6]

\draw[thick]   [red]			(4.472,0.467) -- (5.296, 1);								
\draw[->,thick,>=stealth] [red]	(   7.028 ,2  )--	(5.296, 1);

\draw[->,thick,>=stealth] 	  [red]			(1.732 ,1) -- (2.656 ,0.467) ;

\draw[thick] 		  [red]		(4.472,-0.467) -- (5.296, -1);								
\draw[->,thick,>=stealth]  [red]		(   7.028 , -2  )--	(5.296, -1);

\draw[->,thick,>=stealth] 	  [red]			(1.732 ,-1) -- (2.656 ,-0.467) ;

\draw[->,thick,>=stealth] 	  [red]			 (-2.656 ,-0.467)--(-1.732 ,-1) ;
\draw[->,thick,>=stealth] 	  [red]			 (-2.656 ,0.467)--(-1.732 ,1) ;

\draw[->,thick,>=stealth]   [red]				(-4.472 ,0.467) -- (-5.296, 1);						
\draw[thick]  [red]		(   -7.028 ,2  )--	(-5.296, 1);
\draw[->,thick,>=stealth] 	  [red]			(-4.472 ,-0.467) -- (-5.296, -1);							
\draw[thick]  [red]		(   -7.028 , -2  )--	(-5.296, -1);

\draw[->,thick,>=stealth]  [red]				(-1.732 , -1) -- (-1.732, 0);	
\draw	[thick]  [red]		(-1.732 , 0) -- (-1.732, 1);	
\draw[->,thick,>=stealth] 	 [red]			(1.732 , -1) -- (1.732, 0);	
\draw	[thick]	 [red]	 (1.732 , 0) -- (1.732, 1);

\draw [blue][dashed] (0,0) circle[radius=3.564];

\draw	[fill]							(-3.564 ,0)		circle[radius=0.06];	
\draw	[fill]							(3.564, 0)		circle[radius=0.06];
\draw		[fill]					(0,0)		circle[radius=0.1];

\node[below] at (-3.564,-0.1)			{$-\widetilde{z}_0$};
\node[below] at (3.564,-0.1)			{$\widetilde{z}_0$};

\node[right] at (5, 1.6)					{$\Sigma_1'^{(2)} $};
\node[left] at (-5, 1.6)					{$\Sigma_2'^{(2)}$};
\node[left] at (-5,-1.6)					{$\Sigma_3'^{(2)}$};
\node[right] at (5,-1.6)				{$\Sigma_4'^{(2 )}$};
\node[left] at (-1,1.2)					{$\Sigma_5'^{(2)}$};
\node[left] at (-1,-1.2)					{$\Sigma_7'^{( 2 )}$};
\node[right] at (1,1.2)					{$\Sigma_6'^{(2 )}$};
\node[right] at (1,-1.2)					{$\Sigma_8'^{(2 )}$};
\node[right] at (-1.732, 0) {$\Sigma_9'^{(2)} $};
\node[right] at (1.732, 0) {$\Sigma_{10}'^{(2 )} $};

 \draw [red, fill=red] (1, 1.732) circle [radius=0.05];
\draw[->,>=stealth] [red] (1.2, 1.732) arc(360:0:0.2);
\draw [red, fill=red] (1,-1.732) circle [radius=0.05];
\draw[->,>=stealth] [red] (1.2, -1.732) arc(0:360:0.2);
\draw [red, fill=red]  (0, 2.7) circle [radius=0.05];
\draw[->,>=stealth] [red] (0.3, 2.7) arc(360:0:0.3);
\draw [red, fill=red] (0, -2.7) circle [radius=0.05];
\draw[->,>=stealth] [red]  (0.3, -2.7) arc(0:360:0.3);
\draw[->,>=stealth] (3.4,2) arc(360:0:0.4);
\draw[->,>=stealth] (3.4,-2) arc(0:360:0.4);
\draw [red, fill=red] (3,2) circle [radius=0.1];
\draw [red, fill=red] (3,-2) circle [radius=0.1];

\draw[->,thick,>=stealth] (4.564,0) arc(360:0:1);
\draw[->,thick,>=stealth] (-2.564,0) arc(360:0:1);

\draw[->,photon] (-4, -2.) --  (-4, -1.);
\draw[->,photon] (4, -2) --  (4, -1);
\node [below] at  (-4, -2.131)  {$C_B$};
\node [below] at  (4, -2.131) {$C_A$};

\end{tikzpicture}
\label{fig:contour-E}
\begin{center}
  \begin{tabular}{ccc}
The jumps have exponential decay on the red portion of the contour.
\end{tabular}
 \end{center}
\end{figure}
Note that we omitted the contours passing through the origin since the jumps are identity by definition (see \eqref{R5}-\eqref{R8-}) .
\begin{proposition}
 $E(z)$ admits a classical solution, i.e jump condition \eqref{jump:E} holds pointwise on the contour $\Sigma_E$.
\begin{proof}
Here we invoke to the well-established existence and uniqueness theory from \ref{Solvability}. We then take care of the zero sum condition at the self-intersecting points of $\Sigma_E$. Since the remaining cases follows from symmetry, we will only look $\Sigma'^{(2)}_6 \cap \Sigma'^{(2)}_8 $ and $\Sigma'^{(2)}_6\cap C_A $. The zero sum condition holds at the first point by comparing \eqref{R8+} and the third line of \eqref{jump v9}. For $\Sigma'^{(2)}_6\cap C_A $, we explicitly compute
\begin{align*}
I & = \left[ m^{(out)} (\lambda)\tm^{A'}(\lambda(\zeta) )\left[m^{( out )} (\lambda)\right]^{-1}  \right] \left[  m^{(out )} (\lambda)v^{(2)} \left[m^{( out )} (\lambda)\right]^{-1}  \right]^{-1} \\
  &\quad \times \left[ m^{(out)} (\lambda)m^{A'}(\lambda(\zeta) )\left[m^{( out )} (\lambda)\right]^{-1}  \right] \\
   &= m^{(out)} (\lambda) m^{A'}_+(\lambda(\zeta) )\left( v^{(2)} \right)^{-1} \left( m^{A'}_-(\lambda(\zeta))\right)^{-1}\left[m^{( out )} (\lambda)\right]^{-1} .
\end{align*}
Since $v^{(2)}$ is smooth away from the intersections and zero sum conditions have been verified, this completes the proof.
\end{proof}
\end{proposition}
Furthermore, for $t$ large enough, we can solve a small norm RHP to obtain some desired estimates on $E$. Setting
$$\eta(\lambda)=E_{-}(\lambda)-I$$
then by the standard theory, we have the following singular integral equation
$$ \left( \mathbf{1}-C_{v^{(E)}}  \right) \eta =C_{v^{(E)}} I$$
where the singular integral operator is defined by:
$$  C_{\tv^{(E)}}f =C^-\left( f \left( \tv^{(E)}-I  \right)  \right). $$
We first deduce from \eqref{expansion-A}-\eqref{expansion-B} that
\begin{equation}
\norm{\tv^{(E)}-I}{L^\infty} \lesssim t^{-1/2}.
\end{equation}
Hence the operator norm of $C_{\tv^{(E)}} $
\begin{equation}
\label{norm-vE}
\norm{C_{\tv^{(E)}}f }{ L^2} \leq \norm{f}{L^2} \norm{\tv^{(E)}-I}{L^\infty} \lesssim t^{-1/2}.
\end{equation}
Then the resolvent operator $(1-C_{v^{(E)}})^{-1}$ can be obtained through a \textit{Neumann} series and we obtain the unique solution to Problem \ref{prob: E-2}:
\begin{align}\label{Error-sol}
E(\lambda) &=I+ \dfrac{1}{2\pi i}\int_{C_A\cup C_B} \dfrac{ (1+\eta(s))(\tv^{(E)}(s)-I )   }{s-\lambda}ds\\
    &\quad +\dfrac{1}{2\pi i}\int_{  \Sigma_E \setminus ( C_A\cup C_B) } \dfrac{ (1+\eta(s))(\tv^{(E)}(s)-I )   }{s-\lambda}ds\\
    &=I+ \dfrac{1}{2\pi i}\int_{C_A\cup C_B} \dfrac{ (1+\eta(s))(\tv^{(E)}(s)-I )   }{s-\lambda}ds +\mathcal{O}(e^{-ct}),
\end{align}
where we make use of the fact that $\tv^{(E)}(s)-I $ decays exponentially on $ \Sigma_E \setminus ( C_A\cup C_B) $.

We  also have the bound on $\eta$ on $C_A\cup C_B$:
\begin{equation}
\label{norm-eta}
\norm{\eta}{L^2}\lesssim t^{-1/2}.
\end{equation}
Letting $\lambda\to 0$ and using the bounds given by \eqref{norm-vE} and \eqref{norm-eta} and an application of Cauchy-Schwarz inequality, we obtain
\begin{align}
\label{bound-E-1}
E (0) - \mathcal{O}(e^{-ct})&=I+ \dfrac{1}{2\pi i}\int_{C_A\cup C_B}   \dfrac{ { (1+\eta(s))(\tv^{(E)}(s)-I )   }  } {s}ds  \\
\nonumber
              &=I+\dfrac{1}{2\pi i}\int_{C_A\cup C_B} \dfrac{ \tv^{(E)}(s)-I    }{  s }ds +\dfrac{1}{2\pi i}\int_{C_A\cup C_B} \dfrac{\eta(s) (\tv^{(E)}(s)-I)    }{  s }ds \\
              \nonumber
              &\lesssim I+\dfrac{1}{2\pi i}\int_{C_A\cup C_B} \dfrac{ \tv^{(E)}(s)-I    }{  s }ds +
              \norm{\eta}{L^2} \norm{\tv^{(E)}-I}{L^\infty\cap L^2 } \\
              \nonumber
             & \lesssim  I+\dfrac{1}{2\pi i}\int_{C_A\cup C_B} \dfrac{ \tv^{(E)}(s)-I    }{  s }ds +\mathcal{O}(t^{-1}).
\end{align}

Using the formula for $v^{(E)}$ in \eqref{v-E} and the asymptotic expansions \eqref{expansion-A}-\eqref{expansion-B}, and applying the Cauchy's integral formula leads to
\begin{align}
\label{E-1-2 cauchy}
E(0) &=I + \dfrac{1}{\widetilde{z}_0   \sqrt{\dfrac{2t   }{ (1+\widetilde{z}_0^2)\widetilde{z}_0  } } } \tm^{(out)}(\widetilde{z}_0) \twomat{0}{i ( \delta^0_B)^2 {\beta}^B_{21}}{ -i ( \delta^0_B)^{-2}{\beta}^B_{12} }{0}  [\tm^{(out)}(\widetilde{z}_0)]^{-1}\\
\nonumber
    & \quad + \dfrac{1}{\widetilde{z}_0   \sqrt{\dfrac{2t   }{ (1+\widetilde{z}_0^2)\widetilde{z}_0  } } } \tm^{(out)}(-\widetilde{z}_0) \twomat{0}{- i ( \delta^0_A)^2 \overline{\beta}^A_{12} }{i ( \delta^0_A)^{-2}\overline{\beta}^A_{21} }{0} [\tm^{(out)}(-\widetilde{z}_0)]^{-1}\\
    \nonumber
     & \quad + \mathcal{O}(t^{-1}).
\end{align}
In order to reconstruct the solution, we also need the large $\lambda$ behavior of the solution of \ref{prob: E-2}. Geometrically expanding $(s-\lambda)^{-1}$ for large $\lambda$ in \eqref{Error-sol}, we have
\begin{align}
\label{large-expansion}
E(\lambda)=I+\frac {E_1}{\lambda}+ O(\lambda^{-2})
\end{align}
where
\begin{align*}
E_1=-\frac 1{2\pi i}\int_{C_A\cup C_B} (\tv^{(E)}(s)-I )  ds+ O(t^{-1})
\end{align*}
Similar to the asymptotic expansion at the origin, we have
\begin{align}\label{expansion-E1}
E_1&=\dfrac{1}{\sqrt{\dfrac{2t   }{ (1+\widetilde{z}_0^{-2})\widetilde{z}_0  } } }\tm^{(out)}(\widetilde{z}_0) \twomat{0}{-i ( \delta^0_B)^{-2} {\beta}^B_{21}}{ i ( \delta^0_B)^{2}{\beta}^B_{12} }{0}  [\tm^{(out)}(\widetilde{z}_0)]^{-1}\\
&+ \dfrac{1}{\sqrt{\dfrac{2t   }{ (1+\widetilde{z}_0^2)\widetilde{z}_0  } } } \tm^{(out)}(-\widetilde{z}_0) \twomat{0}{ i ( \delta^0_A)^{-2} \overline{\beta}^A_{12} }{-i ( \delta^0_A)^{2}\overline{\beta}^A_{21} }{0} [\tm^{(out)}(-\widetilde{z}_0)]^{-1}
\end{align}
Combining this with Proposition \ref{expo}, we obtain $\widetilde{M}^\RHP(\lambda)$ in \eqref{factor-LC}.
\section{The $\dbar$-Problem}
\label{sec:dbar}

From \eqref{factor-LC} we have the matrix-valued function
\begin{equation}
\label{N3}
\widetilde{M}^{(3)}(\lambda;x,t) = \widetilde{M}^{(2)}(\lambda;x,t) \left[\widetilde{M}^\RHP(\lambda; x,t)\right]^{-1}.
\end{equation}
The goal of this section is to show that $\widetilde{M}^{(3)}$ only results in an error term $E(x, t)$ with higher order decay rate than the leading order term of the asymptotic formula. The computations and proofs are standard.
Since $\widetilde{M}^\RHP(\lambda; x, t)$ is analytic in $\bbC \setminus \left( \Sigma'^{(2)} \cup \Gamma  \right)$, we may compute
\begin{align*}
\dbar \widetilde{M}^{(3)}(\lambda;x,t) 	&=	\dbar \widetilde{M}^{(2)}(\lambda;x,t) [\widetilde{M}^\RHP(\lambda; x,t)]^{-1}\\	
								&=	\widetilde{M}^{(2)}(\lambda;x,t) \, \dbar \calR^{(2)}(\lambda) \left[\widetilde{M}^\RHP(\lambda; x,t)\right]^{-1}	&\text{(by \eqref{N2.dbar})}\\
								&=	\widetilde{M}^{(3)}(\lambda;x,t) \widetilde{M}^{\RHP}(\lambda; x, t) \, \dbar \calR^{(2)}(\lambda) [\widetilde{M}^\RHP(\lambda; x,t)]^{-1}	& \text{(by \eqref{N3})}\\
								&=	\widetilde{M}^{(3)}(\lambda;x,t)  W(\lambda;x,t)
\end{align*}
where
 \begin{equation}
 \label{W-bound}
 W(\lambda;x,t) = \widetilde{M}^{\RHP}(\lambda;x,t ) \, \dbar \calR^{(2)}(\lambda) [ \widetilde{M}^\RHP(\lambda;x,t )]^{-1}.
 \end{equation}
 We thus arrive at the following pure $\dbar$-problem.
\begin{problem}
\label{prob:DNLS.dbar}
Give $\widetilde{r}(\lambda)\in H^{1,1}_0(\bbR)$, find a continuous matrix-valued function
$\widetilde{M}^{(3)}(\lambda;x,t)$ on $\bbC$ with the following properties:
\begin{enumerate}
\item		$\widetilde{M}^{(3)}(\lambda;x,t) \rarr I$ as $|\lambda| \rarr \infty$;
\smallskip
\item		$\dbar \widetilde{M}^{(3)}(\lambda;x,t) = \widetilde{M}^{(3)}(\lambda;x,t) W(\lambda;x,t)$.
\end{enumerate}
\end{problem}
It is well understood that the solution to this $\dbar$ problem is equivalent to the solution of a Fredholm-type integral equation involving the solid Cauchy transform
$$ (Pf)(\lambda) = \frac{1}{\pi} \int_\bbC \frac{1}{\zeta-\lambda} f(\zeta) \, d\zeta $$
where $d\zeta$ denotes Lebesgue measure on $\bbC$.
\begin{lemma}
A bounded and continuous matrix-valued function $\tm^{(3)}(\lambda;x,t)$ solves Problem \ref{prob:DNLS.dbar} if and only if
\begin{equation}
\label{DNLS.dbar.int}
\widetilde{M}^{(3)}(\lambda;x,t) =I+ \frac{1}{\pi} \int_\bbC \frac{1}{\zeta-\lambda} \widetilde{M}^{(3)}(\zeta;x,t) W(\zeta;x,t) \, d\zeta.
\end{equation}
\end{lemma}

 Using the integral equation formulation \eqref{DNLS.dbar.int}, we will prove the following result:
\begin{proposition}
\label{prop:N3.est}
Suppose that $\widetilde{r} \in H^{1,1}_0(\bbR)$.
Then, for $t\gg 1$, there exists a unique solution $\widetilde{M}^{(3)}(\lambda;x,t)$ for Problem \ref{prob:DNLS.dbar} with the properties that
\begin{enumerate}
\item
\begin{equation}
\label{N3.exp}
\widetilde{M}^{(3)}(\lambda;x,t) =   \widetilde{M}^{(3),0}_1(x,t) +\widetilde{M}^{(3),0}_2(\lambda; x,t)
\end{equation}
with
\begin{equation}
\label{0-dbar}
\lim_{\lambda\to 0} \widetilde{M}^{(3),0}_2(\lambda; x,t)=0.
\end{equation}
Here
\begin{equation}
\label{N31.est}
\left| \widetilde{M}^{(3),0}_1(x,t) \right| \lesssim \tau^{-3/4}.
\end{equation}
\item
\begin{equation}
\label{M3.exp}
\widetilde{M}^{(3)}(\lambda;x,t) = I+\frac {\widetilde{M}^{(3),\infty}_1}{\lambda}+O(\lambda^{-2})
\end{equation}
with
\begin{equation}
\label{M31.est}
\left| \widetilde{M}^{(3),\infty}_1(x,t) \right| \lesssim \tau^{-3/4}
\end{equation}

where the implicit constants in \eqref{N31.est} and \eqref{M31.est} are uniform for
$\widetilde{r}$ in a bounded subset of $H^{1,1}_0(\bbR)$.
\end{enumerate}
\end{proposition}
\begin{proof} Assuming the Lemmas \ref{lemma:dbar.R.bd}--\ref{lemma:N31.est} hold,
 as in \cite{LPS}, we first show that, for large $t$, the integral operator $K_W$
defined by
\begin{equation*}
\left( K_W f \right)(\lambda) = \frac{1}{\pi} \int_\bbC \frac{1}{\zeta-\lambda} f(\zeta) W(\zeta) \, d \zeta
\end{equation*}
is bounded by
\begin{equation}
\label{dbar.int.est1}
\norm{K_W}{L^\infty \rarr L^\infty} \lesssim \tau^{-1/4},
\end{equation}
where the implied constant depends only on $\norm{r}{H^{1,1}}$ (Lemma \ref{lemma:KW}).
It implies that
\begin{equation}
\label{N3.sol}
\widetilde{M}^{(3)} = ( \mathbf{1} -K_W)^{-1}I
\end{equation}
exists as an $L^\infty$ solution of \eqref{DNLS.dbar.int}.
In Lemma \ref{lemma:N31.est} we eventually estimate \eqref{N31.est}
where the constants are uniform in $r$ belonging to a bounded subset of $H_0^{1,1}(\bbR)$.
The estimates on \eqref{N3.exp} from \eqref{N31.est}, and \eqref{dbar.int.est1} are results of the  bounds obtained in the next four lemmas. For \eqref{N3.exp}, we can decompose $\tm^{(3)}(\lambda;x,t)$ into:
\begin{align*}
\widetilde{M}^{(3)}(\lambda;x,t) &=I+ \frac{1}{\pi} \int_\bbC \frac{1}{\zeta-\lambda} \widetilde{M}^{(3)}(\lambda;x,t) W(\zeta;x,t) \, d\zeta\\
                   &= \underbrace{ \left( I+  \frac{1}{\pi} \int_\bbC \frac{1}{\zeta}\widetilde{M}^{(3)}(\zeta;x,t) W(\zeta;x,t) \, d\zeta \right)}_{\widetilde{M}^{(3)}_1} \\
                   &+\underbrace{ \left( \frac{1}{\pi} \int_\bbC \frac{1}{\zeta-\lambda} \widetilde{M}^{(3)}(\zeta;x,t) W(\zeta;x,t) \, d\zeta - \frac{1}{\pi} \int_\bbC \frac{1}{\zeta} m^{(3)}(\zeta;x,t) W(\zeta;x,t) \, d\zeta\right) }_{\widetilde{M}^{(3)}_2}\\
                   &= : \widetilde{M}^{(3)}_1(x,t) +\widetilde{M}^{(3)}_2(\lambda; x,t).
\end{align*}
We set $\lambda=i\sigma$ and let $\sigma\to 0$. Then near the origin, one has
\begin{align*}
\left\vert \dfrac{1}{s-\lambda} \right\vert &= \left\vert \dfrac{s}{s-\lambda} \right\vert \left\vert \dfrac{1}{s} \right\vert \\
                                                    &= \left\vert \dfrac{u^2+v^2}{u^2+(v-\sigma)^2 } \right\vert^{1/2} \left\vert \dfrac{1}{u^2+v^2} \right\vert^{1/2}\\
                                                    &\leq  \dfrac{2}{\sqrt{3}}\left\vert \dfrac{1}{s} \right\vert.
\end{align*}
 Note that by the construction of $W(\lambda; x, t)$ in \eqref{W-bound}, $W$ is bounded in $\bbC$ and in particular near the origin. Therefore the dominated convergence theorem will lead to \eqref{0-dbar}.
\end{proof}
For simplicity we only work with regions $\Omega_8$ and $\Omega_3$. For technical reasons we further divide $\Omega_3$ into two parts. See Figure \ref{fig:three reigons} below.

\begin{figure}[H]
\caption{Four regions }
\vskip 15pt
\begin{tikzpicture}[scale=1.1]

\draw 	[thick] 	(3.564, 0)-- (5.296, 0);	
\draw 	[->,thick,>=stealth] 	(7.028,0 )-- (5.296, 0);
\draw [dashed]  (5.296, 0)--(5.296, 1);
\node[below] at   (5.296, 0) {$\tz_0+1$};


\draw 	[->,thick,>=stealth] 	(0, 0)-- (1.732, 0);	
\draw 	[thick] 	 (1.732, 0)--(3.564,0 );

\draw[thick] 			(3.564 ,0) -- (5.296, 1);								
\draw[->,thick,>=stealth] 	(   7.028 ,2  )--	(5.296, 1);

\draw[->,thick,>=stealth] 			(1.732 ,1) -- (3.564 ,0) ;
\draw[->,thick,>=stealth] 			(0,0)--(1.732 ,1);

%

\draw	[fill]							(3.564, 0)		circle[radius=0.05];
\draw							(0,0)		circle[radius=0.1];

\node[below] at (3.564,-0.1)			{$\tz_0$};

\node[right] at (0.5,1)					{$\Sigma_6^-$};
\node[right] at (2.4, 1)					{$\Sigma_6^+$};
\node[right] at (5, 1.6)					{$\Sigma_1$};

\node[right] at (2, 0.4)				{$\Omega_8^+$};

\node[right] at (4.5, 0.2)				{$\Omega_{3, 1}$};
\node[right] at (6, 0.2)				{$\Omega_{3,2}$};

\node[right] at (1, 0.4)				{$\Omega_8^-$};
\node[right] at (5, 1.6)					{$\Sigma_1$};

\draw [dashed]  (1.732, 0)--(1.732, 1);

\end{tikzpicture}
\label{fig:three reigons}
\end{figure}
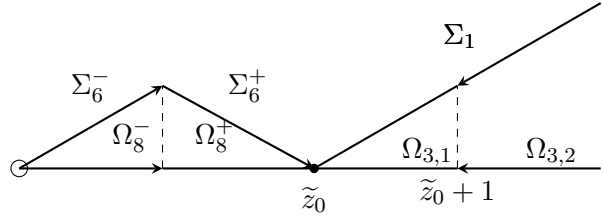
\begin{lemma}
\label{lemma:dbar.R.bd}
\begin{equation}
\label{dbar.R_1.bd}
\left| \dbar R_1 e^{ i\widetilde{\theta}}  \right|
\lesssim		
\begin{cases}
 	\left( |p_1'(\Real (\lambda))| + |\lambda - \widetilde{z}_0 |^{-1/2} + \left\vert \dbar \left( \Xi_\calZ(\lambda) \right)  \right\vert \right) e^{-|u||v|\tau };  \lambda\in \Omega_{3,1} \\
 	\\
 	\left( |p_1'(\Real (\lambda))| + |\lambda - \widetilde{z}_0 |^{-1/2} + \left\vert \dbar \left( \Xi_\calZ(\lambda) \right)  \right\vert \right) e^{-|v|\tau };  \lambda\in \Omega_{3,2},
		\end{cases}
\end{equation}
\begin{equation}
\label{dbar.R_6.bd}
\left| \dbar R_6 e^{- i\widetilde{\theta}}  \right|
\lesssim		
\begin{cases}
 	\left( |p_6'(\Real (\lambda))| +\dfrac{p_6(\Real (\lambda))}{ |\lambda  | }+ \left\vert \dbar \left( \Xi_\calZ(\lambda) \right)  \right\vert \right)e^{-|v| \tau / u^2 }; \, \lambda \in \Omega_8^-\\
 	\\
 	\left( |p_6'(\Real (\lambda))| + |\lambda - \widetilde{z}_0 |^{-1/2} + \left\vert \dbar \left( \Xi_\calZ(\lambda) \right)  \right\vert \right) e^{-|u |  | v|\tau }; \, \lambda \in \Omega_8^+.
		\end{cases}
\end{equation}
\end{lemma}
\begin{proof}
 In $\Omega_3$, set $\lambda=(u + \tz_0)+iv$ and $0\leq v<u$.
\begin{align*}
\text{Re}(i\ttheta)  &= \dfrac{1}{2}( \text{Im} \lambda ) \left( \dfrac{1}{1+\tz_0^2}  \right) \left(   \dfrac{\tz_0^2}{(\text{Re}\lambda )^2 + (\text{Im}\lambda )^2 } -1\right)t \\ 
                            & =\dfrac{1}{2} v \left( \dfrac{1}{1+\tz_0^2}  \right) \left(   \dfrac{ -u^2-2\tz_0 u -v^2}{ ( u+\tz_0)^2+ v^2 } \right)t\\
                            &\leq -\dfrac{1}{2}v \dfrac{1}{1+\tz_0^2}\dfrac{u^2+2\tz_0 u}{( u+\tz_0)^2+v^2}t
                           \end{align*}
 Near $\tz_0$ we obtain
 \begin{equation}
 \label{omega8}
 \text{Re}(i\ttheta) \leq - \dfrac{1}{1+\tz_0^2}\dfrac{\tz_0 uv}{( u+\tz_0)^2+v^2}t \lesssim -|u||v|\tau
 \end{equation}
and away from $\tz_0$, we take
\begin{equation}
 \label{omega8'}
 \text{Re}(i\ttheta) \leq -\dfrac{1}{2}v \dfrac{1}{1+\tz_0^2}\dfrac{u^2 }{( u+\tz_0)^2+v^2}t \lesssim -|v|\tau.
 \end{equation}
In $\Omega_8^-$, setting $\lambda=u+iv$ and using the facts that $u\geq 0$, $v\geq 0$ and $u^2+v^2\leq \tz_0^2/3$ we deduce
\begin{align*}
-\text{Re}(i\ttheta)  &=- \dfrac{1}{2}( \text{Im} \lambda ) \left( \dfrac{1}{1+\tz_0^2}  \right) \left(   \dfrac{\tz_0^2}{(\text{Re}\lambda )^2 + (\text{Im}\lambda )^2 } -1\right)t \\ 
                            & =-\dfrac{1}{2} v \left( \dfrac{1}{1+\tz_0^2}  \right) \left(   \dfrac{\tz_0^2 }{ u^2+ v^2 } -1\right)t\\
                            &\leq -\dfrac{\tz_0^2vt}{6(1+\tz_0^2)u^2} .\\
                            \end{align*}
Finally, in $\Omega_8^+$, we set $\lambda=(u + \tz_0)+iv$ and notice that $- \tz_0/2 <u<0$ and $|u|>|v|$
\begin{align*}
-\text{Re}(i\ttheta)  &=- \dfrac{1}{2}( \text{Im} \lambda ) \left( \dfrac{1}{1+\tz_0^2}  \right) \left(   \dfrac{\tz_0^2}{(\text{Re}\lambda )^2 + (\text{Im}\lambda )^2 } -1\right)t \\ 
                            & =\dfrac{1}{2} v \left( \dfrac{1}{1+\tz_0^2}  \right) \left(   \dfrac{ u^2 + 2\tz_0 u + v^2}{ ( u+\tz_0)^2+ v^2 } \right)t\\
                            &\leq \dfrac{3\sqrt{3}-1}{4\sqrt{3}} \dfrac{\tz_0}{1+\tz_0^2}\dfrac{ uv }{( u+\tz_0)^2+v^2}t\\
                            & \lesssim - |u||v|\tau.
                           \end{align*}
Estimates \eqref{dbar.R_1.bd} and  \eqref{dbar.R_6.bd}  then follow from Lemma \ref{dbar.Ri}. The quantities $p_1'(\Real \lambda)$  and $p_6'(\Real \lambda)$   are all bounded uniformly for  $\tdr$ in a bounded subset of $H^{1,1}_0(\bbR)$.  

\end{proof}

\begin{lemma}
\label{lemma:RHP.bd}For the localized Riemann-Hilbert problem from Problem \ref{MTM.RHP.local}, we have
\begin{align}
\label{RHP.bd1}
\norm{\tm^\RHP(\dotarg; x, t)}{\infty}	&	\lesssim		1,\\[5pt]
\label{RHP.bd2}
\norm{\tm^\RHP(\dotarg; x,t )^{-1}}{\infty}	&	\lesssim	1.
\end{align}
All implied constants are uniform  for  $\tdr$ in a bounded subset of $H^{1,1}_0(\bbR)$.
\end{lemma}

The proof of this lemma is a consequence of the previous section.

\begin{lemma}
\label{lemma:KW}
Suppose that $r\in H^{1,1}_0(\bbR)$.
Then, the estimate \eqref{dbar.int.est1}
holds, where the implied constant is uniform  for  $\tdr$ in a bounded subset of $H_0^{1,1}(\bbR)$.
\end{lemma}

\begin{proof}
To prove \eqref{dbar.int.est1}, first note that
\begin{align}
 \norm{K_W f}{\infty} &\leq \norm{f}{\infty} \int_\bbC \frac{1}{|\lambda-\zeta|}|W(\zeta)| \, dm(\zeta) 
                                \end{align}
so that we need only estimate the right-hand integral. We will prove the estimate in the region $ \lambda\in\Omega_3$ first. From \eqref{W-bound}, it follows
$$ |W(\zeta)| \leq \norm{\tm^{\RHP}}{\infty} \norm{(\tm^{\RHP})^{-1}}{\infty} \left| \dbar R_1\right| |e^{i\ttheta}|.$$
Setting $\lambda=\alpha+i\beta$ and $\zeta=(u+\tz_0)+iv$, the region $\Omega_{3,1}$ corresponds to $u\geq \sqrt{3}v \geq 0 $. We then have from \eqref{dbar.R_1.bd} \eqref{RHP.bd1}, and \eqref{RHP.bd2} that
$$
 \int_{\Omega_{3,1}}  \frac{1}{|\lambda-\zeta|} |W(\zeta)| \, d\zeta  \lesssim  I_1 + I_2 +I_3
$$
where
\begin{align*}
I_1 	&=	\int_0^{1/\sqrt{3}} \int_{\sqrt{3} v}^1 \frac{1}{|\lambda-\zeta|} |p_1'(u)| e^{- uv\tau } \, du \, dv, \\[5pt]
I_2	&=	\int_0^{1/\sqrt{3}} \int_{\sqrt{3} v}^1 \frac{1}{|\lambda-\zeta|} \left| u+iv \right|^{-1/2} e^{- uv\tau } \, du \, dv,\\
I_3	&=	\int_0^{1/\sqrt{3}} \int_{\sqrt{3}v}^1 \frac{1}{|\lambda-\zeta|} \left|  \dbar (\Xi_\calZ(\zeta  ) ) \right| e^{- uv\tau } \, du \, dv.
\end{align*}
It now follows from \cite[proof of Proposition D.1]{BJM} that
$$
|I_1|, \, |I_2|, \, |I_3| \lesssim (\tau)^{-1/4}.
$$
It then follows that
\begin{equation}
\label{omega31}
\int_{\Omega_{3,1}} \frac{1}{|\lambda-\zeta|} |W(\zeta)| \, d\zeta \lesssim (\tau)^{-1/4}
\end{equation}
Similar estimates hold for the integrations over the remaining $\Omega_8^+$:
\begin{equation}
\label{omega8+}
 \int_{\Omega_{8}^+} \frac{1}{|\lambda-\zeta|} |W(\zeta)| \, d\zeta \lesssim (\tau)^{-1/4}. 
\end{equation}
Now we turn to region $\Omega_8^-$ and write
$$
\int_{\Omega_{8}^-}  \frac{1}{|\lambda-\zeta|} |W(\zeta)| \, d\zeta  \lesssim  I_1 + I_2 +I_3
$$
where
\begin{align*}
I_1 	&=	\int_0^{\tz_0/(2\sqrt{3} )} \int_{\sqrt{3} v}^{\tz_0/2} \frac{1}{|\lambda-\zeta|} |p_6'(u)| e^{-v \tau / u^2} \, du \, dv, \\[5pt]
I_2	&=	\int_0^{\tz_0/(2\sqrt{3})} \int_{\sqrt{3} v}^{\tz_0/2}  \frac{1}{|\lambda-\zeta|} \left| u+iv \right|^{-1/2} e^{-v \tau / u^2}\, du \, dv,\\
I_3	&=	\int_0^{\tz_0 / (2\sqrt{3}) } \int_{\sqrt{3}v}^{\tz_0/2}  \frac{1}{|\lambda-\zeta|} \left|  \dbar (\Xi_\calZ(\zeta  ) ) \right|e^{-v \tau / u^2} \, du \, dv.
\end{align*}
Setting $\lambda=\alpha+i\beta$ and using the fact that
$$\norm{\dfrac{1}{\lambda-\zeta}}{L^2_u(v, \infty)}^2\leq \dfrac{\pi}{|v-\beta|}$$
we have the estimate
\begin{align*}
|I_1| &\lesssim \int_0^\infty e^{-v\tau} \int_{\sqrt{3} v}^\infty \dfrac{p'_6(u)}{|\zeta-\lambda |}dudv\\
&\lesssim \norm{r}{H^1}\int_0^\infty \dfrac{e^{-v\tau}}{|v-\beta|^{1/2}} dv\\
&\lesssim \tau^{-1/2}.
\end{align*}
For $I_2$, we choose $p>2$ and use \textit{H\"older}'s inequality to bound
\begin{align*}
|I_2|  & \lesssim \int_0^\infty e^{-v\tau} v^{1/p-1/2}|v-\beta|^{1/q-1}dv\\
        &=  \int_0^\beta e^{-v\tau} v^{1/p-1/2}|v-\beta|^{1/q-1}dv +  \int_\beta^\infty e^{-v\tau} v^{1/p-1/2}|v-\beta|^{1/q-1}dv\\
        &\leq \int_0^1\beta^{1/2}e^{-\beta w\tau}w^{1/p-1/2}(1-w)^{1/q-1}dw +\int_\beta^\infty e^{-v\tau} v^{1/p-1/2}|v-\beta|^{1/q-1}dv\\
        &\lesssim \tau^{-1/2} \int_{0}^1 w^{1/p-1}(1-w)^{1/q-1}dw+\int_0^\infty e^{-\tau w}w^{-1/2}dw\\
        &\lesssim \tau^{-1/2}.
\end{align*}
The estimate on $I_3$ is similar to that of $I_1$.
So we arrive at
\begin{equation}
\label{omega8-}
\int_{\Omega_{8}^-}  \frac{1}{|\lambda-\zeta|} |W(\zeta)| \, d\zeta  \lesssim \tau^{-1/2}.
\end{equation}

Similar procedure gives the following estimate on $\Omega_{3,2}$
\begin{equation}
\label{omega32}
\int_{\Omega_{3,2}}  \frac{1}{|\lambda-\zeta|} |W(\zeta)| \, d\zeta  \lesssim \tau^{-1/2}.
\end{equation}
Combining \eqref{omega31}-\eqref{omega32} leads to \eqref{dbar.int.est1}.
\end{proof}

\begin{lemma}
\label{lemma:N31.est}
The estimate   \eqref{N31.est} 
 holds with constants uniform in $\tdr$ in a bounded subset of $H_0^{1,1}(\bbR)$.
\end{lemma}

\begin{proof}
Recall from \eqref{N3.1} that
\begin{equation}
\label{N3.1}
\tm^{(3)}_1(x,t) = I+\dfrac{1}{\pi} \int_{\bbC}  \dfrac{  \tm^{(3)}(\zeta;x,t) W(\zeta;x,t)}{\zeta} \, d\zeta . 
\end{equation}
This combined with Lemma \ref{lemma:KW} implies
$$ \left|\tm^{(3)}_1(x,t) \right| \lesssim I+ \int_\bbC \dfrac{ |W(\zeta;x,t)|}{|\zeta|} \, d\zeta. $$
We will bound this integral by $\tau^{-3/4}$ modulo constants with the required uniformities. Again we only work with $\Omega_8$ and $\Omega_3$ as shown in Figure
\ref{fig:three reigons}. In $\Omega_8^-$ we have
$$
\int_{\Omega_8^-} \dfrac{ |W(\zeta;x,t)| }{|\zeta|}\, d \zeta	\lesssim  I_1+I_2 +I_3
$$
where
\begin{align*}
I_1 	&=	\int_0^{\tz_0/(2\sqrt{3} )} \int_{\sqrt{3} v}^{\tz_0/2} \dfrac{1}{|\zeta|} |p_6'(u)| e^{-v \tau / u^2} \, du \, dv, \\[5pt]
I_2	&=	\int_0^{\tz_0/(2\sqrt{3})} \int_{\sqrt{3} v}^{\tz_0/2}  \dfrac{1}{|\zeta|^2 }|p_6(u)|  e^{-v \tau / u^2} \, du \, dv,\\
I_3	&=	\int_0^{\tz_0 / (2\sqrt{3}) } \int_{\sqrt{3}v}^{\tz_0/2}  \dfrac{1}{|\zeta|} \left|  \dbar (\Xi_\calZ(\zeta  ) ) \right| e^{-v \tau / u^2} \, du \, dv.
\end{align*}
For $I_1$ applying the Cauchy-Schwarz inequality, we get
\begin{align*}
I_1 	&=	\int_0^{\tz_0/(2\sqrt{3} )} \int_{\sqrt{3} v}^{\tz_0/2} \dfrac{1}{|\zeta|} |p_6'(u)| e^{-v \tau / u^2} \, du \, dv\\
        &\leq \int_0^{\infty}  \left( \int_0^{\tz_0/2} {|{p}'_6(u)|^2}du  \right)^{1/2} \left( \int_{\sqrt{3} v}^{\tz_0/2} \dfrac{ e^{-2v \tau / u^2} }{ u^2  } du  \right)^{1/2} dv\\
        &\lesssim  \int_0^{\infty}  \norm{\tdr}{H^{1}(\bbR)}  \left( \int_{{2/\tz_0} }^{\infty} {e^{-4vw^2\tau}} dw  \right)^{1/2} dv\\
        &  \lesssim  \int_0^{\infty}   \left( \int_{{2/\tz_0} }^{\infty} {e^{-4vw\tau}} dw  \right)^{1/2} dv\\
        & \leq \tau^{-1}.
\end{align*}
Similarly for $I_2$, we have:
\begin{align}
\label{est-r-0}
I_2	 &= \int_0^{\tz_0/(2\sqrt{3} )}  \int_{\sqrt{3} v}^{\tz_0/2} \dfrac{ |{p}_6(\zeta)|}{ |u^2+v^2|^{1/2}  }\dfrac{e^{-v\tau/u^2}}{| u^2+v^2  |^{1/2}  } \, du \, dv\\
\nonumber
   &\lesssim \norm{\tdr}{H^{1}(\bbR)} \int_0^{\tz_0/2} \dfrac{1}{u^2} \int_{ 0}^{u/\sqrt{3}} e^{-v\tau/u^2} dv   du\\
   \nonumber
    & \lesssim \tau^{-1}.
 \end{align}
 Notice that we have shown that $\tdr(\lambda)/\lambda$ is bounded near the origin in proposition \ref{prop:r}. Thus $p_6(u)/u$ is bounded on the interval $(0, \tz_0/2)$. And $I_3\lesssim \tau^{-1}$
 follows from the same argument as that of $I_1$. So we establish
 \begin{equation}
 \label{omega8-'}
 \int_{\Omega_8^-} \dfrac{ |W(\zeta;x,t)| }{|\zeta|}\, d \zeta	 \lesssim \tau^{-1}.
 \end{equation}
 Given the fact that $|\zeta|\geq \tz_0/2 $, it then follows from \cite[Proposition D.2]{BJM} that
\begin{align}
 \int_{\Omega_8^+} \dfrac{ |W(\zeta;x,t)| }{|\zeta|}\, d \zeta	  &\lesssim \tau^{-1}, \\
  \int_{\Omega_{3,1}} \dfrac{ |W(\zeta;x,t)| }{|\zeta|}\, d \zeta	 &\lesssim \tau^{-3/4}.
\end{align}
 We finally turn to region $\Omega_{3,2}$. In this region,
$$
\int_{\Omega_{3,2}} \dfrac{ |W(\zeta;x,t)| }{|\zeta|}\, d \zeta	\lesssim  I_1+I_2 +I_3
$$
where
\begin{align*}
I_1 	&=	\int_0^{\infty} \int_{\tz_0+1}^{\infty} \dfrac{1}{|\zeta|} |p_6'(u)| e^{-v\tau} \, du \, dv, \\[5pt]
I_2	&=	\int_0^{\infty} \int_{\tz_0+1}^{\infty}\dfrac{1}{|\zeta-\tz_0|^{-1/2}} \dfrac{1}{|\zeta|} e^{- v\tau } \, du \, dv,\\
I_3	&=	\int_0^{\infty} \int_{\tz_0+1}^{\infty} \dfrac{1}{|\zeta|} \left|  \dbar (\Xi_\calZ(\zeta  ) ) \right| e^{- v\tau} \, du \, dv.
\end{align*}
For $I_1$ applying the Cauchy-Schwarz inequality, we obtain
\begin{align*}
I_1 &=\int_0^{\infty} \int_{\tz_0+1}^{\infty} \dfrac{1}{\sqrt{u^2+v^2}} |p_6'(u)| e^{-v\tau} \, du \, dv,\\
     &\lesssim  \int_0^{\infty} e^{-v\tau}  dv\\
     & \lesssim \tau^{-1}.
 \end{align*}
 Similarly,for $I_2$ we have
 \begin{align*}
I_2 &=\int_0^{\infty} \int_{\tz_0+1}^{\infty} \dfrac{1}{\sqrt[4]{(u-\tz_0)^2+v^2}}\dfrac{1}{\sqrt{u^2+v^2}} e^{-v\tau} \, du \, dv,\\
     &\lesssim  \int_0^{\infty} e^{-v\tau} \left( \int_1^\infty \dfrac{1}{u^{3/2}} du \right) dv\\
     & \lesssim \tau^{-1}.
 \end{align*}
 And $I_3\lesssim \tau^{-1}$
 follows from the same argument as that of $I_1$. So we establish
 \begin{equation}
 \label{omega32'}
 \int_{\Omega_{3,2}} \dfrac{ |W(\zeta;x,t)| }{|\zeta|}\, d \zeta	 \lesssim \tau^{-1}.
 \end{equation}
 Combining \eqref{omega8-'}-\eqref{omega32'} we arrive at the inequality in \eqref{N31.est}.
\end{proof}

\section{Long-Time Asymptotics inside the Light Cone}
\label{sec:large-time}
We now put together our previous results and formulate the long-time asymptotics of $u(x,t)$ and $v(x,t)$  inside the light cone.  Undoing all transformations we carried out previously,  we get back $\widetilde{M}$:
\begin{equation}
\label{N3.to.N}
\tM(\lambda;x,t) = \tm^{(3)}(\lambda;x,t) \tm^\RHP(\lambda; \tz_0) \left( \calR^{(2)}(\lambda)\right)^{-1} \delta(\lambda)^{\sigma_3}.
\end{equation}
 In the previous sections, we have proven the following results.
\begin{lemma}
\label{lemma:N.to.NRHP.asy}
For $\lambda=i\sigma$ and $\sigma \rarr 0^+$, the following asymptotic relations hold
\begin{align}
\label{N.asy}
\tm^{(3)}(0; x,t) 				&=	I +\mathcal{O}\left( \tau^{-3/4}\right) \\
\label{expan-mLC}
\tm^\RHP (0; x, t)  &= E (0; x, t)\tm^{(sol)}(0; x, t)\\
 \nonumber
             &= \left( I+\mathcal{O} \left( \tau^{-1/2}\right) \right) \tm^{(sol)}(0; x, t)\\
             \label{delta.sigma.asy} 
\delta(0)^{\sigma_3}  &= \twomat{(-1)^l}{0}{0}{ (-1)^{-l} }     \\
\nonumber
          \calR^{(2)}(\lambda) &=I.
\end{align}

\end{lemma}
\begin{lemma}
\label{lemma:N.to.NRHP.asy}
For $\lambda=i\sigma$ and $\sigma \rarr +\infty$, the asymptotic relations
\begin{align}
\label{N.asy}
\tm(\lambda;x,t) 				&=	I + \frac{1}{\lambda} \tm_1(x,t) + o\left(\frac{1}{\lambda}\right)\\
\label{N.RHP.asy}
\tm^\RHP(\lambda;x,t)		&=	I + \frac{1}{\lambda} \tm^\RHP_1(x,t) + o\left(\frac{1}{\lambda}\right)
\end{align}
hold. Moreover,
\begin{equation}
\label{N.to.NRHP.asy} 
\left(\tm_1(x,t)\right)_{12} = \left(\tm^\RHP_1(x,t)\right)_{12} + \bigO{\tau^{-3/4}}.
\end{equation}
\end{lemma}

\begin{proof}
By Lemma \ref{delta} (3), 
the expansion
\begin{equation}
\label{delta.sigma.asy} 
\delta(\lambda)^{\sigma_3} = \twomat{1}{0}{0}{1} + \frac{1}{\lambda} \twomat{\delta_1}{0}{0}{\delta_1^{-1}} + \bigO{\lambda^{-2}} 
\end{equation}
holds, with the remainders in \eqref{delta.sigma.asy}
 uniform in $\tdr$ in a bounded subset of $H^{1}$. 
\eqref{N.asy} follows from \eqref{N3.to.N},  \eqref{N.RHP.asy}, the fact that $\calR^{(2)} \equiv I$ in $\Omega_2$,
and \eqref{delta.sigma.asy}.
Notice the fact that the 
diagonal matrix in \eqref{delta.sigma.asy} does not affect the $12$-component of $\tm$. Hence, for 
$\lambda=i\sigma$, the equality
$$ 
\left(\tm(\lambda;x,t)\right)_{12} = 
		\frac{1}{\lambda}\left(\tm^{(3)}_1(x,t)\right)_{12} + 
		\frac{1}{\lambda}\left(\tm^\RHP_1(x,t)\right)_{12} + o\left(\frac{1}{\lambda}\right)
$$
holds and result now follows from \eqref{N31.est}. The error term results from similar estimates in section\ref{sec:dbar}.
\end{proof}

With the two lemmas above,  we arrive at the asymptotic formula under the reference frame of a given soliton:
$$\mathrm{v}_\ell=\dfrac{x}{t}=\dfrac{1-\rho_\ell^2}{1+\rho_\ell^2}$$
where $\rho_\ell$ corresponds to the $\ell$th soliton (see definition \eqref{data}).
\begin{proposition}
\label{lemma br.RHP.asy}
The function 
\begin{equation}
\label{cos.recon.bis}
u(x,t) = \overline{\left[ \tm (0 ;x,t)\right]}_{12}
\end{equation}
takes the form 
$$ u(x,t)= u^{sol}(x,t)+  u_{as}(x,t)+\mathcal{O}\left( \tau^{-3/4}\right), $$
where  $u^{sol}(x,t)$ is given by \eqref{u-sol} and
\begin{align}
\label{cos-asym}
\overline{u_{as}(x,t)}&=    \dfrac{(-1)^{\ell}e^{-\chi(0)}}{   \sqrt{\dfrac{t  \tz_0 }{ 1+\tz_0^2 } } }\left\lbrace -\left[\tm_{11}^{(sol)}( -\tz_0)\right]^2 \left( i ( \delta^0_A)^{-2} \overline{\beta}_{12} \right) - \left[\tm_{12}^{(sol)} (-\tz_0)\right]^2 \left(  i \left( \delta^0_A\right)^{2}\overline{\beta}_{21} \right)  \right. \\
\nonumber
&\quad   \left. \left[ \tm_{11}^{(sol)} (\tz_0)\right]^2 \left( i ( \delta^0_B)^{-2} {\beta}_{12} \right)  + \left[ \tm_{12}^{(sol)}(\tz_0)\right]^2 \left( i ( \delta^0_B)^{2}{\beta}_{21} \right)\right \rbrace \tm_{22}^{(sol)} (0; x,t)\\
\nonumber
& \quad +\dfrac{(-1)^{\ell} e^{\chi(0)}}{   \sqrt{\dfrac{t  \tz_0 }{ 1+\tz_0^2 } } }\left\lbrace \left[\tm_{12}^{(sol)}( -\tz_0)\tm_{22}^{(sol)}( -\tz_0)\right] \left( i ( \delta^0_A)^{-2} \overline{\beta}_{12} \right) + \left[\tm_{11}^{(sol)} (-\tz_0)\tm_{21}^{(sol)} (-\tz_0)\right] \left(  i \left( \delta^0_A\right)^{2}\overline{\beta}_{21} \right)  \right. \\
\nonumber
&\quad   \left.- \left[\tm_{12}^{(sol)}( \tz_0)\tm_{22}^{(sol)} (\tz_0)\right] \left( i ( \delta^0_B)^{-2} {\beta}_{12} \right) - \left[ \tm_{11}^{(sol)}( \tz_0)\tm_{21}^{(sol)} (\tz_0)\right]^2 \left( i ( \delta^0_B)^{2}{\beta}_{21} \right)\right \rbrace \tm_{12}^{(sol)} (0; x,t)
\end{align}
Also recall that
\begin{align}
 \overline{v(x)} e^{\frac{i}{2} \int_{x}^{+\infty}\left(|u|^{2}+|v|^{2}\right) d y} &=\lim _{|\lambda| \rightarrow \infty} \lambda[\widetilde{M}(x ; \lambda)]_{12}.
\end{align}
so 
\begin{equation}
\label{sin.recon.bis}
v(x,t) =  \lim _{|\lambda| \rightarrow \infty} \overline{\lambda[\widetilde{M}(x ; \lambda)]_{12}}\left[\tm(x,0)  \right]_{11}.
\end{equation}
takes the form 
$$ v(x,t)= v^{sol}( x,t)+  v_{as}(x,t)+\mathcal{O}\left( \tau^{-3/4}\right) $$
with $v^{sol}(x,t)$ given by \eqref{v-sol}. And
\begin{align}
\label{v-asym}
\overline{v_{as}(x, t)}&= \dfrac{\overline{\tm_{12}^{sol}(x,t)}}{   \sqrt{\dfrac{t  \tz_0 }{ 1+ \tz_0^2 } }  }\left\lbrace -\left[ \tm_{11}^{(sol)}( -z_0)\right]^2 \left( i ( \delta^0_A)^{-2} \overline{\beta}_{12} \right) - \left[\tm_{12}^{( sol )} (-z_0)\right]^2 \left(  i ( \delta^0_A)^{2}\overline{\beta}_{21} \right)  \right. \\
\nonumber
&\quad   \left.+ \left[ \tm_{12}^{(sol)} (z_0)\right]^2 \left( i ( \delta^0_B)^2 {\beta}_{12} \right)  +\left[ \tm_{11}^{( sol )}(z_0)\right]^2 \left( i ( \delta^0_B)^{-2}{\beta}_{21} \right) \right\rbrace {\tm_{21}^{( sol )}(0; x, t)} (-1)^{\ell}e^{\chi(0)} \\
\nonumber
  &\quad +  \dfrac{1}{   \sqrt{\dfrac{t  }{ (1+\tz_0^2)\tz_0  } } }\left\lbrace \left[\tm_{11}^{(sol)}( -\tz_0)\right]^2 \left( i ( \delta^0_A)^{-2} \overline{\beta}_{12} \right) + \left[\tm_{12}^{(sol)} (-\tz_0)\right]^2 \left(  i \left( \delta^0_A\right)^{2}\overline{\beta}_{21} \right)  \right. \\
\nonumber
&\quad   \left.- \left[ \tm_{11}^{(sol)} (\tz_0)\right]^2 \left( i ( \delta^0_B)^{-2} {\beta}_{12} \right) - \left[ \tm_{12}^{(sol)}(\tz_0)\right]^2 \left( i ( \delta^0_B)^{2}{\beta}_{21} \right)\right \rbrace \overline{\tm_{11}^{(sol)} (0; x,t)}(-1)^{\ell}e^{\chi(0)}\\
\nonumber
& \quad +\dfrac{\overline{\tm_{12}^{sol}(x,t)}}{   \sqrt{\dfrac{t  \tz_0 }{ 1+\tz_0^2 } } }\left\lbrace \left[\tm_{12}^{(sol)}( -\tz_0)\tm_{22}^{(sol)}( -\tz_0)\right] \left( i ( \delta^0_A)^{-2} \overline{\beta}_{12} \right)+ \left[\tm_{11}^{(sol)} (-\tz_0)\tm_{21}^{(sol)} (-\tz_0)\right] \left(  i \left( \delta^0_A\right)^{2}\overline{\beta}_{21} \right)  \right. \\
\nonumber
&\quad   \left. -\left[\tm_{12}^{(sol)}( \tz_0)\tm_{22}^{(sol)} (\tz_0)\right] \left( i ( \delta^0_B)^{-2} {\beta}_{12} \right)  - \left[ \tm_{11}^{(sol)}( \tz_0)\tm_{21}^{(sol)} (\tz_0)\right] \left( i ( \delta^0_B)^{2}{\beta}_{21} \right)\right \rbrace {\tm_{11}^{(sol)} (0,x,t)}(-1)^{\ell}e^{\chi(0)}
\end{align}
All the terms above have been obtained in Section \ref{sec:local}.

\end{proposition}
Finally in the solitonless region, all the solitons decay exponentially, so we have:
\begin{proposition}
\label{asy pure}
If we choose the frame  $x=\mathrm{v} t$ with $| \mathrm{v} |< 1$ and $\mathrm{v} \neq ( 1-\rho_j^2)/(1+\rho_j^2) $ for all $1\leq j \leq N_1$, then we have that
\begin{align*}
u(x,t)&=u_{as}(x,t)+ \mathcal{O}\left(\tau^{-3/4}\right)\\
v(x, t)&=v_{as}(x,t)+\mathcal{O}\left(\tau^{-3/4}\right)
\end{align*}
with
\begin{align}
\label{u-as-pure}
\overline{u_{as}(x,t)}&=    \dfrac{(-1)^{\ell}}{   \sqrt{\dfrac{t  \tz_0 }{ 1+\tz_0^2 } } }\left\lbrace - \left( i ( \delta^0_A)^{-2} \overline{\beta}_{12} \right) + \left( i ( \delta^0_B)^{-2} {\beta}_{12} \right)  \right \rbrace \\
\label{v-as-pure}
\overline{v_{as}(x,t)}&=    \dfrac{(-1)^{\ell}}{   \sqrt{\dfrac{t   }{ (1+\tz_0^2)\tz_0 } } }\left\lbrace\left( i ( \delta^0_A)^{-2} \overline{\beta}_{12} \right) - \left( i ( \delta^0_B)^{-2} {\beta}_{12} \right)  \right \rbrace 
\end{align}

\end{proposition}
\begin{figure}[h]
\caption{The Augmented Contour $\Sigma$}
\vspace{.5cm}
\label{figure-Sigma1}
\begin{tikzpicture}[scale=0.9]
 \draw[ thick] (0,0) -- (-3,0);
\draw[ thick] (-3,0) -- (-5,0);
\draw[thick,->,>=stealth] (0,0) -- (3,0);
\draw[ thick] (3,0) -- (5,0);
\node[above] at 		(2.5,0) {$+$};
\node[below] at 		(2.5,0) {$-$};
\node[right] at (3.5 , 2) {$\Gamma_j$};
\node[right] at (3.5 , -2) {$\Gamma_j^*$};
\draw[->,>=stealth] (3.4,2) arc(360:0:0.4);
\draw[->,>=stealth] (3.4,-2) arc(0:360:0.4);
\draw [red, fill=red] (3,2) circle [radius=0.05];
\draw [red, fill=red] (3,-2) circle [radius=0.05];
\node[right] at (5 , 0) {$\bbR$};

\draw (0,0) [dashed] circle [ radius=2.9 ];

\draw [red, fill=red] (0, 1) circle [radius=0.05];
\draw[->,>=stealth] (0.3, 1) arc(360:0:0.3);
\draw [red, fill=red] (0, -1) circle [radius=0.05];
\draw[->,>=stealth] (0.3, -1) arc(0:360:0.3);
\node[above] at 		(0, 1.25) {$\Gamma_k$};
\node[below] at 		(0, -1.25) {$\Gamma_k^*$};
 \node [below] at (2.9,0) {\footnotesize $\tz_0$};
    \node [below] at (-2.9,0) {\footnotesize $-\tz_0$};
     \draw	[fill, red]  (1.529, 1.006)		circle[radius=0.05];	  
     \draw[->,>=stealth] (1.729, 1.006) arc(360:0:0.2);
     \draw	[fill, red]  (1.529, -1.006)		circle[radius=0.05];	  
     \draw[->,>=stealth] (1.729, -1.006) arc(0:360:0.2);
    \node[above]  at (0, 0) {\footnotesize $\text{Re}(i\theta)>0$};
    
      \node[below]  at (0, 0) {\footnotesize $\text{Re}(i\theta)<0$};
    
     \node[above]  at (0, 3) {\footnotesize $\text{Re}(i\theta)<0$};
     \node[below]  at (0, -3) {\footnotesize $\text{Re}(i\theta)>0$};
     \node[above] at 		(1.9, 1.0) {$\Gamma_\ell$};
\node[below] at 		(2.1, -1.0) {$\Gamma_\ell^*$};
\end{tikzpicture}
\begin{center}
\begin{tabular}{ccc}
	
soliton ({\color{red} $\bullet$} ) 
\end{tabular}
\end{center}
\end{figure}
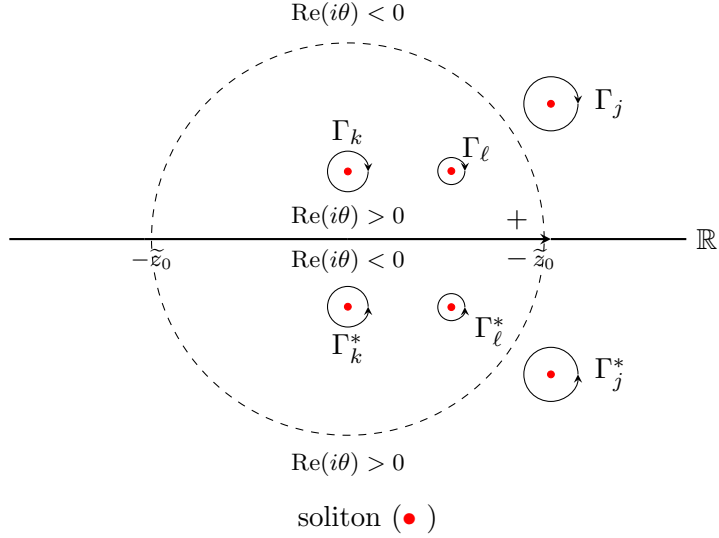

\section{Long-Time Asymptotics Outside the light cone}
\label{sec:outside}
We now turn to the study of the asymptotic behavior when $|x/t|>1$. We first deal with the case $x/t>1$. Our starting point is RHP Problem \ref{RHP2-v}. As it will become clear later, outside the light cone there are only higher order decay terms comparing to the interior of the light cone. For the purpose of brevity we are only going to display the calculations directly related to error terms. 
\subsection{$x/t > 1$}
\label{subsec: out}
Setting $\lambda=u+iv$, we have for $v>0$, 
\begin{equation}
\label{theta-out}
2\text{Re}i\ttheta(\lambda; x, t)=-\left(   1+\dfrac{x}{t} \right)v t +  \left( 1-\dfrac{x}{t} \right)\dfrac{v t}{u^2+v^2}  \leq -\left(   1+\dfrac{x}{t} \right)v t <0
\end{equation}
Similarly for the case $v<0$,
\begin{equation}
\label{theta-out'}
2\text{Re}\left[-i\ttheta(\lambda; x, t)  \right] =\left(   1+\dfrac{x}{t} \right)v t +  \left( \dfrac{x}{t}-1 \right)\dfrac{v t}{u^2+v^2}  \leq \left(   1+\dfrac{x}{t} \right)v t <0.
\end{equation}
Since all the pole conditions  have desired decay properties, we only need the  following upper/lower factorization on $\bbR$:
\begin{equation}
\label{v-ul}
e^{-(i/2)\ttheta\ad\sigma_3}\tv(\lambda)	= \Twomat{1}{0}{-\lambda\overline{\tdr(\lambda)}  e^{-i\ttheta}}{1}\Twomat{1}{-\tdr(\lambda)   e^{i\ttheta}}{0}{1}.
						\quad \
						\lambda \in\bbR 
\end{equation}
and the contour deformation:
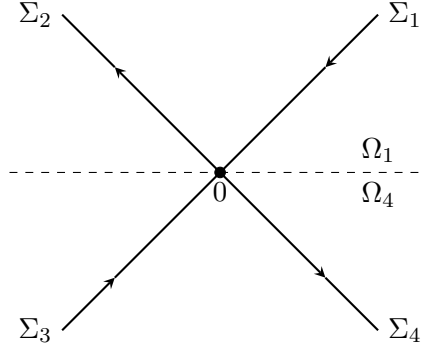
\begin{figure}[H]
\caption{$\Sigma$-outside }
\vskip 15pt
\begin{tikzpicture}[scale=0.7]
\draw[->,thick,>=stealth] 	(3, 3) -- (2,2);						
\draw[thick]		(0,0) -- (2,2);		
\draw[->,thick,>=stealth]  		(0,0) -- (-2,2); 				
\draw[thick]	(-3,3) -- (-2,2);	
\draw[->,thick,>=stealth]		(-3,-3) -- (-2,-2);							
\draw[thick]						(-2,-2) -- (0,0);
\draw[thick,->,>=stealth]		(0,0) -- (2,-2);								
\draw[thick]						(2,-2) -- (3,-3);
\draw	[fill]							(0,0)		circle[radius=0.1];	
\node [below] at (0,0) {0};
\draw [dashed] (0,0)--(4,0);
\draw [dashed] (-4,0)--(0,0);
\node[right] at (3,3)					{$\Sigma_1$};
\node[left] at (-3,3)					{$\Sigma_2$};
\node[left] at (-3,-3)					{$\Sigma_3$};
\node[right] at (3,-3)				{$\Sigma_4$};
\node[above] at (3,0)					{$\Omega_1$};
\node[below] at (3,0)					{$\Omega_4$};
\end{tikzpicture}
\label{fig:Painleve}
\end{figure}
For brevity, we only discuss the situation in $\Omega_1$. In $\Omega_1$, we define
\begin{align*}
	R_1	&=	\begin{cases}
						\twomat{0}{-\tdr(\lambda)  e^{i\ttheta}  }{0}{0}		
								&	\zeta \in (0,\infty)\\[10pt]
								\\
						\twomat{0}{-\tdr( 0 ) e^{i\ttheta (z )  } (1-\Xi_\calZ)  }{0}{0}	
								&	\zeta	\in \Sigma_1
					\end{cases}
	\end{align*}
and the interpolation is given by 
$$ \left( \tdr(0)+ \left( \tdr\left( \text{Re} \lambda \right) -\tdr( 0)  \right) \cos 2\phi  \right)    (1-\Xi_\calZ)=\tdr\left( \text{Re} \lambda \right)   \cos 2\phi     (1-\Xi_\calZ)$$
given $\tdr(0)=0$.
So we arrive at the $\dbar$-derivative in $\Omega_1$ :
\begin{align}
\dbar R_1&= \left[ \left(  \tdr'\left( u\right) \cos 2\phi- 2\dfrac{ \tdr(u  )    }{  \left\vert \lambda \right\vert   } e^{i\phi} \sin 2\phi  \right)  (1-\Xi_\calZ) - \tdr\left( u \right)   \cos 2\phi\, \dbar (\Xi_\calZ( \lambda ) )\right] e^{i\ttheta}
\end{align}
\begin{equation}
\label{R1.bd1}
|W|=\left| \dbar R_1  \right| 	\lesssim\left( |   \tdr'\left( u  \right) | +\dfrac{|\tdr(u)| }{  |\lambda| } + \dbar (\Xi_\calZ( \lambda ) )  \right) e^{\text{Re}i\ttheta(\lambda; x, t) }.
\end{equation}
We proceed as in the previous section and study the integral equation related to the $\dbar$ problem. Setting $\lambda=\alpha+i\beta$ and $\zeta=u+iv$, the region $\Omega_1$ corresponds to $u\geq v \geq 0 $. We decompose the integral operator into three parts:
$$
 \int_{\Omega_1}  \dfrac{1}{|\lambda -\zeta|} |W(\zeta)| \, d\zeta  \lesssim  I_1 + I_2 +I_3
$$
where
\begin{align*}
I_1 	&=	\int_0^\infty \int_v^\infty \dfrac{1}{|\lambda-\zeta|} \left\vert  \tdr'\left( u  \right)   \right\vert e^{-vt} \, du \, dv, \\[5pt]
I_2	&=	\int_0^\infty \int_v^\infty \frac{1}{|\lambda-\zeta|}  \dfrac {1} {  \left| u  +iv \right|^{1/2} } e^{-vt} \, du \, dv,\\[5pt]
I_3 	&=	\int_0^\infty \int_v^\infty \dfrac{1}{|\lambda-\zeta|} \left\vert  \dbar (\Xi_\calZ(\zeta) )   \right\vert e^{-v t} \, du \, dv. 
\end{align*}
Following the argument of proving \eqref{omega8-}, we obtain
$$
 \int_{\Omega_1}  \dfrac{1}{|\lambda-\zeta|} |W(\zeta)| \, d\zeta  \lesssim  t^{-1/2}.
$$
We now  calculate the decay rate of the integral
$$
 \int_{\Omega_1} \dfrac{ |W(\zeta)|}{|\zeta|} \, d\zeta
.$$
Again we decompose the integral above into three parts
\begin{align*}
I_1 	&=	\int_0^{ \infty } \int_{ v}^{\infty} \dfrac{1}{|\zeta|} |\tdr'(u)| e^{\text{Re}i\ttheta(\zeta; x, t)} \, du \, dv, \\[5pt]
I_2	&=	\int_0^{ \infty } \int_{ v}^{\infty} \dfrac{1}{|\zeta|^2 }|\tdr(u)|  e^{\text{Re}i\ttheta(\zeta; x, t) } \, du \, dv,\\
I_3	&=	\int_0^{ \infty } \int_{ v}^{\infty} \dfrac{1}{|\zeta|} \left|  \dbar (\Xi_\calZ(\zeta  ) ) \right| e^{\text{Re}i\ttheta(\zeta; x, t) } \, du \, dv.
\end{align*}

For $I_1$ one has
\begin{align*}
I_1 	&\lesssim \int_0^{ \infty } \int_{ 1}^{\infty} \dfrac{1}{|\zeta|} |{\tdr'}(u) | e^{-v t} \, du \, dv +  \int_0^{ 1 } \int_{ v}^{1} \dfrac{1}{|\zeta|} |{\tdr}(u) | e^{-vt/u^2 } \, du \, dv\\
       &\lesssim t^{-1}.
\end{align*}
Similarly for $I_2$:
\begin{align}
\label{est-r-out}
I_2	 &\lesssim \int_0^{1}  \int_{v}^{1} \dfrac{ |{\tdr}(u)|}{ |u^2+v^2|  }{e^{-v\tau/u^2}}\, du \, dv + \int_0^{\infty}   e^{-v t}   \int_{ 1}^{\infty} \dfrac{ |{\tdr}(u)|}{ |u^2+v^2|  }\, du \, dv\\
\nonumber
   &\lesssim \norm{\tdr}{H^{1}(\bbR)} \int_0^{1} \dfrac{1}{u^2} \int_{ 0}^{u} e^{-v\tau/u^2} dv   du + t^{-1}\\
   \nonumber
    & \lesssim t^{-1}.
 \end{align}
 The estimate on $I_3$ is similar to that of $I_1$.
Following the same procedure given in Section \ref{sec:large-time} we obtain the following asymptotic formulas:
\begin{align}
\label{outside-cos}
u(x,t)  &= \mathcal{O}\left( t^{-1}\right)\\
\label{outside-sin}
v(x,t)    &=\mathcal{O}\left( t^{-1}\right).
\end{align}

\subsection{$x/t<-1$}
First notice that for $v>0$, 
\begin{equation}
\label{-theta-out}
2\text{Re}i\ttheta(\lambda; x, t)=-\left(   1+\dfrac{x}{t} \right)v t +  \left( 1-\dfrac{x}{t} \right)\dfrac{v t}{u^2+v^2}  \geq -\left(   1+\dfrac{x}{t} \right)v t >0
\end{equation}
Similarly for $v<0$,
\begin{equation}
\label{-theta-out'}
2\text{Re}\left[-i\ttheta(\lambda; x, t)  \right] =\left(   1+\dfrac{x}{t} \right)v t +  \left( \dfrac{x}{t}-1 \right)\dfrac{v t}{u^2+v^2}  \geq \left(   1+\dfrac{x}{t} \right)v t >0.
\end{equation}
 Define the scalar function:
\begin{equation}
\label{psi-0}
\psi(\lambda) = \left(\prod_{j=1}^{N_1}\dfrac{\lambda-\overline{\lambda_j}}{\lambda-\lambda_j}\right) \exp \left(  \dfrac{1}{2\pi i}\int_{-\infty}^{\infty}\dfrac{\log( 1+|\tdr(s)|^2)}{s-\lambda} {d\zeta}  \right) .
\end{equation}
It is straightforward to check that if $\tM(z;x,t)$ solves Problem \ref{RHP2-v}, then the new matrix-valued function $\tm^{(1)}(\lambda;x,t)=\tM(\lambda ;x,t)\psi(\lambda)^{\sigma_3}$ has the following jump matrices:
\begin{equation}
e^{(i/2)\ttheta\ad\sigma_3}\tv^{(1)}(\lambda) 	=\Twomat{1}{\dfrac{-\psi_+^2 {\tdr}}{1+\lambda|\tdr|^2} e^{i\ttheta}}{0}{1}\Twomat{1}{0}{\dfrac{-\psi_-^{-2}  \lambda\overline{\tdr}}{1+\lambda|\tdr|^2}  e^{-i\ttheta}}{1},
						\quad \lambda \in\bbR 
\end{equation}

\begin{align}
\label{v-br+}
e^{(i/2)\theta\ad\sigma_3}\tv^{(1)}(\lambda)= 	\begin{cases}
						\twomat{1}{\dfrac{\left[ (1/\psi )'(\lambda)\right]^{-2} }{c_j  (\lambda-\lambda_j)}e^{i\ttheta} }{0}{1}	&	\lambda\in \Gamma_j, \\
						\\
						\twomat{1}{0}{\dfrac{ \left[\psi'( \overline{\lambda_j } )\right]^{-2}}{\overline{c_j}  (\lambda -\overline{\lambda_j })} e^{-i\ttheta}}{1}
							&	\lambda \in \Gamma_j^*.
					\end{cases}
\end{align}

We now see that all entries in \eqref{v-br+} decay exponentially as $t\to\infty$, so we are allowed to reduce the RHP to a problem on $\bbR$ and perform the same contour deformation of the case when $x/t>1$ in Figure \ref{fig:Painleve}. Again for brevity, we only discuss the situation in $\Omega_4$. In $\Omega_4$, we define
\begin{align*}
	R_1	&=	\begin{cases}
						\twomat{0}{ \dfrac{\psi_+^2 \lambda\tdr}{1+\lambda|\tdr|^2} e^{i\ttheta} }{ 0 }{0}		
								&	\lambda \in (0,\infty)\\[10pt]
								\\
						\twomat{0}{ \dfrac{\psi_+^2(0) \overline{\tdr(0)}}{1+|\tdr(0)|^2} e^{i\ttheta}\left(1-\Xi_\calZ\right) }	{0}{0}
								&	\lambda	\in \Sigma_4
					\end{cases}
	\end{align*}
and the interpolation is given by 
$$  \dfrac{ {\tdr(u)}}{1+u|\tdr(u)|^2}  (  \cos 2\phi)  \psi^2(z)    (1-\Xi_\calZ)  e^{-i\ttheta} $$
and consequently
\begin{align}
\dbar R_1&= \left[ \left( \left( \dfrac{ {\tdr(u)}}{1+u|\tdr(u)|^2}  \right)'\cos 2\phi- 2\dfrac{ \tdr(u)}{(1+u|\tdr(u)|^2)|\lambda|} e^{i\phi} \sin 2\phi  \right)  (1-\Xi_\calZ) - \tdr\left( u \right)   \cos 2\phi \dbar (\Xi_\calZ( \lambda ) )\right] e^{-i\ttheta}
\end{align}
\begin{equation}
\label{R1.bd1}
|W|=\left| \dbar R_1  \right| 	\lesssim\left(  \left\vert  \left(  \dfrac{ {\tdr(u)}}{1+u|\tdr(u)|^2} \right)'\right\vert +\dfrac{| {\tdr(u)}| }{  |\lambda|(   1+u|\tdr(u)|^2) } + \dbar (\Xi_\calZ( z) )  \right) e^{- v t}.
\end{equation}
The asymptotic formulas are the same as \eqref{outside-cos}-\eqref{outside-sin}.

\section{Soliton Resolution and Full Asymptotic Stability}\label{sec:summary}
In this section, we put together the results using inverse scattering transform from all the previous sections, and deduce the soliton resolution for MTM with  generic data. Then we use the long-time asymptotics to obtain the full asymptotic stability of reflectionless nonlinear structures.

\subsection{Soliton resolution}
\begin{theorem}
	\label{thm:maindetail}Given the generic initial data $\left(u_{0},v_{0}\right)\in \mathcal{I}\times \mathcal{I}$ in the sense of Definition \ref{def:generic'}. Let $u(x,t)$ and $v(x,t)$ be the solution to MTM \eqref{eq: MTM}, then the solution can be written as the superposition of solitons and  radiation:
	\[
	u(x,t)=\sum_{\ell=1}^{N_{1}} u^{sol}_{\ell }(x,t)+R_{u}(x,t)+\text{Err}(x,t)
	\]
	and
	\[
	v(x,t)=\sum_{\ell=1}^{N_{1}} v^{sol}_{\ell }(x,t)+R_{v}(x,t)+ \text{Err}(x,t)
	\]
	where $u^{sol}_{\ell }(x,t)$, $v^{sol}_{\ell }(x,t)$
are given by \eqref{u-sol} and \eqref{v-sol} respectively. As for the radiation terms, $R_{u}(x,t)$ and $R_{v}(x,t)$ are in different regions given by the following formulas. Set $x=\mathtt{v}t$,
	we obtain that
	\begin{itemize}
		\item[1.] For $\mathtt{v}=\mathtt{v}_{\ell}^{sol}$,  the velocity of the $\ell $th soliton, $R_u(x,t)$ and $R_v(x,t)$ are given by \eqref{cos-asym} and \eqref{v-asym} and $\text{Err}(x,t)=\mathcal{O}(\tau^{-3/4})$.
		\item[2.]  For $\mathtt{v}\neq \mathtt{v}_{\ell}^{sol}$ for any $\ell$, and $|\mathtt{v}|<1$ , $R_u(x,t)$ and $R_v(x,t)$ are given by \eqref{u-as-pure} and \eqref{v-as-pure} and $\text{Err}(x,t)=\mathcal{O}(\tau^{-3/4})$.
		
		\item[3.] For  $\left|\mathtt{v}\right|>1$,  $R_u(x,t)$ and $R_v(x,t)$ are given by 
		\eqref{outside-cos} and \eqref{outside-sin}.
		\item [4.] For $ x/t \to 1^-$ as $t \to \infty $  $R_u(x,t)$ and $R_v(x,t)$ are given by  \eqref{app-cos} and \eqref{app-sin}.
		\item[5.] For $ x/t \to -1^+$ as $t \to \infty $  $R_u(x,t)$ and $R_v(x,t)$ are given by  \eqref{app-cos-} and \eqref{app-sin-}.
	\end{itemize}
\end{theorem}

\subsection{Full asymptotic stability}\label{subsec:fullasym}
To study the asymptotic stability, we first construct the general  reflectionless solution. Suppose we have the following
discrete scattering data
\begin{equation}
S_{D}=\left\{ 0, \left\{ \lambda_{0,j},c_{0,j}\right\} _{j=1}^{N_{1}}\right\} \in H_0^{1,1}(\bbR)\otimes\mathbb{C}^{2N_{1}}.\label{eq:discrete}
\end{equation}
Assume that $\lambda_{0,j}:=\rho_{0,j}e^{i\omega_{0,j}}$ we denote
\[
\mathtt{v}_{0,j}^{sol}=\frac{1-\rho_{0,j}^2}{1+\rho_{0,j}^2}
\]
Moreover, for genericity, we assume that all $v_{0,j}^{sol}$
are different. Then one can construct two pure soliton solutions $\text{U}\left(x,t\right)$, $\text{V}\left(x,t\right)$
using the discrete scattering data \eqref{eq:discrete} via Problem \ref{RHP2-v}
. We consider the asymptotic stability of the sum 
\begin{equation}
\text{U}(x,t)+\text{V}(x,t)=\sum_{j=1}^{N_{1}}u_{j}^{sol}(x,t)+\sum_{j=1}^{N_{1}} v^{sol}_{j}\left(x,t\right)
\end{equation}
provided the solitons in these sums are sufficiently separated. 

\smallskip
With preparations above and Theorem \ref{thm:maindetail}, we state a corollary regarding the asymptotic stability of $\text{U}\left(x,t\right)$ and $\text{V}(x,t)$.
\begin{corollary}
	\label{cor:stabN}Consider the reflectionless solution $\text{U}\left(x,t\right)$ and $\text{V}\left(x,t\right)$
	to the massive Thirring model \eqref{eq: MTM}. Suppose
	\[
	\left\Vert (u_0, v_0)-\left(U(x, 0), V(x, 0)\right)\right\Vert _{\mathcal{I}\times \mathcal{I}}<\epsilon
	\]
	for $0< \epsilon\ll1$ small enough. Let $(u, v)$ be the solution
	to the MTM with the initial data $(u_0, v_0)$, then
	there exist scattering data
	\begin{equation}
	\mathcal{S}=\left\{ \tdr\left(\lambda\right),\left\{ \lambda_{1,j},\tc_{1,j}\right\} _{j=1}^{N_{1}}\right\} \in H^{1,1}_0(\bbR) \otimes\mathbb{C}^{2N_{1}}\label{eq:newscatt}
	\end{equation}
	computed in terms of $(u_0, v_0)$ such that
	\[
	\left\Vert \tdr\right\Vert _{H^{1,1}}+\sum_{j=1}^{N_{1}}\left(\left|\lambda_{0,j}-\lambda_{1,j}\right|+\left|\tc_{0,j}-\tc_{1,j}\right|\right)\lesssim\epsilon.
	\]
	Moreover, with the scattering data $\mathcal{S}$ given by \eqref{eq:newscatt}, one can write the solution
	
	\[
	u\left(x,t\right)=\sum_{j=1}^{N_{1}}u_j^{sol}(x,t)+u_{as}(x,t)
	\]
	and
	\[
	v\left(x,t\right)=\sum_{j=1}^{N_{1}}v_j^{sol}(x,t)+v_{as}(x,t)
	\]
	where $u_j^{sol}(x,t)$, $v_j^{sol}(x,t)$ are reconstructed via solving Problem \ref{prob:MTM.sol} using scattering
	data \eqref{eq:newscatt} respectively. Here again the radiation
	terms $u_{as}\left(x,t\right)$ and $v_{as}\left(x,t\right)$ have the
	asymptotics given by Theorem \ref{thm:maindetail} with scattering data \eqref{eq:newscatt}.
\end{corollary}

\section{well-posedness}
\label{sec:global}
In the direct scattering process, it requires that initial data to be in $\mathcal{I}\times \mathcal{I}$. Here we provide a sketch of the proof of the global well-posedness in this space for the sake of completeness. First of all, we record the well-posedness of the MTM in the energy space.
\begin{theorem}\cite{Candy}
Let $s\geq 0$ and $u_0,v_0\in H^s(\mathbb R)$. There exists a global solution $(u,v)\in C(\mathbb R,H^s(\mathbb R))$ to \eqref{eq: MTM} such that the charge is conserved, so
\begin{eqnarray}
\begin{aligned}
\|u_0\|_{L_{x}^{2}}^{2}+\|v_0\|_{L_{x}^{2}}^{2}=\|u\|_{L_{x}^{2}}^{2}+\|v\|_{L_{x}^{2}}^{2}
\end{aligned}
\end{eqnarray}
for every $t\in \mathbb R$. Moreover, the solution is unique in a subsequence of $C(\mathbb R, L^2_{loc}(\mathbb R))$ and we have continuous dependence on initial data.
\end{theorem}
This theorem can be established by decomposing the solution into an approximately linear component and a component with improved integrability. For detailed proof , see Timothy Candy\cite{Candy}. Next, the following theorem will show the global well-posedness of the massive Thirring model in the weighted energy space.
\begin{theorem}
Given $(u_0,v_0)\in \mathcal{I}\times \mathcal{I}$, there exists a unique solution
\begin{eqnarray}
\begin{aligned}
(u,v)\in C\left(\mathbb R, \mathcal{I}\times \mathcal{I}\right)
\end{aligned}
\end{eqnarray}
to the massive Thirring model with initial data $(u_0,v_0)$.
\end{theorem}
\begin{proof}
Given $u,v\in H^{2}(\mathbb R)$, following the result in Timothy Candy \cite{Candy}, there exists a unique solution $(u,v)\in C(\mathbb R, H^{2}(\mathbb R)\times H^{2}(\mathbb R))$.\\
With unweighted results above, one can pass to the weighted space directly. Given $(u_0,v_0)\in \mathcal{I}\times\mathcal{I}$, following notations above, one can find the solution $(u,v)$ such that $u,v \in H^2(\mathbb R)$. Our goal is to show that $(u,v)$ is in the weighted Sobolev space.
We use the equation and then directing differentiation leads to
\begin{eqnarray}
\begin{aligned}
\dfrac {d}{dt}E_2^{\epsilon}(u,v) \lesssim E_2(u,v)+E_2^{\epsilon}(u,v)+E_1^{\epsilon}(u,v),
\end{aligned}
\end{eqnarray}
where $E_2(u,v)$ and $E_2^{\epsilon}(u,v)$ are defined via
\begin{eqnarray}
\begin{aligned}
E_k^{\epsilon}(u,v)=\int_{\mathbb R}\left(|\partial^k_{x} u|^2+|\partial^k_{x} v|^2\right)x^2e^{-\epsilon |x|}dx
\end{aligned}
\end{eqnarray}
and
\begin{eqnarray}
\begin{aligned}
E_k(u,v)=\int_{\mathbb R}\left(|\partial^k_{x} u|^2+|\partial^k_{x} v|^2\right)dx.
\end{aligned}
\end{eqnarray}
If we differentiate the weighed energy directly and integrate by parts, then we will obtain these first order weighted energies, which can not be controlled by $E_2^{\epsilon}(u,v)$. Thus, the first step is to show
\begin{eqnarray}
\begin{aligned}
\dfrac {d}{dt} E_0^{\epsilon}(u,v)\lesssim E_0(u,v)+ E_0^{\epsilon}(u,v)
\end{aligned}
\end{eqnarray}
Indeed, differentiating the weighed energy and integrating by parts, one has
\begin{eqnarray}
\begin{aligned}
\dfrac {d}{dt} E_0^{\epsilon}(u,v)=&\int_{\mathbb R}\left((|v|^2)_x-(|u|^2)_x\right)x^2e^{-\epsilon|x|}dx\\
=&\int_{\mathbb R}\left(|u|^2-|v|^2\right)\partial_x(x^2e^{-\epsilon|x|})dx\\
\leq&\int_{\mathbb R}\left(|u_0|^2+|v_0|^2\right)dx+\int_{\mathbb R}\left(|u|^2+|v|^2\right)x^2e^{-\epsilon |x|}dx
\end{aligned}
\end{eqnarray}
Gronwall's inequality, we get
\begin{eqnarray}
\begin{aligned}
E_0^{\epsilon}(u,v)\leq C_T E_0^{\epsilon}(u_0,v_0).
\end{aligned}
\end{eqnarray}
The same argument above results in
\begin{eqnarray}\label{E1}
\begin{aligned}
E_1^{\epsilon}(u,v)\leq C_T E_1^{\epsilon}(u_0,v_0).
\end{aligned}
\end{eqnarray}
The next step is to estimate the weighted energy $E_2^{\epsilon}$. Differentiating the weighed modified norm and integrating by parts, one has
\begin{eqnarray}\label{E2}
\begin{aligned}
\dfrac {d}{dt}E_2^{\epsilon}(u,v)&=\int_{\mathbb R}\left( u_{xx}\bar{u}_{xxt}+u_{xxt}\bar{u}_{xx}+v_{xx}\bar{v}_{xxt}+v_{xxt}\bar{v}_{xx}\right)x^2e^{-\epsilon |x|}dx\\
&=\int_{\mathbb R}\left(|u_{xx}|^2-|v_{xx}|^2\right)\partial_x(x^2e^{-\epsilon|x|})dx+2\Im\int_{\mathbb R}\left(u_{xx}(|v|^2\bar{u})_{xx}+v_{xx}(|u|^2\bar{v})_{xx}\right)x^2e^{-\epsilon|x|}dx\\
&\lesssim E_2^{\epsilon}(u,v)+E_2(u,v)+E_1^{\epsilon}(u,v)
\end{aligned}
\end{eqnarray}
Finally, combining \eqref{E1} and \eqref{E2} together, Gronwall's inequality gives
\begin{eqnarray}
\begin{aligned}
E_2^{\epsilon}(u,v)\leq C_T E_2^{\epsilon}(u_0,v_0).
\end{aligned}
\end{eqnarray}
Passing $\epsilon$ to $0$ by the monotone convergence theorem, we get the following result,
\begin{eqnarray}
\begin{aligned}
\int_{\mathbb R}\left(|u_{xx}|^2+|v_{xx}|^2\right)x^2dx\leq C_T \int_{\mathbb R}\left(|u_{0,xx}|^2+|v_{0,xx}|^2\right)x^2dx <\infty,
\end{aligned}
\end{eqnarray}
for $t\in [-T,T]$ for all $T\in \mathbb R$.
\end{proof}

\section{Approaching the light cone}
In this section we discuss the situation when $|x/t|\to 1$ as $t\to \infty$. When $|x/t|$ is close to $1$ enough, we will not have any solitons in this space time region. So in the following section we omit all discrete scattering data for brevity.
\subsection{$0<x/t<1$} 
\label{subsec: app}
We write $\lambda=u+iv$. Notice that for $v>0$ and $u>\sqrt{\tz_0}= \sqrt[4]{(t-x)/(t+x)}$
\begin{align*}
2\text{Re}i\ttheta(\lambda; x, t) &=-\left(   1+\dfrac{x}{t} \right)v t +  \left( 1-\dfrac{x}{t} \right)\dfrac{v t}{u^2+v^2}\\
&  \leq -\left(   1+\dfrac{x}{t} \right)v t  +\left(   1-\dfrac{x}{t} \right) \dfrac{v}{\tz_0} t \\
& \leq -\left(   1+\dfrac{x}{t} \right)v t +\sqrt{1-\dfrac{x^2}{t^2}}vt\\
&\leq -vt.
\end{align*}
Given $\frac{x}{t}\rightarrow1$ which implies
$\tz_0 \to 0$,  \text{as} $t \to \infty$ we only need the  following upper/lower factorization on $\bbR$:
\begin{equation}
\label{v-ul}
e^{(i/2)\ttheta\ad\sigma_3}\tv(\lambda)	= \Twomat{1}{0}{-\lambda\overline{\tdr(\lambda)}  e^{2i\ttheta}}{1}\Twomat{1}{ -\tdr(\lambda)  e^{-2i\ttheta}}{0}{1}
						\quad \lambda \in\bbR .
\end{equation}
We will again perform the contour deformation and write the solution as a product of solutions to a $\dbar$-problem and a ``localized" Riemann-Hilbert problem.
\begin{figure}[H]
\caption{$\Sigma^{(1)}:\text{near the light cone}$}
\vskip 15pt
\begin{tikzpicture}[scale=0.7]
\draw[thick]		(5, 3) -- (4,2);						
\draw[->,thick,>=stealth] 		(2,0) -- (4,2);		
\draw[thick] 			(-2,0) -- (-4,2); 				
\draw[->,thick,>=stealth]  	(-5,3) -- (-4,2);	
\draw[->,thick,>=stealth]		(-5,-3) -- (-4,-2);							
\draw[thick]						(-4,-2) -- (-2,0);
\draw[thick,->,>=stealth]		(2,0) -- (4,-2);								
\draw[thick]						(4,-2) -- (5,-3);
\draw[thick]		(0,0)--(2,0);
\draw[thick,->,>=stealth] (-2,0) -- (0, 0);
\draw	[fill]							(-2,0)		circle[radius=0.1];	
\draw	[fill]							(2,0)		circle[radius=0.1];
\draw [dashed] (2,0)--(6,0);
\draw [dashed] (-6,0)--(-2,0);
\node[below] at (-2,-0.25)			{$-\sqrt{\tz_0}$};
\node[below] at (2,-0.25)			{$\sqrt{\tz_0}$};
\node[right] at (5,3)					{$\Sigma^{(1)}_1$};
\node[left] at (-5,3)					{$\Sigma^{(1)}_2$};
\node[left] at (-5,-3)					{$\Sigma^{(1)}_3$};
\node[right] at (5,-3)				{$\Sigma^{(1)}_4$};
\node[above] at (4.5,0)           {$\Omega_1$};
\node[above] at (-4.5,0)           {$\Omega_2$};
\node[below] at (-4.5,0)           {$\Omega_3$};
\node[below] at (4.5,0)           {$\Omega_4$};
\end{tikzpicture}
\label{fig:contour-scale-1}
\end{figure}
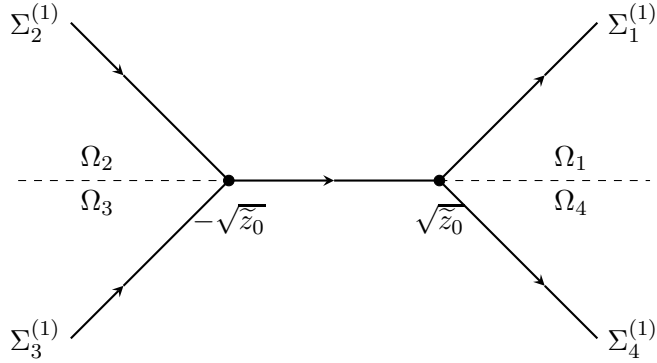
For brevity, we only discuss the $\dbar$-problem in $\Omega_1$.  In $\Omega_1$, we define
\begin{align*}
	R_1	&=	\begin{cases}
						\twomat{0}{-\tdr(\lambda)  e^{i\ttheta}  }{0}{0}		
								&	\lambda \in (\sqrt{\tz_0} ,\infty)\\[10pt]
								\\
						\twomat{0}{-\tdr( \sqrt{\tz_0} ) e^{2i\theta (z )  } }{0}{0}	
								&	\lambda	\in \Sigma_1^{(1)}
					\end{cases}
	\end{align*}
and the interpolation is given by
$$R_1=  \tdr( \sqrt{\tz_0}  )+ \left( {\tdr}\left( \lambda \right) -\tdr\left( \sqrt{\tz_0} \right)  \right) \cos 2\phi  . $$
So we arrive at the $\dbar$-derivative in $\Omega_1$ :
\begin{align}
\dbar R_1&=  \left(  {\tdr}'\left( u\right) \cos 2\phi- 2\dfrac{ {\tdr}\left( u\right)-   \tdr( \sqrt{\tz_0}  )  }{  \left\vert \lambda -\sqrt{\tz_0} \right\vert   } e^{i\phi} \sin 2\phi  \right) e^{i\ttheta}.
\end{align}
\begin{equation}
\label{R1.bd app}
|W|=\left| \dbar R_1  \right| 	\lesssim\left( |   {\tdr}'\left( u\right)| +\dfrac{|{\tdr}\left( u\right)- \tdr( \sqrt{\tz_0}  ) | }{  |z-\sqrt{\tz_0}| }  \right) e^{- v t}.
\end{equation}
From the fact that $\tdr\in H^{1,1}_0$ we deduce that
$$\tdr(\sqrt{\tz_0})={o}(\sqrt{\tz_0})$$
so the solution to model problem on the deformed contour will take the form
\begin{equation}
\label{model-app}
I+o(\sqrt{\tz_0})
\end{equation}
Then we proceed as in the previous section and study the integral equation related to the $\dbar$-problem. The estimates will follow from the same arguments in section \ref{sec:outside}. All we need is a uniform estimate of the following integrals near $\tz_0$ as $\tz_0\to 0$:

{\small \begin{align*}
\int_0^{ \infty } \int_{ v +\sqrt{\tz_0} }^{\infty} \dfrac{|{\tdr}(u)- \tdr( \sqrt{\tz_0}  ) | }{  |(u-\sqrt{\tz_0})^2+v^2 |^{1/2} |u^2+v^2|^{1/2} }  e^{- v t } \, du \, dv & \leq \int_0^{ \infty } \int_{ v +\sqrt{\tz_0} }^{\infty} \dfrac{|{\tdr}(u)| }{  |(u-\sqrt{\tz_0})^2+v^2 |^{1/2} |u^2+v^2|^{1/2} }  e^{- v t } \, du \, dv \\
&+\int_0^{ \infty } \int_{ v  +\sqrt{\tz_0}  }^{\infty} \dfrac{|r( \sqrt{\tz_0}  ) | }{  |(u-\sqrt{\tz_0})^2+v^2 |^{1/2} |u^2+v^2|^{1/2} }  e^{- v t } \, du \, dv\\
 &=:\tilde{I}_1+\tilde{I}_2.
\end{align*}}
Using the same argument as \eqref{est-r-out}
\begin{align*}
\tilde{I}_1 \lesssim t^{-1}.
\end{align*}
For $\tilde{I}_2$, again using \textit{H\"older}'s inequality, for $1<p<2$ and $1/p+1/q=1$
\begin{align*}
\tilde{I}_2 \leq \int_0^{\infty}  e^{-vt}  \left( \int_{ \sqrt{\tz_0} }^\infty \dfrac{| \tdr(\sqrt{\tz_0})  |^q}{|u|^{q}}du  \right)^{1/q} \left( \int_{ v}^{\infty} \dfrac{1}{ |v|^{p} |1+(u^2/v^2)|^{p/2}  } du  \right)^{1/p} dv.
\end{align*}
Using the fact that $\lim_{\lambda\to 0} \tdr(\lambda)/\lambda $ is bounded near the origin, a simple calculation gives 
$$ \left( \int_{ \sqrt{\tz_0} }^\infty \dfrac{| \tdr(\sqrt{\tz_0})  |^q}{|u|^{q}}du  \right)^{1/q} <\infty.$$
It follows as $\tz_0 \to 0$ that
\begin{align*}
\tilde{I}_2 \lesssim t^{-1/p}.
\end{align*}
Similarly:
{ \begin{align*}
\int_0^{ \infty } \int_{ v +\sqrt{\tz_0} }^{\infty} \dfrac{\left|{\tdr}'\left( u\right)\right| }{ |u^2+v^2|^{1/2} }  e^{- v t } \, du \, dv & \leq \int_0^{ 1 } \int_{ v }^{1} \dfrac{\left|{\tdr}'\left( u\right) \right| }{ |u^2+v^2|^{1/2} }  e^{- v t } \, du \, dv \\
&\quad+\int_0^{ \infty } \int_{ 1 }^{\infty} \dfrac{\left|{\tdr}'\left( u\right) \right| }{  |u^2+v^2|^{1/2}  } e^{- v t } \, du \, dv\\
 &\lesssim t^{-1/p}+t^{-1}.
\end{align*}}
Similarly to Section \ref{sec:outside} we obtain the following asymptotic formulas as $x/t\to 1$ for $0<x/t<1$:
\begin{align}
\label{app-cos}
u(x,t)  &= o(\sqrt{z_0}) +\mathcal{O}\left(  t^{-1/p}\right)\\
\label{app-sin}
v(x,t)    &=o(\sqrt{z_0}) + \mathcal{O}\left( t^{-1/p }\right).
\end{align}
\subsection{$x/t \to -1$} Finally we give a brief discussion on the asymptotic formula of \eqref{eq: MTM} in the region where $x/t \to -1$ as $t\to \infty$.  We instead work with the spectral problem \eqref{eq: n-s} and establish a parallel version of Proposition \ref{prop:r} and  RHP Problem \ref{RHP2-v}. We then repeat the nonlinear steepest descent method on this new RHP. To further facilitate computation, we make the following change of variable:
$$\lambda \mapsto \dfrac{1}{\lambda}.$$
Thus we have 
\begin{align*}
\breve{r}(\lambda) &=\tdr\left( \dfrac{1}{\lambda} \right)\\
\breve{\theta}(\lambda; x, t)&= \dfrac{1}{2}\left(  \left( \dfrac{1}{\lambda}-\lambda  \right)\dfrac{x}{t} +\left( \lambda+\dfrac{1}{\lambda}  \right) \right)t\\
\breve{z_0} &= \sqrt{ \dfrac{t+x}{t-x}}.
\end{align*}
It is easy to check that 
\begin{itemize}
\item[I.] $x/t>-1$ implies that 
for $v>0$, 
\begin{equation}
\label{theta-out-}
4\text{Re}i\breve{\theta}(\lambda; x, t)=-\left(   1 -\dfrac{x}{t} \right)v t +  \left( 1+\dfrac{x}{t} \right)\dfrac{v t}{u^2+v^2};
\end{equation}
\item[II.] $x/t <-1$ implies that 
for $v>0$, 
\begin{equation}
\label{theta-out-}
4\text{Re}i\breve{\theta}(\lambda; x, t)=-\left(   1 -\dfrac{x}{t} \right)v t +  \left( 1+\dfrac{x}{t} \right)\dfrac{v t}{u^2+v^2}<-\left(   1 -\dfrac{x}{t} \right)v t  .
\end{equation}
\end{itemize}
For case I above we will follow the same argument in subsection \ref{subsec: app} to obtain 
\begin{align}
\label{app-cos-}
u(x,t) &= o( \breve{z_0}) +\mathcal{O}\left(  t^{-1/p}\right)\\
\label{app-sin-}
v(x,t)    &=o(\sqrt{ \breve{ z_0} }) + \mathcal{O}\left( t^{-1/p }\right).
\end{align}
Case II is similar to the situation in subsection \ref{subsec: out} so the formulas are 
\begin{align}
\label{app-cos-'}
u(x,t)  &= \mathcal{O}\left(  t^{-1}\right)\\
\label{app-sin-'}
v(x,t)    &= \mathcal{O}\left( t^{-1}\right).
\end{align}
{ All the implicit constants above only depend on the Sobolev norm of the reflection coefficient  $\tdr(\lambda)\in H^{1,1}_0(\bbR)$.}

\end{document}